\documentclass[psamsfonts]{amsart}

\calclayout

\usepackage{comment}
\usepackage{mathrsfs}
\usepackage{amsfonts}
\usepackage{amsmath}
\usepackage{amsthm}
\usepackage{graphicx}
\usepackage{mathrsfs}
\usepackage[dvipsnames]{xcolor}
\usepackage{tikz}
\usepackage{tikz-cd}
\usepackage{amscd}
\usepackage{cancel}
\usepackage{amssymb}
\usepackage{rotating}
\usepackage[colorlinks,citecolor = blue, linkcolor=blue]{hyperref}
\usepackage{bm}
\usepackage[shortlabels]{enumitem}
\usepackage{subcaption}

\newcommand{\Z}[1]{\ensuremath{\mathbb{Z}/#1\mathbb{Z}}}

\newcommand{\bx}{\boldsymbol{x}}
\newcommand{\by}{\boldsymbol{y}}
\newcommand{\bz}{\boldsymbol{z}}

\newcommand{\ba}{\boldsymbol{a}}
\newcommand{\bb}{\boldsymbol{b}}
\newcommand{\btheta}{\boldsymbol{\Theta}}
\newcommand{\balpha}{\boldsymbol{\alpha}}
\newcommand{\bbeta}{\boldsymbol{\beta}}
\newcommand{\bgamma}{\boldsymbol{\gamma}}
\newcommand{\bdelta}{\boldsymbol{\delta}}

\theoremstyle{plain}
\newtheorem{thm}{Theorem}[section]
\newtheorem{cor}[thm]{Corollary}
\newtheorem{prop}[thm]{Proposition}
\newtheorem{claim}[thm]{Lemma}

\theoremstyle{definition}
\newtheorem{defn}[thm]{Definition}

\theoremstyle{remark}
\newtheorem{rem}[thm]{Remark}

\newtheorem{remark}[thm]{Remark}

\title{Projective Naturality in Heegaard Floer Homology}

\author{Mike Gartner}
\thanks{MG was partly supported by NSF Grant DMS-1810893.}

\begin{document}
	
\maketitle

\begin{abstract}    
Let $\text{Man}_{*}$ denote the category of closed, connected, oriented and based $3$-manifolds, with basepoint preserving diffeomorphisms between them. Juh\'asz, Thurston and Zemke showed that the Heegaard Floer invariants are natural with respect to diffeomorphisms, in the sense that there are functors $$HF^{\circ}: \text{Man}_{*} \rightarrow \mathbb{F}_{2}[U]\text{-}\text{Mod}$$ whose values agree with the invariants defined by Ozsv\'ath and Szab\'o. The invariant associated to a based $3$-manifold comes from a transitive system in $\mathbb{F}_{2}[U]\text{-}\text{Mod}$ associated to a graph of embedded Heegaard diagrams representing the $3$-manifold. We show that the Heegaard Floer invariants yield functors $$HF^{\circ}: \text{Man}_{*} \rightarrow \text{Trans}(P(\mathbb{Z}[U]\text{-}\text{Mod}))$$ to the category of transitive systems in a projectivized category of $\mathbb{Z}[U]$-modules. 
In doing so, we will see that the transitive system of modules associated to a $3$-manifold actually comes from an underlying transitive system in the projectivized homotopy category of chain complexes over $\mathbb{Z}[U]\text{-}\text{Mod}$. We discuss an application to involutive Heegaard Floer homology, and potential generalizations of our results.
\end{abstract}

{
	\hypersetup{linkcolor=black}
	\tableofcontents
}

\section{Introduction}
\label{Introduction}
The Heegaard Floer invariants associated to closed, oriented 3-manifolds were defined in the work of Ozsv\'ath and Szab\'o \cite{Disks1}. There it was shown that to each such 3-manifold, one can associate an isomorphism class of $\mathbb{Z}[U]$-module. Furthermore, cobordisms between 3-manifolds were shown to induce maps between the invariants \cite{FourManifoldInvariants}. However, there was a gap in the proof of the naturality of these maps. Showing that these invariants are natural with respect even to diffeomorphisms is subtle, and involves detailed consideration of the dependence of the invariants on the choices of Heegaard data, basepoints and embeddings of Heegaard diagrams involved in their construction.

These subtleties were studied extensively by Juh\'asz, Thurston and Zemke in \cite{Naturality}. There they explicated a particular type of loop of Heegaard moves, simple handleswaps, which previous work did not preclude from potentially yielding monodromy in the Heegaard Floer invariants. Moves analogous to these simple handleswap moves were previously studied in detail and suggested as possible candidates for loops with monodromy in the work of Sarkar (eg in \cite{BasepointMoving}). Through a careful analysis of a space of embedded Heegaard diagrams, Juh\'asz, Thurston and Zemke exhausted all possible monodromies and obstructions to the Heegaard Floer assignments being natural with respect to diffeomorphisms, and were then able to provide a minimal set of requirements which could be checked to verify such naturality. They then checked that these requirements are satisfied for all variants of Heegaard Floer homology with coefficients in $\mathbb{F}_{2}$. By building on the work in \cite{FourManifoldInvariants} and \cite{Naturality}, Zemke described in \cite{GraphTQFT} the dependence of the cobordism maps defined in \cite{FourManifoldInvariants} on basepoints. Using this dependence, Zemke completed the verification of the fact that the cobordism maps are in fact natural (over $\mathbb{F}_{2}$) with respect to composition of cobordisms (when the cobordisms are appropriately decorated with graphs).

In this paper we explain the necessary modifications that must be made to obtain naturality with respect to diffeomorphisms of all variants of Heegaard Floer homology, but with coefficients in $\mathbb{Z}$. The most immediate goal of our work is simply to fill a gap in the literature. We hope this will be useful both as a resource for non-experts who aim to understand Heegaard Floer homology itself, and as groundwork which can be used to better understand other invariants associated with Heegaard Floer homology. For example, the contact invariants defined in \cite{Contact} have proven to be extremely effective in detecting subtle contact properties, and both their definition and many of their applications require the ability to nail down particular elements in the modules $HF^{\circ}$, and the ability to effectively compare two such elements in the same module. We also note that the results in \cite{Naturality} and the analogous integral results presented here are necessary steps for establishing naturality of the integral Heegaard Floer invariants with respect to cobordisms.

\subsection{Statement of Main Results}	
In order to study naturality of many flavors of Heegaard Floer homology and Knot Floer homology simultaneously, Juh\'asz, Thurston and Zemke work with sutured $3$-manifolds. They consider a graph $\mathcal{G}$ which encodes the combinatorial structure of a space of sutured Heegaard diagrams related by certain Heegaard moves. Roughly, the vertices of $\mathcal{G}$ correspond to isotopy diagrams of sutured manifolds, and between any two such isotopy diagrams there are edges which describe whether they are related by any of the standard Heegaard moves, or additionally whether they are related by a diffeomorphism. The graph $\mathcal{G}$ contains many sutured isotopy diagrams which are not relevant to the consideration of closed 3-manifolds, so in considering the closed $3$-manifold invariants $HF^{\circ}$ attention is restricted to a subgraph $\mathcal{G}(\mathcal{S}_{\text{man}})$. This is the full subgraph of $\mathcal{G}$ whose vertices consist only of those isotopy diagrams representing sutured manifolds which can be constructed from a closed 3-manifold in a prescribed way. Since we are only concerned with results regarding closed $3$-manifolds in this paper, we will minimize the role of sutured manifolds, and phrase our results in terms of a graph which is isomorphic to $\mathcal{G}(\mathcal{S}_{\text{man}})$ which we denote by $\mathcal{G}_{\text{man}}$. This graph has vertices corresponding to isotopy diagrams of closed, pointed $3$-manifolds, where the isotopies are required to be supported away from the basepoint. Edges in $\mathcal{G}_{\text{man}}$ correspond to sequences of handleslides, stabilizations and diffeomorphisms.

To study naturality using these graphs, we consider the two notions of a \emph{Heegaard invariant} introduced in \cite{Naturality}. The first, a \emph{weak Heegaard invariant} valued in a category $\mathcal{C}$, is simply a morphism of graphs from $\mathcal{G}_{\text{man}}$ to $\mathcal{C}$ under which all edges in the domain get mapped to isomorphisms. In this language, we can summarize one of the invariance results shown in \cite{Disks1} as stating that the morphisms of graphs  
$$HF^{\circ}: G_{\text{man}} \rightarrow\mathcal{C}$$ 
for $\mathcal{C} =  \mathbb{Z}[U]\text{-}\text{Mod}$ or $\mathcal{C} =  \mathbb{F}_{2}[U]\text{-}\text{Mod}$ determined by Heegaard Floer homology are weak Heegaard invariants. The second notion, that of a \emph{strong Heegaard invariant}, serves as a minimal set of conditions which are needed to ensure that a weak Heegaard invariant yields a natural invariant of the underlying $3$-manifolds; precisely, the authors show that the image of a strong Heegaard invariant $HF^{\circ}: G_{\text{man}} \rightarrow \mathcal{C}$, when appropriately restricted, forms a transitive system in $\mathcal{C}$. This step occupies a majority of the work in the paper, and none of the results in this step depend on the target category $\mathcal{C}$. The authors then prove that, in the case when $\mathcal{C} = \mathbb{F}_{2}[U]\text{-}\text{Mod}$, such a transitive system yields a functor 
$$HF^{\circ}: \text{Man}_{*} \rightarrow \mathbb{F}_{2}[U]\text{-}\text{Mod}.$$ 
Finally, they establish that $HF^{\circ}: G_{\text{man}} \rightarrow \mathbb{F}_{2}[U]\text{-}\text{Mod}$ is in fact a strong Heegaard invariant, completing their proof that the invariants $HF^{\circ}$ yield functors from $\text{Man}_{*}$ to $\mathbb{F}_{2}[U]\text{-}\text{Mod}$.

Our main goal here is to establish similar results for $\mathcal{C} = P(\mathbb{Z}[U]\text{-}\text{Mod})$, the quotient category obtained from $\mathbb{Z}[U]\text{-}\text{Mod}$ by the relation $f \sim -f$ for all $f \in \text{Hom}_{\mathbb{Z}[U]\text{-}\text{Mod}}$. Said simply, we want to show that naturality holds over $\mathbb{Z}$, up to a sign. We will consider a category $\text{Trans}(P(\mathbb{Z}[U] \text{-} \text{Mod}))$ of transitive systems in $P(\mathbb{Z}[U]\text{-}\text{Mod})$, and our main result will be:

\begin{thm}\label{Functoriality}
	There are functors 
	$$\widehat{HF}, HF^{-}, HF^{+}, HF^{\infty}: \text{Man}_{*} \rightarrow \text{Trans}(P(\mathbb{Z}[U] \text{-} \text{Mod}))$$
	whose values on a based 3-manifold $(Y,z)$ are isomorphic to the modules defined in \cite{Disks1}. Furthermore, isotopic diffeomorphisms have the same image under $HF^{\circ}$.
\end{thm}

\begin{remark}
	The finite rank variant $HF_{\text{red}}$ of Heegaard Floer homology defined in \cite[Definition 4.7]{Disks1} arises as a suitable quotient (or submodule) of $HF^{\pm}$, and Theorem \ref{Functoriality} implies that this variant also yields a functor $HF_{\text{red}}: \text{Man}_{*} \rightarrow \text{Trans}(P(\mathbb{Z}[U] \text{-} \text{Mod}))$.
\end{remark}

We will import wholesale the logical structure of \cite{Naturality} used to prove the analog of Theorem \ref{Functoriality} appearing there. It will therefore suffice to show that $HF^{\circ}:  \mathcal{G}_{\text{man}} \rightarrow P(\mathbb{Z}[U]\text{-}\text{Mod})$ is a strong Heegaard invariant. We will in fact show something slightly stronger. Let $\text{Kom}(\mathbb{Z}[U]\text{-}\text{Mod})$ denote the homotopy category of chain complexes over $\mathbb{Z}[U]\text{-}\text{Mod}$, and, as described above, let $P(\text{Kom}(\mathbb{Z}[U]\text{-}\text{Mod}))$ denote the projectivization of this category. Finally, let  $\text{Trans}(P(\text{Kom}(\mathbb{Z}[U]\text{-}\text{Mod})))$ denote the category of transitive systems in $P(\text{Kom}(\mathbb{Z}[U]\text{-}\text{Mod}))$. We will unpack the precise meaning of these categories in Section \ref{TransitiveSystemSection}. A majority of the paper will be occupied with showing:

\begin{thm}\label{StrongHeegaard2}
	The morphisms
	$$\widehat{CF}, CF^{-}, CF^{+}, CF^{\infty}: \mathcal{G}_{\text{man}} \rightarrow \text{Trans}(P(\text{Kom}(\mathbb{Z}[U]\text{-}\text{Mod})))$$
	are strong Heegaard invariants.
\end{thm}

While proving Theorem \ref{StrongHeegaard2} we will show the analogous result holds on the level of homology:

\begin{cor}\label{StrongHeegaard}
	The morphisms
	$$\widehat{HF}, HF^{-}, HF^{+}, HF^{\infty}: \mathcal{G}_{\text{man}} \rightarrow P(\mathbb{Z}[U] \text{-} \text{Mod})$$
	are strong Heegaard invariants.
\end{cor}

We will establish Theorem \ref{StrongHeegaard2} in Sections \ref{MainProof} and \ref{Moduli}. We will also obtain from Theorem \ref{StrongHeegaard2} the following statement about the constituent chain complexes.

\begin{cor}\label{TransitiveSystemResult}
	Given a closed, connected, oriented and based 3-manifold $(Y,z)$ and a $\text{Spin}^{c}$-structure $\mathfrak{s}$ over $Y$, the $\mathbb{Z}[U]$-module chain complexes $CF^{\circ}(\mathcal{H}, \mathfrak{s})$, ranging over all strongly $\mathfrak{s}$-admissible embedded Heegaard diagrams $\mathcal{H}$ for $(Y,z)$, fit into a transitive system of homotopy equivalences in $P(\text{Kom}(\mathbb{Z}[U]\text{-}\text{Mod}))$ with respect to the maps induced by sequences of pointed handleslides, stabilizations, isotopies, and diffeomorphisms of Heegaard surfaces which are isotopic to the identity in $Y$.
	
\end{cor}

\begin{remark}
	The Heegaard Floer invariants arise as direct sums of invariants 
	$$HF^{\circ}(Y,z) = \bigoplus_{\mathfrak{s} \in \text{Spin}^{c}(Y)} HF^{\circ}(Y,z,\mathfrak{s})$$ associated to triples $(Y,z, \mathfrak{s})$ for $\mathfrak{s} \in \text{Spin}^{c}(Y)$. All of the main results have refined statements regarding these invariants of  $(Y,z, \mathfrak{s})$. Theorem \ref{StrongHeegaard2}, Corollary \ref{StrongHeegaard} and Corollary \ref{TransitiveSystemResult} also depend on choices of coherent orientation systems, which we omit from the statements here. For now, we note that all of the results above hold in particular for the Heegaard Floer chain complexes defined with respect to the canonical coherent orientation systems constructed in \cite{Disks2}. The precise conditions required of the coherent orientation systems implicitly appearing in the results above will be specified in Definition \ref{WeakHeegaardFloer1}.

\end{remark}

\subsection{Further Directions and Applications}
We now point out some applications and potential generalizations of our results. Given two based $3$-manifolds $(Y_{1},z_{1})$ and $(Y_{2},z_{2})$, a cobordism $W$ between them decorated with a choice of path in $W$ from $z_{1}$ to $z_{2}$, and a choice of $\mathfrak{t} \in \text{Spin}^{c}(W)$, Ozsv\'ath and Szab\'o constructed in \cite{FourManifoldInvariants} cobordism maps:
$$ F_{W, \mathfrak{t}}^{\circ}: HF^{\circ}(Y_{1}, z_{1}, \mathfrak{t}|_{Y_{1}}) \rightarrow HF^{\circ}(Y_{2}, z_{2}, \mathfrak{t}|_{Y_{2}}).$$
(The choice of path is not made explicit in \cite{FourManifoldInvariants}). In \cite{GraphTQFT} Zemke extended the results in \cite{Naturality} to show that over $\mathbb{F}_{2}$ these maps are well-defined and natural with respect to composition of decorated cobordisms. We expect that our results can be used in a similar way to establish such naturality over $\mathbb{Z}$, up to an overall sign. Furthermore, in \cite{FourManifoldInvariants} Ozsv\'ath and Szab\'o showed how naturality of the Heegaard Floer invariants with respect to decorated cobordisms can be used to define the so called mixed invariants of closed $4$-manifolds. Given a closed $4$-manifold $X$ and a choice of $\mathfrak{t} \in \text{Spin}^{c}(X)$, these take the form of maps
$$ \Phi_{X, \mathfrak{t}}: \Lambda^{*}(H_{1}(X; \mathbb{F}_{2}) / \text{Tors}) \otimes_{\mathbb{F}_{2}} \mathbb{F}_{2}[U] \rightarrow \mathbb{F}_{2}.$$
These share many of the features of the Seiberg-Witten invariants, and serve as powerful tools in detecting subtle smooth information. If one can establish naturality with respect to cobordisms over $\mathbb{Z} / \pm$, we would obtain corresponding mixed invariants
$$ \Phi_{X, \mathfrak{t}}: \Lambda^{*}(H_{1}(X; \mathbb{Z}) / \text{Tors}) \otimes_{\mathbb{Z}} \mathbb{Z}[U] \rightarrow \mathbb{Z} / \pm $$
which we expect would provide fruitful extra information. In fact, before the gap in the literature was noticed, the integral mixed invariants had already been extensively studied in papers including \cite{CobApps3}, \cite{CobApps1} and \cite{CobApps2}, so establishing naturality with respect to cobordisms over $\mathbb{Z}$ would immediately prove useful, and would likely also be useful for computations and applications in the future.

A second application of our work comes from involutive Heegaard Floer homology, defined by Hendricks and Manolescu in \cite{HM}. To describe it, fix a closed $3$-manifold $Y$ and  $\mathfrak{s} \in \text{Spin}^{c}(Y)$. Given a pointed Heegaard diagram $\mathcal{H} = (\Sigma, \balpha, \bbeta, z)$ for $(Y,z)$, there is a conjugate diagram $\overline{\mathcal{H}} = (-\Sigma, \bbeta, \balpha, z)$ for $(Y,z)$ given by reversing the orientation on the surface and switching the role of the $\alpha$ and $\beta$ curves. Under suitable admissibility hypotheses, there is a chain isomorphism 
$$\eta_{\mathcal{H} \rightarrow \overline{\mathcal{H}}}: CF^{\circ}(
\mathcal{H}, \mathfrak{s}) \rightarrow CF^{\circ}(
\overline{\mathcal{H}}, \overline{\mathfrak{s}}) $$
given by mapping intersection points to themselves \cite[Theorem 2.4]{Disks2}. Note that the role of coherent orientations here is not yet relevant, as Hendricks and Manolescu work over $\mathbb{F}_{2}$. Using the results in \cite{Naturality}, Hendricks and Manolescu showed that the $\mathbb{F}_{2}$ analog of Corollary \ref{TransitiveSystemResult} holds: the modules $CF^{\circ}(\mathcal{H}, \mathfrak{s})$ fit into a transitive system in the homotopy category of chain complexes of $\mathbb{F}_{2}[U]$-modules with respect to the maps induced by the Heegaard moves appearing in Corollary \ref{TransitiveSystemResult}. Thus, since $\mathcal{H}$ and $\overline{\mathcal{H}}$ represent the same $3$-manifold, there is a chain homotopy equivalence
$$\Phi(\overline{\mathcal{H}}, \mathcal{H}):CF^{\circ}(
\overline{\mathcal{H}}, \overline{\mathfrak{s}}) \rightarrow CF^{\circ}(
\mathcal{H}, \overline{\mathfrak{s}}) $$
of complexes of $\mathbb{F}_{2}[U]$-modules which is well defined up to homotopy. Using these maps, they consider the map $\iota : = \Phi(\overline{\mathcal{H}}, \mathcal{H}) \circ \eta_{\mathcal{H} \rightarrow \overline{\mathcal{H}}}$, which is well defined up to homotopy, and which is shown to be a homotopy involution in \cite[Lemma 2.5]{HM}. They then use it to construct an invariant of $Y$ as follows.

There is a $\Z2$ action on $\text{Spin}^{c}(Y)$ given by conjugation. Let $[\text{Spin}^{c}(Y)]$ denote the set of orbits in $\text{Spin}^{c}(Y)$ under this action. Given an orbit $\overline{\omega} \in [\text{Spin}^{c}(Y)]$, let
$$CF^{\circ}(\mathcal{H}, \overline{\omega}) = \bigoplus_{\mathfrak{s} \in \overline{\omega}} CF^{\circ}(\mathcal{H}, \mathfrak{s}).$$
The authors investigate the map $(1 + \iota)$, considered as a chain map between complexes of $\mathbb{F}_{2}[U]$-modules, and consider its cone $$CFI(\mathcal{H}, \overline{\omega}) := \text{Cone}(1 + \iota) = \left( CF^{\circ}(\mathcal{H}, \overline{\omega})[-1] \oplus CF^{\circ}(\mathcal{H}, \overline{\omega}), \partial_{\text{cone}} = \begin{pmatrix} \partial & 0 \\ 1 + \iota & -\partial \end{pmatrix}     \right).$$
Here $CF^{\circ}(\mathcal{H}, \overline{\omega})[-1]$ indicates the shifted chain complex, whose degree $n$ piece is given by $\left(CF^{\circ}(\mathcal{H}, \overline{\omega})[-1] \right)_{n}= CF^{\circ}(\mathcal{H}, \overline{\omega})_{n-1}$. They then introduce a formal variable $Q$ of degree $-1$ satisfying $Q^{2} =0$, and rewrite the map being coned over as
$$ CF^{\circ}(\mathcal{H}, \overline{\omega}) \xrightarrow{Q \cdot(1 + \iota)} Q \cdot CF^{\circ}(\mathcal{H}, \overline{\omega})[-1].$$
As one can readily check, the cone and its differential can then be rewritten as
\begin{equation}
\label{Q-action}
\text{Cone}(1 + \iota) = \left( CF^{\circ}(\mathcal{H}, \overline{\omega})[-1] \otimes \mathbb{F}_{2}[Q]/(Q^{2}), \partial + Q(1 + \iota) \right).
\end{equation}
Considered in this way, it is a complex of modules over the ring $\mathcal{R} = \mathbb{F}_{2}[Q,U]/(Q^{2})$. The authors then show that the quasi-isomorphism class of the complex $CFI(\mathcal{H}, \overline{\omega})$ of $\mathcal{R}$-modules thus defined is an invariant of $(Y, \overline{\omega})$. 

We now explain how Corollary \ref{TransitiveSystemResult} can be used to construct a version of such an invariant defined over $\mathbb{Z}$. Before doing so, we make a remark on the reliance of the following discussion on orientation systems.
\begin{remark}
	First we note that the proof establishing that $\eta$ is an isomorphism given in \cite[Theorem 2.4]{Disks2} implicitly proves the statement with respect to an arbitrary coherent orientation system $\mathfrak{o}$ over the domain $\mathcal{H}$ and, ostensibly, the \emph{same} coherent orientation system over the codomain $\overline{\mathcal{H}}$ (the use of the word \emph{same} makes sense because the underlying diagrams for the domain and codomain of $\eta$ are the same aside from labeling and orientations). However, to avoid this consideration we will simply focus attention here on the case where both diagrams are equipped with canonical orientation systems, as defined in \cite{Disks2}. We note that the maps $\Phi$ take canonical orientation systems to canonical orientation systems, since more generally any sequence of maps induced by Heegaard moves takes a canonical orientation system to a canonical orientation system. This follows from the facts that Heegaard moves induce module isomorphisms on the totally twisted module $\underline{HF}^{\infty}$ (see \cite[Section 8]{Disks2}), and that the canonical orientation system on a diagram can be characterized by the isomorphism type of $\underline{HF}^{\infty}$ (see \cite[Theorem 10.12]{Disks2}). Similarly, the isomorphism $\eta$ can be defined with respect to canonical orientation systems on both diagrams. Indeed, since the proof of \cite[Theorem 2.4]{Disks2} also shows that the map $\eta$ yields an isomorphism between the totally twisted module $\underline{HF}^{\infty}$ associated to a diagram equipped with the canonical orientation, and the totally twisted module associated to the reversed diagram equipped with the induced orientation system, the induced orientation system in this case must be the canonical one. With these remarks in mind, we will omit all reference to coherent orientation systems from our notation and description; all remarks in the remainder of the description of this application apply only to the canonical orientation systems. 
\end{remark}

Fix again a $3$-manifold $Y$, and diagrams $\mathcal{H}$ and $\overline{\mathcal{H}}$ representing $Y$ as above. Since $\mathcal{H}$ and $\overline{\mathcal{H}}$ represent the same $3$ manifold, we obtain from Corollary \ref{TransitiveSystemResult} (at most) two homotopy classes of chain homotopy equivalences
$$\pm \Psi(\overline{\mathcal{H}}, \mathcal{H}): CF^{\circ}(
\overline{\mathcal{H}}, \overline{\mathfrak{s}}) \rightarrow CF^{\circ}(
\mathcal{H}, \overline{\mathfrak{s}})$$
associated to sequences of Heegaard moves relating the two diagrams. The set $\{\pm \Psi(\overline{\mathcal{H}}, \mathcal{H})\}$ is well defined up to chain homotopy. We thus obtain two homotopy classes of maps $\pm \iota := \pm \Psi(\overline{\mathcal{H}}, \mathcal{H}) \circ \eta_{\mathcal{H} \rightarrow \overline{\mathcal{H}}}$. The same argument used in \cite[Lemma 2.5]{HM} to show that $\iota$ is a homotopy involution over $\mathbb{F}_{2}$ now shows that $\pm \iota$ both have order at most $4$ (up to homotopy) over $\mathbb{Z}$. We define
$$CFI_{\pm}(\mathcal{H}, \overline{\omega}) := \text{Cone}(1 \pm \iota),$$
where now both complexes are considered as complexes of $\mathbb{Z}[U]$-modules. While we can no longer conclude the maps $\pm \iota$ are homotopy involutions, we still obtain that the collection of the two quasi-isomorphism classes of the complexes of $\mathbb{Z}[U]$-modules that we obtain is an invariant of the underlying $3$-manifold.

\begin{thm} \label{InvolutiveInvariants}
	With respect to the canonical orientation systems of \cite{Disks2}, the unordered pair of quasi-ismorphism classes determined by the complexes $$CFI_{\pm}(\mathcal{H}, \overline{\omega})$$
	(considered as complexes of $\mathbb{Z}[U]$-modules) is an invariant of $(Y,\overline{\omega}, z)$.
\end{thm}
\begin{proof}
	The proof is essentially the same as that in \cite{HM}, but we include a sketch of it here for the reader's convenience. 
	
	Fix $(Y,z,\overline{\omega})$, and consider a diagram $\mathcal{H}$ and its conjugate $\overline{\mathcal{H}}$ as above. As we noted earlier, for the fixed diagram $\mathcal{H}$ the collection of the two chain homotopy equivalences $\{\pm \Psi(\overline{\mathcal{H}}, \mathcal{H})\}$ is well defined up to chain homotopy by Corollary \ref{TransitiveSystemResult}. Thus so too is the collection $\{\pm \iota\}$. We conclude that the set of the two cones $\{CFI_{\pm}(\mathcal{H}, \overline{\omega})\}$ associated to $(\mathcal{H}, \overline{\omega})$ is well defined up to chain homotopy equivalence.
	
	Next, we consider the dependence on the choice of diagram. Consider a different diagram $\mathcal{H}'$ for $(Y,z)$ and its conjugate $\overline{\mathcal{H}'}$. We obtain corresponding collections $\{\pm \Psi(\overline{\mathcal{H}'}, \mathcal{H}')\}$ and $\{\pm \iota'\}$ which are both well defined up to homotopy, and $\{CFI_{\pm}(\mathcal{H}', \overline{\omega})\}$ well defined up to homotopy equivalence. Choose some fixed sequence of Heegaard moves connecting $\mathcal{H}$ to $\mathcal{H'}$, and consider either of the (at most two) corresponding chain homotopy equivalences $\pm \Psi(\mathcal{H}, \mathcal{H}')$ furnished by Corollary \ref{TransitiveSystemResult}. We denote our choice by $\Psi(\mathcal{H}, \mathcal{H}')$. Consider the following diagram involving the four cone complexes in question
	\begin{equation}\label{InvolutiveCommutes}
	\begin{tikzcd}
	CF^{\circ}(\mathcal{H}, \overline{\omega})[-1] \arrow{rr}{1 \pm \iota} \arrow{dd}{\Psi(\mathcal{H}, \mathcal{H}')} & & CF^{\circ}(\mathcal{H}, \overline{\omega}) \arrow{dd}{\Psi(\mathcal{H}, \mathcal{H}')}   \\
	& &   \\
	CF^{\circ}(\mathcal{H}', \overline{\omega})[-1] \arrow{rr}{1 \pm \iota'}  & & CF^{\circ}(\mathcal{H}', \overline{\omega})  \\
	\end{tikzcd} 
	\end{equation}
	
	We claim that for a fixed choice in $\{\pm \iota \}$, the diagram commutes up to homotopy for at least one of the two choices in $\{\pm \iota ' \}$. We denote our choice of the fixed homotopy class in the top row by $\iota$. To establish the claim, we need to show that
	$$\Psi(\mathcal{H}, \mathcal{H}') \circ \Psi(\overline{\mathcal{H}}, \mathcal{H}) \circ \eta_{\mathcal{H} \rightarrow \overline{\mathcal{H}}} \sim \pm \Psi(\overline{\mathcal{H}'}, \mathcal{H}') \circ \eta_{\mathcal{H}' \rightarrow \overline{\mathcal{H}'}} \circ \Psi(\mathcal{H}, \mathcal{H}').$$
	We note that 
	$$ \eta_{\mathcal{H}' \rightarrow \overline{\mathcal{H}'}} \circ \Psi(\mathcal{H}, \mathcal{H}') \circ \eta_{\overline{\mathcal{H}} \rightarrow \mathcal{H}} \sim \pm \Psi(\overline{\mathcal{H}}, \overline{\mathcal{H}'}).$$
	To see this, observe that $\Psi(\mathcal{H}, \mathcal{H}')$ is a map induced by some sequence of Heegard moves. The map resulting from precomposing and postcomposing this map with the isomorphisms $\eta$ can be realized as the map induced on $CF^{\circ}(\overline{\mathcal{H}})$ by the same set of Heegaard moves giving rise to $\Psi(\mathcal{H}, \mathcal{H}')$ (recall the maps $\eta$ have no effect on the attaching curves). Thus the conjugated map is homotopic to $\pm \Psi(\overline{\mathcal{H}}, \overline{\mathcal{H}'})$ by Corollary \ref{TransitiveSystemResult}. We thus conclude that 
	\begin{align*}
		\Psi(\overline{\mathcal{H}'}, \mathcal{H}') \circ \eta_{\mathcal{H}' \rightarrow \overline{\mathcal{H}'}} \circ \Psi(\mathcal{H}, \mathcal{H}') &\sim \pm  \Psi(\overline{\mathcal{H}'}, \mathcal{H}') \circ \Psi(\overline{\mathcal{H}}, \overline{\mathcal{H}'}) \circ \eta_{\mathcal{H} \rightarrow \overline{\mathcal{H}}} \\
		&\sim \pm \Psi(\mathcal{H}, \mathcal{H}') \circ \Psi(\overline{\mathcal{H}}, \mathcal{H}) \circ \eta_{\mathcal{H} \rightarrow \overline{\mathcal{H}}}		
	\end{align*}
	where the last two maps being homotopic up to a sign is again guaranteed by Corollary \ref{TransitiveSystemResult}.
	Having established that the diagram with $\iota$ in the top row commutes up to chain homotopy for at least one choice of $\{\pm \iota '\}$ in the bottom row, the argument in \cite{HM} now applies directly to establish that $\text{Cone}(1 + \iota)$ is quasi-isomorphic to at least one of the cones $\text{Cone}(1 \pm \iota')$. This concludes the proof.
\end{proof}

In the case of rational homology three spheres, the pair of quasi-isomorphism classes in Theorem \ref{InvolutiveInvariants} can actually be distinguished from one another to furnish two distinct invariants.

\begin{cor}\label{InvolutiveInvariantsModel}
	Let Y be a rational homology three sphere. One can specify the maps $\iota$ such that, with respect to the canonical orientation systems of \cite{Disks2}, the quasi-isomorphism classes determined by:
	$$CFI_{+}(\mathcal{H}, \overline{\omega})$$
	and 
	$$CFI_{-}(\mathcal{H}, \overline{\omega})$$
	(considered as complexes of $\mathbb{Z}[U]$-modules) are each invariants of $(Y, \overline{\omega}, z)$.
\end{cor}

\begin{proof}
	Since $Y$ is a rational homology three sphere, for each $\mathfrak{s} \in \text{Spin}^{c}(Y)$ we have 
	$$HF^{\infty}(Y, \mathfrak{s}) \cong \mathbb{Z}[U, U^{-1}]$$ as $\mathbb{Z}[U]$-modules by \cite[Theorem 10.1]{Disks2}. 
	
	Consider first the case of a $\mathbb{Z}/ 2\mathbb{Z}$-invariant $ \text{spin}^{c}$ structure. For each such $ \text{spin}^{c}$ structure $\mathfrak{s}$ , the maps $\pm \iota$ are homotopy equivalences, so induce graded module isomorphisms on $HF^{\infty}(Y, \mathfrak{s})$. Since $HF^{\infty}(Y, \mathfrak{s}) \cong \mathbb{Z}[U, U^{-1}]$ there are precisely two such morphisms: $\pm \text{Id}$. For each $\mathbb{Z}/ 2\mathbb{Z}$-invariant $\mathfrak{s}$, choose $\iota$ to be the map which induces $-\text{Id}$ on $HF^{\infty}(Y, \mathfrak{s})$. This can be accomplished for all invariant $ \text{spin}^{c}$ structures even with a fixed choice of sign on each map $\Psi(\overline{\mathcal{H}}, \mathcal{H})$, by altering the signs of the maps $\eta$ when necessary. Then the proof of Theorem \ref{InvolutiveInvariants} carries over directly to show the quasi-isomorphism class determined by 
	$$CFI_{+}(\mathcal{H}, \{\mathfrak{s}\})$$ is an invariant of $(Y,\mathfrak{s}, z)$. One must only note that the Diagram \eqref{InvolutiveCommutes} commutes with no sign ambiguity for $\iota$ and $\iota'$ specified by our definition. Indeed, the proof of Theorem \ref{InvolutiveInvariants} shows that the diagram commutes with $1 + \iota$ on top for one of $1 \pm \iota'$ on the bottom, but the diagram could not even commute at the level of homology if $\iota$ induced $-\text{Id}$ and $\iota'$ induced $\text{Id}$. By the same argument, $$CFI_{-}(\mathcal{H}, \{\mathfrak{s}\})$$ also yields an invariant.
	
	Next consider a $\mathbb{Z}/ 2\mathbb{Z}$-orbit $\overline{\omega} = \{\mathfrak{s}, \overline{\mathfrak{s}}\}$ coming from a pair of non-invariant $ \text{spin}^{c}$ structures. We have two 
	homotopy equivalences,
	$ \iota_{\mathfrak{s} \rightarrow \overline{\mathfrak{s}}}: CF^{\circ}(\mathcal{H}, \mathfrak{s}) \rightarrow CF^{\circ}(\mathcal{H}, \overline{\mathfrak{s}})$
	and
	$ \iota_{\overline{\mathfrak{s}} \rightarrow \mathfrak{s}}: CF^{\circ}(\mathcal{H}, \overline{\mathfrak{s}}) \rightarrow CF^{\circ}(\mathcal{H}, \mathfrak{s}).$
	The total map
	$$(1 + \iota): CF^{\circ}(\mathcal{H}, \mathfrak{s}) \oplus CF^{\circ}(\mathcal{H}, \overline{\mathfrak{s}})  \rightarrow CF^{\circ}(\mathcal{H}, \mathfrak{s}) \oplus CF^{\circ}(\mathcal{H}, \overline{\mathfrak{s}})$$
	takes the form 
	$$ (x, y) \mapsto (x + \iota_{\overline{\mathfrak{s}} \rightarrow \mathfrak{s}}(y),  y + \iota_{\mathfrak{s} \rightarrow \overline{\mathfrak{s}}}(x)).$$
	
	Define the signs on these maps such that $ \iota_{\overline{\mathfrak{s}} \rightarrow \mathfrak{s}} \circ \iota_{\mathfrak{s} \rightarrow \overline{\mathfrak{s}}}$ induces $\text{Id}$ on $HF^{\infty}(Y, \mathfrak{s})$ and $\iota_{\mathfrak{s} \rightarrow \overline{\mathfrak{s}}} \circ \iota_{\overline{\mathfrak{s}} \rightarrow \mathfrak{s}}$ induces $\text{Id}$ on $HF^{\infty}(Y, \overline{\mathfrak{s}})$. As above, the choice of signs can be incorporated into the definition of the maps $\eta$. The proof of Theorem \ref{InvolutiveInvariants} again shows this gives a well defined invariant $$CFI_{+}(\mathcal{H}, \overline{\omega}).$$ Similarly, the choice where $\iota_{\mathfrak{s} \rightarrow \overline{\mathfrak{s}}} \circ \iota_{\overline{\mathfrak{s}} \rightarrow \mathfrak{s}}$ and $ \iota_{\overline{\mathfrak{s}} \rightarrow \mathfrak{s}} \circ \iota_{\mathfrak{s} \rightarrow \overline{\mathfrak{s}}}$ both induce $-\text{Id}$ gives a well defined invariant $$CFI_{-}(\mathcal{H}, \overline{\omega}). \qedhere$$
\end{proof}

\begin{remark}
	The two rational homology sphere invariants given in Corollary \ref{InvolutiveInvariantsModel} give rise to distinct involutive Heegaard Floer homologies $HFI^{\infty}(Y, \overline{\omega})$. Namely, one can compute that
	
	$$HFI^{\infty}_{+}(Y, \overline{\omega}) \cong \mathbb{Z}[U, U^{-1}, Q] / (Q^{2})$$
	while 
	
	$$HFI^{\infty}_{-}(Y, \overline{\omega}) \cong \mathbb{Z}/ 2\mathbb{Z}[U, U^{-1}].$$

	To see this, consider the short exact sequence of chain complexes that results from the definition of $CFI^{\infty}(Y, \overline{\omega}) = \text{Cone}(1 + \iota)$:
	
	\begin{equation*}
	\begin{tikzcd}
	0 \arrow{r} & CF^{\infty}(\mathcal{H}, \overline{\omega}) \arrow{r}{i} & CFI^{\infty}(Y, \overline{\omega}) \arrow{r}{p} &  CF^{\infty}(\mathcal{H}, \overline{\omega}) \arrow{r} & 0
	\end{tikzcd} 
	\end{equation*}
	This gives rise a to a long exact sequence in homology
	\begin{equation*}
	\begin{tikzcd}
	 & \cdots \arrow{r}{p_{*}} & HF^{\infty}(Y, \overline{\omega}) \arrow[out=-10, in=170, swap, overlay]{dll}{\delta}\\
	 HF^{\infty}(Y, \overline{\omega}) \arrow{r}{i_{*}} & HFI^{\infty}(Y, \overline{\omega}) \arrow{r}{p_{*}} &  HF^{\infty}(Y, \overline{\omega}) \arrow[out=-10, in=170, swap, overlay]{dll}{\delta} \\
	 HF^{\infty}(Y, \overline{\omega}) \arrow{r}{i_{*}}  & \cdots &
	\end{tikzcd} 
	\end{equation*}
	for which the connecting morphism $\delta$ is precisely the induced map $(1 + \iota)_{*}$. 
	
	Consider first the case of invariant $ \text{spin}^{c}$ structures. When $\iota$ is chosen such that $(1 + \iota)_{*} = 0$, we get a split short exact sequence, so
	
	$$ HFI^{\infty}_{+}(Y, \{\mathfrak{s}\}) \cong HF^{\infty}(Y, \mathfrak{s}) \oplus HF^{\infty}(Y, \mathfrak{s}) \cong \mathbb{Z}[U, U^{-1}]\oplus \mathbb{Z}[U, U^{-1}] .$$
	Tracing through the identification analogous to that in Equation \eqref{Q-action}, this gives
	$$HFI^{\infty}_{+}(Y, \{\mathfrak{s}\}) \cong \mathbb{Z}[U, U^{-1}, Q] / (Q^{2})$$
	as a module over $\mathbb{Z}[Q, U, U^{-1}] / (Q^{2})$.
	When $\iota$ is chosen such that $\delta = (1 + \iota)_{*} = 2$, we instead obtain
	
	$$ HFI^{\infty}_{-}(Y, \{\mathfrak{s}\}) \cong \text{Coker}(\delta )\oplus  \text{Ker}(\delta) \cong HF^{\infty}(Y, \mathfrak{s}) / 2 \cdot HF^{\infty}(Y, \mathfrak{s}) \cong \mathbb{Z} / 2\mathbb{Z}[U, U^{-1}].$$
	Here $HFI^{\infty}_{-}(Y, \mathfrak{s})$ is a $\mathbb{Z}[Q, U, U^{-1}] / (Q^{2})$-module where $Q$ acts by zero.

	In the case of non-invariant $ \text{spin}^{c}$ structures, we can use $\iota_{\mathfrak{s} \rightarrow \overline{\mathfrak{s}}}$ to identify $HF^{\infty}(Y, \mathfrak{s}) \cong HF^{\infty}(Y, \overline{s})$ and consider the map
	
	$$\phi: HF^{\infty}(Y, \mathfrak{s}) \oplus HF^{\infty}(Y, \mathfrak{s}) \rightarrow HF^{\infty}(Y, \mathfrak{s}) \oplus HF^{\infty}(Y, \mathfrak{s})$$
	defined by the composition $\phi = (1\oplus(\iota_{\mathfrak{s} \rightarrow \overline{\mathfrak{s}}})^{-1}) \circ (1 + \iota) \circ (1\oplus\iota_{\mathfrak{s} \rightarrow \overline{\mathfrak{s}}}).$
	More explicitly, we have: 
	$$ \phi(x, y) = (x + \iota_{\overline{\mathfrak{s}} \rightarrow \mathfrak{s}} \circ \iota_{\mathfrak{s} \rightarrow \overline{\mathfrak{s}}}(y), (\iota_{\mathfrak{s} \rightarrow \overline{\mathfrak{s}}})^{-1} \circ  \iota_{\mathfrak{s} \rightarrow \overline{\mathfrak{s}}} (x + y)
	).$$
	For $CFI^{\infty}_{+}$, we defined the constituent maps such that $\iota_{\overline{\mathfrak{s}} \rightarrow \mathfrak{s}} \circ \iota_{\mathfrak{s} \rightarrow \overline{\mathfrak{s}}}$ induces $\text{Id}$, so this becomes:
	$$\phi(x, y) = (x + y, x + y)$$
	and
	$$ HFI^{\infty}_{+}(Y, \{\mathfrak{s}, \overline{\mathfrak{s}}\}) \cong \text{Coker}(\delta )\oplus  \text{Ker}(\delta) \cong \mathbb{Z}[U, U^{-1}] \oplus \mathbb{Z}[U, U^{-1}].$$	
	For $CFI^{\infty}_{-}$, we defined the constituent maps such that $\iota_{\overline{\mathfrak{s}} \rightarrow \mathfrak{s}} \circ \iota_{\mathfrak{s} \rightarrow \overline{\mathfrak{s}}}$ induces $-\text{Id}$, so this becomes:
	$$\phi(x, y) = (x - y, x + y)$$
	and 
	$$ HFI^{\infty}_{-}(Y, \{\mathfrak{s}, \overline{\mathfrak{s}}\}) \cong \text{Coker}(\delta )\oplus  \text{Ker}(\delta) \cong \mathbb{Z} / 2\mathbb{Z}[U, U^{-1}].$$
	The claimed structures as modules over $\mathbb{Z}[Q, U, U^{-1}] / (Q^{2})$ follows as above.
\end{remark}

\begin{remark}
	It is plausible that Corollary \ref{InvolutiveInvariantsModel} actually extends to the general case of closed, connected, oriented 3-manifolds. To specify an individual invariant in this general case would require a method by which one could naturally make a choice for signs on $\iota$. An approach here would be to make an argument like the one in the proof of Corollary \ref{InvolutiveInvariantsModel}, but by taking advantage of the standard form for the totally twisted module $\underline{HF^{\infty}}(Y)$, rather than the standard form for $HF^{\infty}(Y)$ for rational homology spheres. Indeed, by \cite[Theorem 10.12]{Disks2} the totally twisted module associated with any $\text{Spin}^{\text{c}}$ structure is isomorphic to $\mathbb{Z}[U, U^{-1}]$ (as a $\mathbb{Z}[U, U^{-1}]$-module). Using this fact, one could presumably again pick out particular models for $CFI_{+}$ and $CFI_{-}$. What would remain to be shown is that there are analogs to Theorem \ref{Functoriality} and Corollary \ref{TransitiveSystemResult} for the totally twisted complexes, and that these results could be used to carry over the argument used in the proof of Theorem \ref{InvolutiveInvariants}. We expect that the main results in this paper do carry over to the totally twisted complexes, but we leave investigation of this subtlety to the interested reader.
\end{remark}

\subsection{Organization of the Paper}
We begin in Section \ref{background} by recalling the notion of sutured $3$-manifolds and sutured Heegaard diagrams, as all of the results in \cite{Naturality} are phrased in this setting. We discuss a correspondence between sutured and closed $3$-manifolds, and use the correspondence to translate a graph of sutured diagrams central to setting of \cite{Naturality} into an equivalent graph of closed diagrams which we use throughout the remainder of the paper. In Section \ref{HeegaardInvariants} we introduce and rephrase the notions of weak and strong Heegaard invariants defined in \cite{Naturality}. Section \ref{TransitiveSystemSection} deals with setting up the algebraic framework in which our main results are phrased, and in particular includes the definitions of the projectivizations and categories of transitive systems appearing in Theorems \ref{Functoriality} and \ref{StrongHeegaard2}. In Section \ref{ProjectiveNaturalityFromStrong}, we deduce Theorem \ref{Functoriality} and Corollary \ref{TransitiveSystemResult} from Theorem \ref{StrongHeegaard2} and Corollary \ref{StrongHeegaard}. In Sections \ref{WeakHeegaardFloerInvariantsSection} and \ref{MainProof} we recall the constructions involved in defining the integral Heegaard Floer chain complexes, and establish that these constructions yield suitably defined weak Heegaard invariants. In Section \ref{MainProof}, we check that these weak Heegaard invariants satisfy all but one of the axioms required of a strong Heegaard invariant. In Section \ref{Moduli} we carry out the main work and establish that these weak Heegaard invariants also satisfy the last axiom, known as simple handleswap invariance. Finally, in Section \ref{surgerytrianglesection} we explain that the construction of the surgery exact triangle works without modification in our setting.

\subsection{Acknowledgements}
I would like to thank the anonymous referee for providing many invaluable and detailed comments and suggestions, and in particular for their help in pointing out errors in the original manuscript and proposing alternative approaches where necessary. It is also my pleasure to thank Robert Lipshitz for his support and encouragement throughout the course of the writing of this paper, for many helpful conversations, and for all of his help as my graduate advisor.

\section{Background}
\label{background}

In order to introduce notation and terminology for the remainder of the paper, we give a quick summary of some relevant background on sutured manifolds and Heegaard diagrams. To unify the approach, the results in \cite{Naturality} are most often phrased in terms of sutured manifolds. Since we are interested here in the closed variants of Heegaard Floer homology, we will set up some background in order to be able to rephrase the results we use from \cite{Naturality} in language more typically used for the closed invariants. 

To begin, we will describe how moves on sutured Heegaard diagrams relate to the typical Heegaard moves one considers on Heegaard diagrams for closed 3 manifolds. Next we will recall the definition of the graph of sutured isotopy diagrams $\mathcal{G}(\mathcal{S}_{\text{man}})$ introduced in Section \ref{Introduction}, and describe an isomorphism to a graph $\mathcal{G}_{\text{man}}$ of closed isotopy diagrams which we will consider instead of $\mathcal{G}(\mathcal{S}_{\text{man}})$ throughout the remainder of the paper. We refer the reader to \cite[Section 2.1]{Naturality} for a more detailed treatment of all of the background in this section. 

\subsection{Background on Sutured Manifolds}

In this paper we will be concerned primarily with closed $3$ manifolds, but we will need to refer to numerous results about sutured $3$ manifolds along the way. In particular, our results depend on notions of sutured $3$ manifolds, sutured diagrams and embedded sutured diagrams for such manifolds, various notions of equivalence of such diagrams, and sutured Heegaard moves. While these notions may be standard, some inequivalent definitions certainly exist, so we explicitly refer the reader to \cite{Naturality} for background on the definitions we will use throughout this paper. We note that the sutured Heegaard moves play a role analogous to that of pointed Heegaard moves on Heegaard diagrams for closed $3$-manifolds. There are moves called $\balpha$ and $\bbeta$ equivalences (which correspond to sequences of handleslides), as well as stablizations and destablizations, isotopies, and diffeomorphisms. Finally, we note that by restricting attention to the isotopy class of attaching curves on a diagram, one obtains a well-defined notion of a sutured \emph{isotopy diagram}, and one can make sense of sutured Heegaard moves considered as moves on the isotopy diagrams (eg there is a well-defined notion of a diffeomorphism of isotopy diagrams). We again refer the reader to \cite{Naturality} for the relevant definitions of such sutured Heegaard moves; the main relevance here will be their relation to Heegaard moves on diagrams for closed $3$ manifolds. 

\subsection{A Correspondence Between Closed and Sutured Manifolds}

Our goal in this paper is to ultimately establish facts about the Heegaard Floer invariants for closed 3 manifolds, so we need a way to translate between sutured and closed manifolds in the cases of interest. Furthermore, certain properties of this correspondence are needed to ensure that the techniques used to obtain functoriality in \cite{Naturality} which we import can be applied to the closed setting of interest here. For our purposes, it will be sufficient to note that there is a correspondence between closed, oriented and based $3$ manifolds and sutured manifolds, and that under this correspondence:
\begin{enumerate}
	\item Isotopies of attaching curves in the sutured diagram yield \emph{pointed} isotopies (ie\ isotopies which do not cross the basepoint) of attaching curves in the closed diagram.
	\item Diffeomorphisms of sutured isotopy diagrams yield pointed diffeomorphisms of pointed closed isotopy diagrams.
	\item Stabilizations of sutured isotopy diagrams correspond to stabilizations of pointed isotopy diagrams.
	\item Two sutured isotopy diagrams $H_{1} = (\Sigma, \balpha_{1}, \bbeta_{1})$ and $H_{2}= (\Sigma, \balpha_{2}, \bbeta_{2})$ are $\balpha$-equivalent if and only if the curves $\balpha_{1}$ and $\balpha_{2}$ are related by a sequence of handleslides in the corresponding pointed isotopy diagrams, where the handleslides never cross the basepoint. The analogous statement holds for $\bbeta$-equivalent sutured isotopy diagrams.
\end{enumerate}

Since these last sorts of equivalences will play a prominent role throughout the paper, we introduce terminology introduced in \cite{FourManifoldInvariants} to describe them:
\begin{defn}
	Given two closed, pointed Heegaard diagrams $\mathcal{H}_{1} = (\Sigma, \balpha_{1}, \bbeta_{1}, z)$ and $\mathcal{H}_{2} = (\Sigma, \balpha_{2}, \bbeta_{2}, z)$ we say they are \emph{strongly equivalent} if they are related by a sequence of isotopies and handleslides which do not cross the basepoint. If the diagrams are related by a sequence of isotopies, and handleslides which occur only among the $\balpha$ curves, we say the diagrams are \emph{strongly $\balpha$-equivalent}. If the diagrams are related by a sequence of isotopies, and handleslides which occur only among the $\bbeta$ curves, we say the diagrams are \emph{strongly $\bbeta$-equivalent}.
\end{defn}

\subsection{Graphs of Heegaard Diagrams}
\label{GraphsofHeegaardMoves}
Following \cite[Definition 2.22]{Naturality}, construct a directed graph $\mathcal{G}$ as follows. The class of vertices, $|\mathcal{G}|$, of $\mathcal{G}$ is given by the class of isotopy diagrams of sutured manifolds. Given two isotopy diagrams $H_{1},H_{2} \in |\mathcal{G}|$, the oriented edges from $H_{1}$ to $H_{2}$ come in four flavors $$\mathcal{G}(H_{1}, H_{2}) = \mathcal{G}_{\alpha}(H_{1},H_{2}) \cup  \mathcal{G}_{\beta}(H_{1},H_{2})  \cup \mathcal{G}_{\text{stab}}(H_{1},H_{2})   \cup \mathcal{G}_{\text{diff}}(H_{1},H_{2}). $$
Here \begin{enumerate}
	\item $\mathcal{G}_{\alpha}(H_{1},H_{2})$ consists of a single edge if the diagrams are $\balpha$-equivalent.
	\item$\mathcal{G}_{\beta}(H_{1},H_{2})$ consists of a single edge if the diagrams are $\bbeta$-equivalent.
	\item $\mathcal{G}_{\text{stab}}(H_{1},H_{2})$ consists of a single edge if the diagrams are related by a stabilization or destabilization.
	\item$\mathcal{G}_{\text{diff}}(H_{1},H_{2})$ consists of a collection of edges, with one edge for each diffeomorphism between the isotopy diagrams.
\end{enumerate}

We denote by $\mathcal{G}_{\alpha}, \mathcal{G}_{\beta}, \mathcal{G}_{\text{stab}}$ and $\mathcal{G}_{\text{diff}}$ the subgraphs of $\mathcal{G}$ arising from only considering the corresponding edges on the class of vertices $|\mathcal{G}|$.

There is an analog of the Reidemeister Singer theorem for sutured manifolds (applied to sutured \textit{diagrams}):

\begin{prop}\cite[Proposition 2.23]{Naturality}
	\label{ReidemeisterSinger1}
	Two isotopy diagrams $H_{1}$, $H_{2} \in |\mathcal{G}|$ can be connected by an oriented path in $\mathcal{G}$ if and only if they define diffeomorphic sutured manifolds. 
\end{prop}
\begin{remark}
	By the definition of $\mathcal{G}$, if there is an unoriented path from $H_{1}$ to $H_{2}$ then there is also an oriented path from $H_{1}$ to $H_{2}$. 
\end{remark}

Let $S(H)$ denote the sutured manifold associated to the isotopy diagram $H$. Given any set $\mathcal{S}$ of diffeomorphism types of sutured manifolds, denote by $\mathcal{G}(\mathcal{S})$ the full subgraph of $\mathcal{G}$ spanned by those isotopy diagrams $H$ for which $S(H) \in \mathcal{S}$. For our purposes, the case of interest will be $\mathcal{S} = \mathcal{S}_{\text{man}}$. This is the set of diffeomorphism types of sutured manifolds which arise as the images of closed, oriented, based 3-manifolds under the correspondence discussed above.

Let $\mathcal{G}_{\text{man}}$ be the oriented graph with vertices given by pointed isotopy Heegaard diagrams of closed, connected $3$ manifolds, and with the edges from an isotopy diagram $H_{1}$ to an isotopy diagram $H_{2}$ given by
$$\mathcal{G}_{\text{man}}(H_{1}, H_{2}) = \mathcal{G}_{\text{man}}^{\alpha}(H_{1},H_{2}) \cup  \mathcal{G}_{\text{man}}^{\beta}(H_{1},H_{2})  \cup \mathcal{G}_{\text{man}}^{\text{stab}}(H_{1},H_{2})   \cup \mathcal{G}_{\text{man}}^{\text{diff}}(H_{1},H_{2})$$
where \begin{enumerate}
	\item $\mathcal{G}_{\text{man}}^{\alpha}(H_{1},H_{2})$ consists of a single edge if the diagrams are strongly $\balpha$-equivalent.
	\item$\mathcal{G}_{\text{man}}^{\beta}(H_{1},H_{2})$ consists of a single edge if the diagrams are strongly $\bbeta$-equivalent.
	\item $\mathcal{G}_{\text{man}}^{\text{stab}}(H_{1},H_{2})$ consists of a single edge if the diagrams are related by a stabilization or destabilization.
	\item$\mathcal{G}_{\text{man}}^{\text{diff}}(H_{1},H_{2})$ consists of a collection of edges, with one edge for each pointed diffeomorphism between the isotopy diagrams.
\end{enumerate}
We provide a sketch of a piece of the graph $\mathcal{G}_{\text{man}}$ in Figure \ref{GmanGraph} below. The following analog of Proposition \ref{ReidemeisterSinger1} holds in the closed and pointed setting:

\begin{prop}\cite[Proposition 7.1]{Disks1}
	\label{ReidemeisterSinger2}
	Two isotopy diagrams $H_{1}$, $H_{2} \in |\mathcal{G}_{man}|$ can be connected by an oriented path in $\mathcal{G}_{man}$ if and only if they define diffeomorphic pointed manifolds. 
\end{prop}

Finally we note that the preceding arguments specify an isomorphism of graphs
\begin{equation}
\label{IsomorphismofGraphs}
T:\mathcal{G}(\mathcal{S}_{\text{man}}) \rightarrow \mathcal{G}_{\text{man}}
\end{equation}
which we will use implicitly in the remainder of the paper to rephrase certain results from \cite{Naturality} in terms of $\mathcal{G}_{\text{man}}$.

\begin{figure}[h!]
	\centering
	\begin{tikzpicture}
	\node[anchor=south west,inner sep=0] (image) at (0,0) {\includegraphics[width=\textwidth]{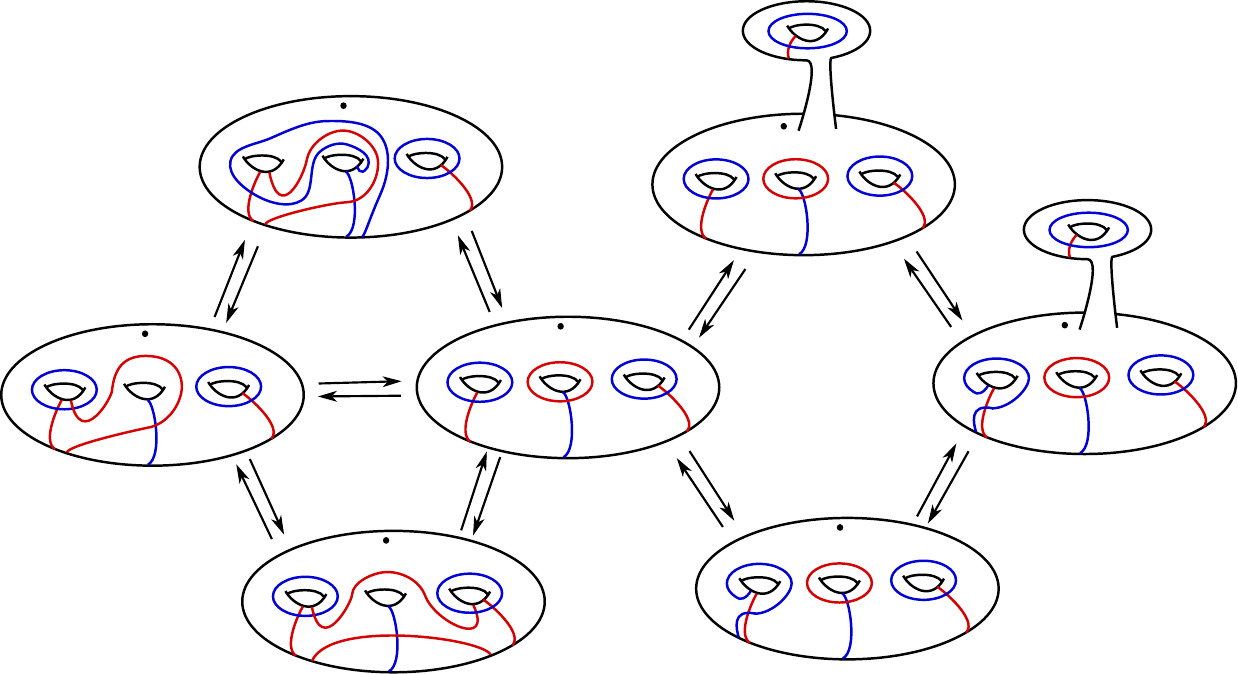}};
	\begin{scope}[x={(image.south east)},y={(image.north west)}]
	\node at (.165,.6) {$\bbeta$};
	\node at (.415,.61) {$d$};
	\node at (.29,.46) {$\balpha$};
	\node at (.418,.25) {$\balpha$};
	\node at (.18,.23) {$\balpha$};
	\node at (.55,.58) {$\sigma$};
	\node at (.55,.25) {$d$};
	\node at (.78,.25) {$\sigma$};
	\node at (.78,.59) {$d$};
	\end{scope}
	\end{tikzpicture}
	\caption{An illustration of a small subgraph in $\mathcal{G}_{\text{man}}$. The vertices are isotopy diagrams, which in the picture are depicted by particular Heegaard diagrams representing the isotopy class. We label each pair of edges with $\balpha, \bbeta, \sigma$ or $d$ according to whether the given pair of edges corresponds to a strong $\balpha$-equivalence, a strong $\bbeta$-equivalence, a stabilization/destabilization pair, or a diffeomorphism pair respectively. We use the convention that on each Heegaard diagram the collection of red attaching curves is denoted $\balpha$ while the collection of blue attaching curves is denoted $\bbeta$.}
	\label{GmanGraph}
\end{figure}

\section{Heegaard Invariants}
\label{HeegaardInvariants}
We now make precise two notions of what one might mean by a Heegaard invariant of closed $3$-manifolds. For the interested reader's convenience, we note that the definitions originally given in \cite{Naturality} apply to sutured manifolds and the graph $\mathcal{G}(\mathcal{S}_{\text{man}})$. Instead, we state here the equivalent definitions phrased in terms of closed manifolds and the graph $\mathcal{G}_{\text{man}}$.

Suppose we produce some assignment of algebraic objects to Heegaard diagrams (the vertices of the graph $\mathcal{G}_{\text{man}}$), and an assignment of maps between these algebraic objects to each Heegaard move between two diagrams (the edges of $\mathcal{G}_{\text{man}}$). Given Proposition \ref{ReidemeisterSinger2}, the minimal requirement we should ask of such an assignment to obtain an invariant of the underlying $3$-manifold is for edges in $\mathcal{G}_{\text{man}}$ to be assigned isomorphisms. Given any category $\mathcal{C}$, we have:

\begin{defn}\cite[Compare Definition 2.24]{Naturality} \label{ClosedWeakHeegaardInvariant}
	A \emph{weak Heegaard invariant of closed $3$-manifolds} is a morphism of graphs $F: \mathcal{G}_{\text{man}} \rightarrow \mathcal{C}$ for which $F(e)$ is an isomorphism for all edges $e \in \mathcal{G}_{\text{man}}$.
\end{defn}

Of course, this level of invariance was established for Heegaard Floer homology at the outset. 

\begin{thm}[\cite{Disks1}]
	\label{HeegaardFloerisWeak}
	The morphisms
	$$\widehat{HF}, HF^{-}, HF^{+}, HF^{\infty}: \mathcal{G}_{\text{man}} \rightarrow \mathbb{F}_{2}[U]\text{-}\text{Mod}$$
	and 
	$$\widehat{HF}, HF^{-}, HF^{+}, HF^{\infty}: \mathcal{G}_{\text{man}} \rightarrow \mathbb{Z}[U] \text{-}\text{Mod}$$
	are weak Heegaard invariants of closed $3$-manifolds.
\end{thm}

The above results also immediately yield

\begin{cor}
	The morphisms
	$$HF^{\circ}: \mathcal{G}_{\text{man}} \rightarrow P(\mathbb{Z}[U]\text{-}\text{Mod})$$
	are weak Heegaard invariants of closed $3$-manifolds.
\end{cor}

In Section \ref{WeakHeegaardFloerInvariantsSection} we will recall the definition of these morphisms of graphs precisely. In particular, since the vertices of $\mathcal{G}_{\text{man}}$ are isotopy diagrams, we will need to explain the meaning of $HF^{\circ}(H)$ when $H$ is an isotopy diagram rather than a particular Heegaard diagram representing the isotopy class.

\begin{remark}
	For the reader referencing the corresponding results stated in \cite{Naturality}, we note that in \cite[Theorem 2.26]{Naturality}, Theorem \ref{HeegaardFloerisWeak} is instead phrased as ``$HF^{\circ}: \mathcal{G}(\mathcal{S}_{\text{man}}) \rightarrow \mathbb{F}_{2}[U]\text{-}\text{Mod}$ are weak Heegaard invariants". Of course, as they were originally defined $HF^{\circ}$ are invariants assigned to closed, pointed Heegaard diagrams; the meaning of $HF^{\circ}(H)$ for $H$ a sutured isotopy diagram in this statement is interpreted as follows. Recall that vertices of $\mathcal{G}(\mathcal{S}_{\text{man}})$ correspond to isotopy diagrams $H$ of sutured manifolds corresponding to closed, oriented $3$-manifolds Y. Given an actual sutured diagram $\mathcal{H} = (\Sigma, \balpha, \bbeta)$ (not up to isotopy) for such a $3$-manifold, the boundary of the Heegaard surface $\Sigma$ is $S^{1}$, so it can be capped off with a disk to obtain a closed surface $\overline{\Sigma}$ and a pointed Heegaard diagram $\overline{\mathcal{H}} = (\overline{\Sigma}, \balpha, \bbeta, z)$ for $Y$, where the basepoint $z$ is chosen to lie in the disk. Thus given a sutured diagram $\mathcal{H}$ representing the isotopy diagram $H$, we define $CF^{\circ}(\mathcal{H}):= CF^{\circ}(\overline{\mathcal{H}})$. Finally, we will describe how the collection $\{ CF^{\circ}(\mathcal{H}) \}$ gives rise to $CF^{\circ}(H)$ in Section \ref{WeakHeegaardFloerInvariants}. Equivalently, using the isomorphism of graphs $T$ specified in Equation \eqref{IsomorphismofGraphs}, the definitions above will amount to defining $HF^{\circ}(H) := HF^{\circ}(T(H))$ for $H$ a sutured isotopy diagram.
\end{remark}

Let $\text{Man}_{*}$ be the category whose class of objects consists of closed, connected, oriented and based 3-manifolds, and whose morphisms are basepoint preserving diffeomorphisms. In \cite{Disks1} and \cite{FourManifoldInvariants}, significant progress was made towards showing that the weak Heegaard invariants in the theorem above can in fact be assembled into functors from $\text{Man}_{*}$ to $\mathbb{F}_{2}[U] \text{-} \text{Mod}$. However, there was a gap in the proof. In \cite{Naturality}, the authors carefully analyzed the dependence of such a result on the nature of embedded (versus abstract) Heegaard diagrams, and basepoints, and set up a framework which allowed them to finish this program. To do so, they introduced a stronger notion of a Heegaard invariant which we now describe. 

To begin, we introduce some terminology for particular subgraphs in $\mathcal{G}_{\text{man}}$ (or more generally in $\mathcal{G}$) which will serve as minimal data on which this new notion of invariance will rely. 

\begin{defn}\cite[Definition 2.29]{Naturality}
	\label{distinguishedrectangle}
	A \emph{distinguished rectangle} is a subgraph of $\mathcal{G}_{\text{man}}$ of the form
	\begin{equation*}
	\begin{tikzcd}
	H_1  \arrow{r}{e} \arrow{d}{f} & H_{2} \arrow{d}{g}   \\
	H_3 \arrow{r}{h} & H_4   \\
	\end{tikzcd} 
	\end{equation*}
	which satisfies one of the following conditions:
	
	\begin{enumerate}
		\item The arrows $e$ and $h$ are strong $\balpha$-equivalences, and the arrows $f$ and $g$ are strong $\bbeta$-equivalences.
		\item The arrows $e$ and $h$ are either both  strong $\balpha$-equivalences or both strong $\bbeta$-equivalences, and the arrows $f$ and $g$ are stabilizations.
		\item The arrows $e$ and $h$ are either both  strong $\balpha$-equivalences or both strong  $\bbeta$-equivalences, and the arrows $f$ and $g$ are diffeomorphisms. Furthermore, $f=g$ (Note in this case $\Sigma_{1}=\Sigma_{2}$, and $\Sigma_{3}= \Sigma_{4}$, so this requirement makes sense).
		\item All of the arrows $e,f,g$ and $h$ are stabilizations. Furthermore, there are disjoint disks $D_{1},D_{2} \subset \Sigma_{1}$ and disjoint punctured tori $T_{1},T_{2} \subset \Sigma_{4}$ such that $\Sigma_{1} \setminus (D_{1} \cup D_{2})= \Sigma_{4} \setminus (T_{1} \cup T_{2})$, $\Sigma_{2} = (\Sigma_{1} \setminus D_{1}) \cup T_{1}$, and $\Sigma_{3} = (\Sigma_{1} \setminus D_{2}) \cup T_{2}$.
		\item The arrows $e$ and $h$ are stabilizations, and the arrows $f$ and $g$ are diffeomorphisms. Furthermore, the diffeomorphism $g$ is an extension of the diffeomorphism $f$ in the following sense. There are disks $D_{1} \subset \Sigma_{1}$, $D_{3} \subset \Sigma_{3}$ and punctured tori $T_{2} \subset \Sigma_{2}$, $T_{4} \subset \Sigma_{4}$ such that $\Sigma_{1} \setminus D_{1} = \Sigma_{2} \setminus T_{2}$, $\Sigma_{3} \setminus D_{3} = \Sigma_{4} \setminus T_{4}$, $f(D_{1}) = D_{2}$, $g(T_{3}) = T_{4}$ and $f |_{\Sigma_{1} \setminus D_{1}} = g |_{\Sigma_{2} \setminus T_{2}}$.
	\end{enumerate}
	
\end{defn}

We illustrate cases 4 and 5 schematically in Figures \ref{Distinguishedrectanglecase4} and \ref{Distinguishedrectanglecase5} below.

\begin{figure}[h]
	\centering
	\begin{tikzpicture}
	\node[anchor=south west,inner sep=0] (image) at (0,0) {\includegraphics[width=0.7\textwidth]{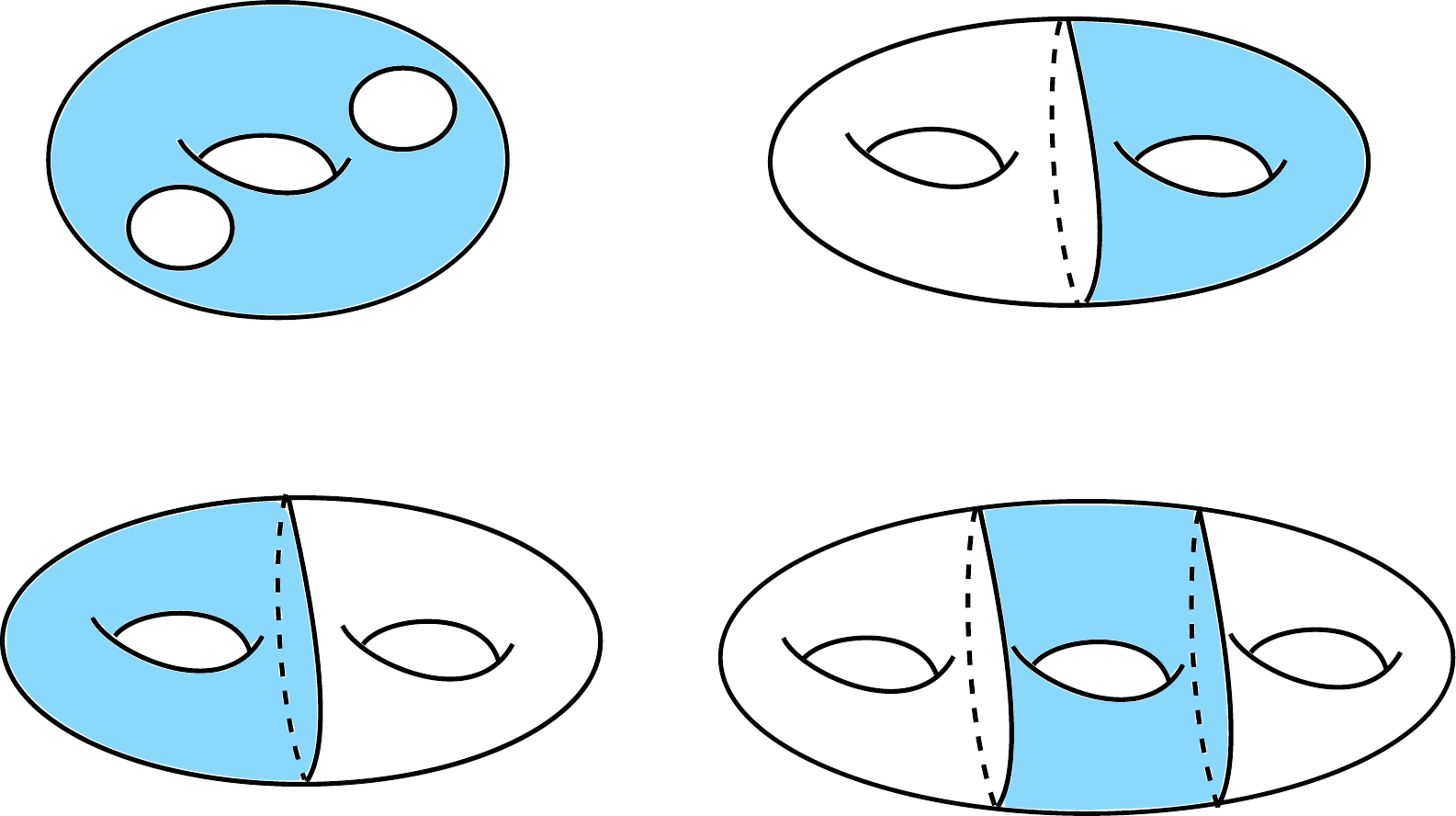}};
	\begin{scope}[x={(image.south east)},y={(image.north west)}]
	\node at (.12,.72) { $D_{1}$};
	\node at (.275,.86) { $D_{2}$};
	\node at (.67,.7) { $T_{1}$};
	\node at (.27,.3) { $T_{2}$};
	\node at (.63,.09) { $T_{1}$};
	\node at (.88,.29) { $T_{2}$};
	\node (A) at (.2,.58) {};
	\node (B) at (.2,.42) {};
	\draw [->] (A) -- (B) node[midway, right] {$f$};
	\node (C) at (.75,.58) {};
	\node (D) at (.75,.42) {};
	\draw [->] (C) -- (D) node[midway, right] {$g$};
	\node (E) at (.38,.35) {};
	\node (F) at (.52,.35) {};
	\draw [->] (E) to [out = 30, in = 150] (F);
	\node (G) at (.38,.7) {};
	\node (H) at (.52,.7) {};
	\node at (.45,.71) {$e$};
	\node at (.45,.43) {$h$};
	\draw [->] (G) to [out = -30, in = 210] (H);
	\end{scope}
	\end{tikzpicture}
	\caption{A schematic illustrating case 4 in the definition of a distinguished rectangle. The blue regions indicate the identifications specified in case 4. For ease of visualization, we suppress the attaching curve data in the initial diagram and in the stabilizations.}
	\label{Distinguishedrectanglecase4}
\end{figure}

\begin{figure}[h!]
	\centering
	\begin{tikzpicture}
	\node[anchor=south west,inner sep=0] (image) at (0,0) {\includegraphics[width=0.7\textwidth]{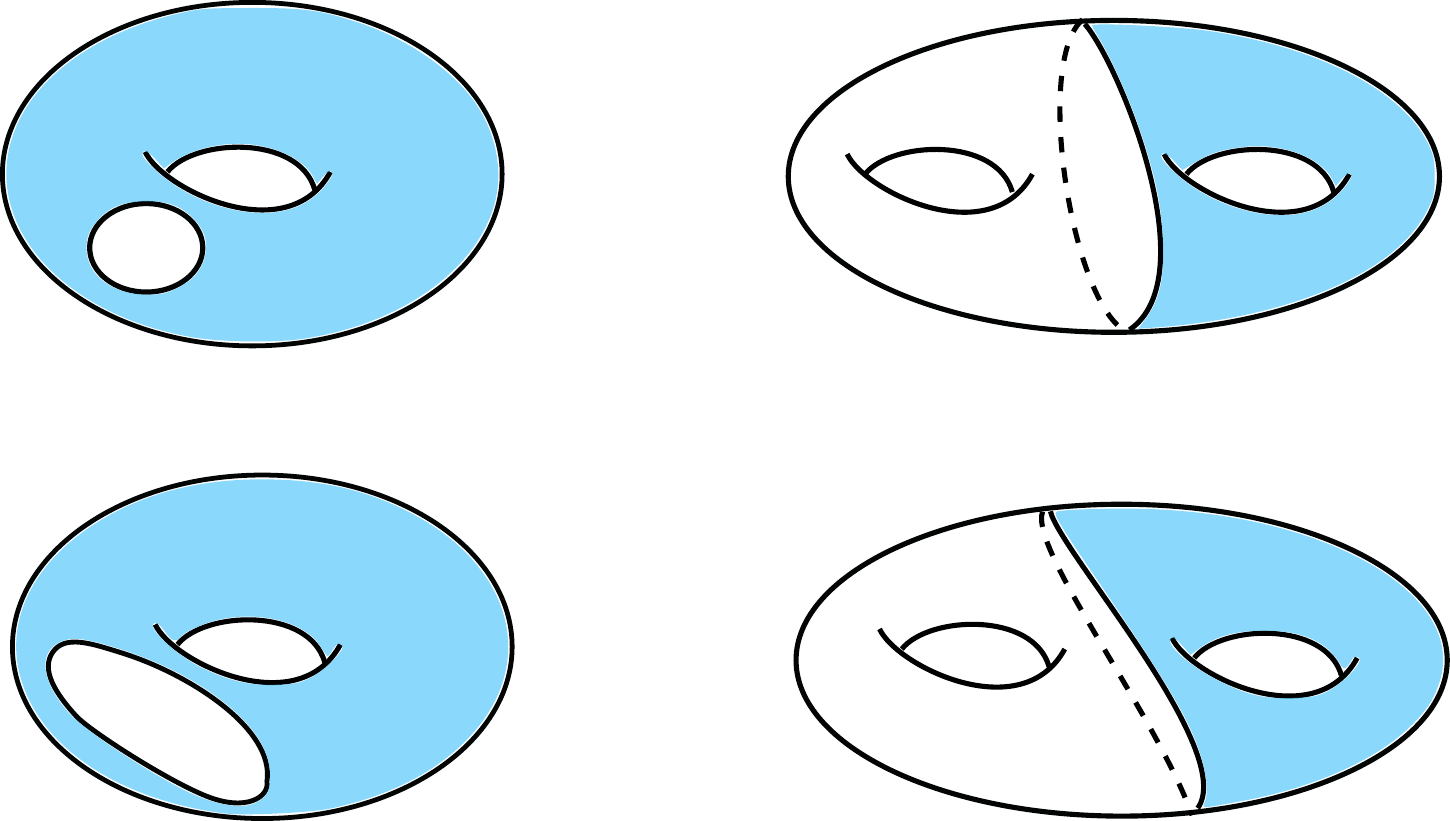}};
	\begin{scope}[x={(image.south east)},y={(image.north west)}]
	\node at (.1,.7) { $D_{1}$};
	\node at (.11,.12) { $D_{3}$};
	\node at (.67,.68) { $T_{2}$};
	\node at (.7,.1) { $T_{4}$};
	\node (A) at (.2,.58) {};
	\node (B) at (.2,.42) {};
	\draw [->] (A) -- (B) node[midway, right] {$f$};
	\node (C) at (.75,.58) {};
	\node (D) at (.75,.42) {};
	\draw [->] (C) -- (D) node[midway, right] {$g$};
	\node (E) at (.38,.35) {};
	\node (F) at (.52,.35) {};
	\draw [->] (E) to [out = 30, in = 150] (F);
	\node (G) at (.38,.7) {};
	\node (H) at (.52,.7) {};
	\node at (.45,.71) {$e$};
	\node at (.45,.43) {$h$};
	\draw [->] (G) to [out = -30, in = 210] (H);
	\end{scope}
	\end{tikzpicture}
	\caption{A schematic illustrating case 5 in the definition of a distinguished rectangle. The blue regions indicate the identifications of the regions specified in case 5. For ease of visualization, we suppress the attaching curve data in each diagram.}
	\label{Distinguishedrectanglecase5}
\end{figure}

\begin{defn}\cite[Definition 2.31]{Naturality}\label{SimpleHandleswap}
	A \emph{simple handleswap} is a subgraph of $\mathcal{G}_{\text{man}}$ of the form
	\begin{equation*}
	\begin{tikzcd}
	H_1  \arrow{rd}{e} &     \\
	H_3 \arrow{u}{g} & H_2 \arrow{l}{f}   \\
	\end{tikzcd} 
	\end{equation*}
	
	such that:
	\begin{enumerate}
		\item The isotopy diagrams $H_{i}$ are given by $H_{i} = (\Sigma \# \Sigma_{0}, [\balpha_{i}], [\bbeta_{i}])$, where $\Sigma_{0}$ is a genus two surface.
		\item $e$ is a strong $\balpha$-equivalence, $f$ is a strong $\bbeta$-equivalence, and $g$ is a diffeomorphism.
		\item In the punctured genus two surface $P = (\Sigma \# \Sigma_{0}) \setminus \Sigma$, the above triangle is equivalent to the triangle in Figure \ref{handleswappic} in the following sense. There are diffeomorphisms from $P \cap H_{i}$ to the green discs labeled $H_{i}$ in the figure, such that the image of the $\alpha$ curves are the red circles in the figures, and the image of the $\beta$ curves are the blue circles in the figures.
		\item The diagrams $H_{1}, H_{2}$ and $H_{3}$ are identical when restricted to $\Sigma$.
	\end{enumerate}
\end{defn}

\begin{figure}[h!]
	\centering
	\begin{tikzpicture}
	\node[anchor=south west,inner sep=0] (image) at (0,0) {\includegraphics[width=\textwidth]{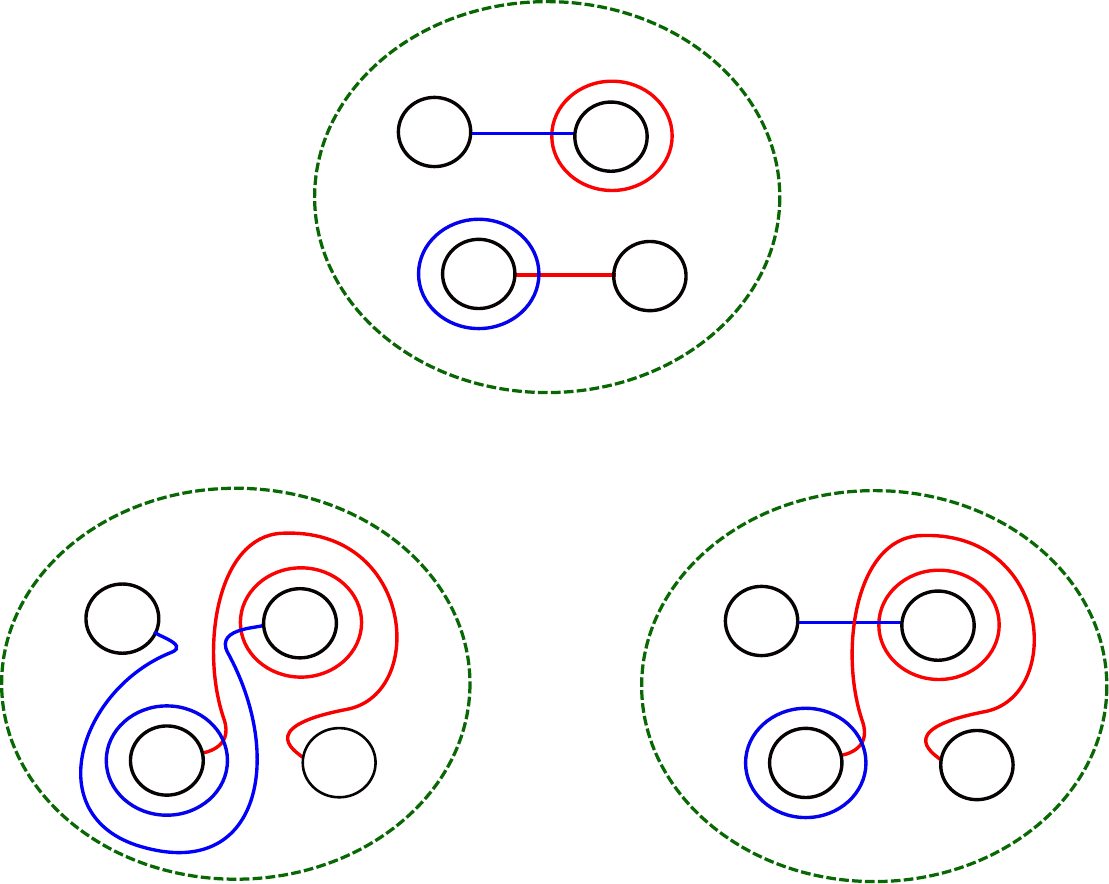}};
	\begin{scope}[x={(image.south east)},y={(image.north west)}]
	\node at (.71,.9) { \huge $H_{1}$};
	\node at (.965,.4) {\huge $H_{2}$};
	\node at (.41,.07) {\huge $H_{3}$};
	\node (A) at (.65,.58) {};
	\node (B) at (.7,.45) {};
	\draw [->] (A) -- (B) node[midway, right] { \LARGE $e$};
	\node (C) at (.57,.3) {};
	\node (D) at (.43,.3) {};
	\draw [->] (C) -- (D) node[midway, above] {\LARGE $f$};
	\node (E) at (.31,.45) {};
	\node (F) at (.36,.58) {};
	\draw [->] (E) -- (F) node[midway, right] {\LARGE $g$};
	\node[color = red] at (.52,.66) {\large $\alpha_{1}$};
	\node[color = red] at (.63,.85) {\large $\alpha_{2}$};
	\node[color = blue] at (.36,.7) {\large $\beta_{1}$};
	\node[color = blue] at (.47,.87) {\large $\beta_{2}$};
	\node[color = red] at (.96,.25) {\large $\alpha_{1}'$};
	\node[color = blue] at (.05,.15) {\large $\beta_{1}'$};
	\node at (.39,.85) {\Huge F};
	\node at (.555,.845) {\reflectbox{ \Huge F}};
	\node at (.43,.69) {\Huge R};
	\node at (.59,.69) {\reflectbox{ \Huge R}};
	\node at (.68,.29) {\Huge F};
	\node at (.85,.29) {\reflectbox{ \Huge F}};
	\node at (.72,.14) {\Huge R};
	\node at (.885,.13) {\reflectbox{ \Huge R}};
	\node at (.11,.29) {\Huge F};
	\node at (.27,.29) {\reflectbox{ \Huge F}};
	\node at (.15,.14) {\Huge R};
	\node at (.31,.13) {\reflectbox{ \Huge R}};

	\end{scope}
	\end{tikzpicture}
	\caption{The standard simple handleswap.}
	\label{handleswappic}
\end{figure}

With these notions in hand, the stronger sense of invariance we will ask of our Heegaard invariants is as follows.

\begin{defn}\cite[Definition 2.32]{Naturality} \label{StrongHeegaardInvariant}
	A \emph{strong Heegaard invariant of closed $3$-manifolds} is a weak Heegaard invariant $F: \mathcal{G}_{\text{man}} \rightarrow \mathcal{C}$ that additionally satisfies the following axioms:
	\begin{enumerate}
		\item \textbf{Functoriality}: The restriction of $F$ to $\mathcal{G}^{\alpha}_{\text{man}}$, $\mathcal{G}^{\beta}_{\text{man}}$ and $\mathcal{G}^{\text{diff}}_{\text{man}}$ are functors to $\mathcal{C}$. If $e: H_1 \rightarrow H_2$ is a stabilization and $e': H_2 \rightarrow H_1$ is the corresponding destabilization, then $F(e') = F(e)^{-1}$.
		\item \textbf{Commutativity}: For every distinguished rectangle in $\mathcal{G}_{\text{man}}$,
		\begin{equation*}
		\begin{tikzcd}
		H_1 \arrow{r}{e} \arrow{d}{f} & H_2 \arrow{d}{g}    \\
		H_3 \arrow{r}{h} & H_4     \\
		\end{tikzcd}
		\end{equation*}
		
		we have $F(g) \circ F(e) = F(h) \circ F(f)$.
		
		\item \textbf{Continuity}: If $H \in | \mathcal{G}_{\text{man}}|$ and $e \in \mathcal{G}^{\text{diff}}_{\text{man}}(H, H)$ is a diffeomorphism isotopic to $\text{Id}_{\Sigma}$, then $F(e) = \text{Id}_{F(H)}$.
		
		\item \textbf{Handleswap Invariance}: For every simple handleswap in $\mathcal{G}_{\text{man}}$,
		\begin{equation*}
		\begin{tikzcd}
		H_1  \arrow{rd}{e} &     \\
		H_3 \arrow{u}{g} & H_2 \arrow{l}{f}   \\
		\end{tikzcd} 
		\end{equation*}
		
		we have $F(g) \circ F(f) \circ F(e) = \text{Id}_{F(H_1)}$.
		
	\end{enumerate}
\end{defn}

As we will summarize in Section \ref{ProjectiveNaturalityFromStrong}, it was shown in \cite{Naturality} that for any weak Heegaard invariant the axioms required above are sufficient to ensure the images of the invariant, when restricted to a particular subgraph of $\mathcal{G}_{\text{man}}$ whose vertices represent a fixed $3$-manifold, form a transitive system in the given category. For certain categories $\mathcal{C}$, this in turn is enough to ensure that the assignments of the invariants can be understood as a functor from an appropriate category of $3$-manifolds.

\section{Transitive Systems of Chain Complexes and Projectivization}
\label{TransitiveSystemSection}
In this section we describe the algebraic framework which will be necessary to phrase our projective functoriality results. To begin with, we recall the following fundamental notions.

\begin{defn}
	\label{DirectedSetDefinition}
	A \emph{directed set} $(I, \leq)$ is a set $I$ together with a reflexive and transitive binary relation $\leq$, such that for every pair of elements $a,b \in I$ there is an element $c \in I$ with $a \leq c$ and $b \leq c$.
	
\end{defn}

\begin{defn}
	\label{TransitiveSystem}
	Let $\mathcal{C}$ be a category, and $(I, \leq)$ be a directed set. Given a collection of objects $\{O_{i}\}$ in $\mathcal{C}$ indexed by $I$, and a collection of morphisms $\{f_{i,j}: O_{i} \rightarrow O_{j}\}$ for all $i, j \in I$ with $i \leq j$, we say the collections are a \emph{transitive system in $\mathcal{C}$ (indexed by I)} if they satisfy:
	\begin{enumerate}
		\item $f_{i,i} = \text{Id}_{O_{i}}$
		\item $f_{i,k} = f_{j,k} \circ f_{i,j}$
	\end{enumerate}
\end{defn}
We also have the following notion of morphisms between transitive systems:
\begin{defn}
	\label{MapOfTransitiveSystems}
	Given two transitive systems $T_{1} = \{I_{1}, \leq, \{O_{i}\},  \{f_{i,j}\}\}$ and $T_{2} = \{I_{2}, \leq, \{P_{i}\}, \{g_{i,j}\}\}$ in a category $\mathcal{C}$, a \emph{morphism of transitive systems} $(M, \{n_{i}\})$  from $T_{1}$ to $T_{2}$ consists of a map of directed sets $M: I_{1} \rightarrow I_{2}$ and a collection of morphisms $\{n_{i}: O_{i} \rightarrow P_{M(i)} \}$ in $\mathcal{C}$ such that for all $i,j \in I_{1}$ with $i \leq j$  the squares
	\begin{equation*}
	\begin{tikzcd}
	O_{i}  \arrow{r}{n_{i}} \arrow{d}{f_{i,j}} &  P_{M(i)} \arrow{d}{g_{M(i),M(j)}}   \\
	O_{j} \arrow{r}{n_{j}} & P_{M(j)}   \\
	\end{tikzcd} 
	\end{equation*}
	commute in $\mathcal{C}$.
	We denote the resulting category of transitive systems in $\mathcal{C}$ by $\text{Trans}(\mathcal{C})$.
\end{defn}

Finally, given a transitive system in $\text{Trans}(\mathcal{C})$ indexed by J, we obtain what one might call a two dimensional transitive system. Such a two dimensional transitive system naturally has the structure of a transitive system in $\mathcal{C}$ indexed by $I \times J$, where $(i,j) \leq (i',j')$ if and only if $i \leq i'$ and $j \leq j'$.

We now explain how these notions will arise in the context of our results. We will begin by considering the category $\text{Kom}(\mathbb{Z}[U]\text{-}\text{Mod})$, the homotopy category of chain complexes of $\mathbb{Z}[U]$-modules. To each pointed isotopy diagram $H$, corresponding to a vertex of $\mathcal{G}_{\text{man}}$, we will assign a transitive system $CF^{-}(H) \in \text{Trans}(\text{Kom}(\mathbb{Z}[U]\text{-}\text{Mod}))$. To be more explicit about the nature of this construction and bridge the gap to the language defined above, given an isotopy diagram $H$ we consider the directed set $(I, \leq)$ with $I$ the set of Heegaard diagrams in the given isotopy class, and $\leq$ the (in this case trivial, equivalence) relation on the set indicating existence of an isotopy between two elements. Then $CF^{-}(H)$ will be a transitive system in $\text{Kom}(\mathbb{Z}[U]\text{-}\text{Mod})$ indexed by $(I, \leq)$, with the objects in the transitive system being the Heegaard Floer chain complexes associated to individual diagrams in the fixed isotopy class, and the morphisms in the transitive systems being certain continuation maps between such complexes. The details of precisely how these assignments are made will be specified throughout the course of Section \ref{WeakHeegaardFloerInvariantsSection}. To a diffeomorphism, strong $\balpha$-equivalence, strong $\bbeta$-equivalence, or stabilization between two such isotopy diagrams $H_{1}$ and $H_{2}$ we will then associate a morphism of transitive systems from $CF^{-}(H_{1})$ to $CF^{-}(H_{2})$. Together, these assignments will yield a morphism of graphs
$$CF^{-}: \mathcal{G}_{\text{man}} \rightarrow \text{Trans}(\text{Kom}(\mathbb{Z}[U]\text{-}\text{Mod})).$$
This morphism of graphs may not be a strong Heegaard invariant. We will however be able to establish that this morphism of graphs satisfies the axioms required of a strong Heegaard invariant up to an overall sign in each of the axioms (2), (3) and (4) appearing Definition \ref{StrongHeegaardInvariant}. 

Equivalently, we will phrase this result in terms of an appropriate projectivization. Recall that given any category $\mathcal{C}$, with an equivalence relation $\sim$ on every hom set which furthermore respects composition, we may form the quotient category $\mathcal{C} = \mathcal{C} / \sim$. This is the category whose objects are those of $\mathcal{C}$, and whose morphisms are equivalence classes of morphisms with respect to $\sim$. Given an additive category $\mathcal{C}$, we define the \textit{projectivization of $\mathcal{C}$}, $P(\mathcal{C})$, to be the quotient category of $\mathcal{C}$ with respect to the relation $f \sim -f$ for all morphisms $f$. The last statement in the preceding paragraph is then given precisely by the following statement: considering now the category of transitive systems in the projectivized homotopy category, $\text{Trans}(P(\text{Kom}(\mathbb{Z}[U]\text{-}\text{Mod})))$, we will show that the morphism of graphs above yields a strong Heegaard invariant
$$CF^{-}: \mathcal{G}_{\text{man}} \rightarrow \text{Trans}(P(\text{Kom}(\mathbb{Z}[U]\text{-}\text{Mod}))).$$

\begin{remark} While the proliferation of transitive systems may seem undesirable, we were unable to produce another framework in which our naturality results could be phrased. There appear to be two issues that arise if one tries to use the same framework developed in \cite{Naturality} to phrase our projective results. 
	
	The first issue comes from the fact that the statement in Theorem \ref{StrongHeegaard2} is concerned with the Floer chain complexes. If one wanted to dispense with the category of transitive systems appearing in that statement, one would need to assign a \emph{single} chain complex $CF^{\circ}(H)$ of $\mathbb{Z}[U]\text{-}\text{modules}$ to an isotopy diagram $H$. As we will recall in the next section, what the Heegaard Floer construction actually produces for each isotopy diagram $H$ is a transitive system of chain homotopy equivalences between chain complexes of $\mathbb{Z}[U]\text{-}\text{modules}$. In general, it is not clear how one should define an object like a colimit of such a transitive system of chain complexes to obtain a single chain complex. We note that it seems likely that this issue is in fact a non-issue, for the following reason. We expect our transitive system of chain homotopy equivalences is homotopy coherent in the sense of \cite{HomotopyLimits}, which if true would allow one to define a single chain complex $CF^{\circ}(H)$ via a homotopy colimit. Indeed, that our transitive systems are homotopy coherent in this sense seems likely to follow from the results in \cite{FlexibleConstruction}.    
	
	However, even if one could assign to each isotopy diagram a single chain complex $CF^{\circ}(H)$, there is another key obstruction to phrasing Theorem \ref{Functoriality} without the use of transitive systems. In the proof of Theorem \ref{Functoriality}, which will be given in Section \ref{ProjectiveNaturalityFromStrong}, we will associate to each closed, pointed $3$-manifold a transitive system in $P(\mathbb{Z}[U]\text{-}\text{Mod})$. The author is unaware of a notion of a colimit in $P(\mathbb{Z}[U]\text{-}\text{Mod})$ which would allow Theorem \ref{Functoriality} to be stated without transitive systems, in such a way that it is also not merely reduced to a statement about the $\mathbb{F}_{2}$ invariants.
\end{remark}

\section{Projective Naturality from Strong Heegaard Floer Invariants}
\label{ProjectiveNaturalityFromStrong}
In this section we prove Theorem \ref{Functoriality} assuming Corollary \ref{StrongHeegaard}, which we will prove in turn in Section \ref{MainProof}. Our argument will follow the same logical structure as that used to prove the analogous result over $\mathbb{F}_{2}$ appearing in \cite[Theorem 1.5]{Naturality}. We provide the argument here for the reader's convenience, but note that the scheme is essentially the same.

In \cite{Naturality} Juh\'asz, Thurston and Zemke show that the images of any strong Heegaard invariant, appropriately restricted, fit into a transitive system. To make this precise, we introduce a few more definitions.

\begin{defn}
	Suppose $H_{1}$ and $H_{2}$ are embedded isotopy diagrams for a closed, oriented, pointed $3$-manifold $(Y, z)$, with Heegaard surfaces $\iota_{1}, \iota_{2}: (\Sigma_{1},z) , (\Sigma_{2},z) \hookrightarrow (Y,z)$. We say a diffeomorphism of isotopy diagrams $d: H_{1} \rightarrow H_{2}$ is \emph{isotopic to the identity in M} if $\iota_{2} \circ d: \Sigma_{1} \rightarrow (Y,z)$ is isotopic to $\iota_{1}: \Sigma_{1} \rightarrow (Y,z)$ relative to the basepoint.
\end{defn}

\begin{defn}
	Given $(Y, z)$, let $(\mathcal{G}_{\text{man}})_{(Y,z)}$ be the following subgraph of $\mathcal{G}_{\text{man}}$ whose vertices are embedded isotopy diagrams for $(Y,z)$. The edges $e \in (\mathcal{G}_{\text{man}})_{(Y,z)}(H_{1}, H_{2})$ between two isotopy diagrams again come in four flavors:
	
	$$(\mathcal{G}_{\text{man}})_{(Y,z)}(H_{1}, H_{2}) = \mathcal{G}_{\text{man}}^{\alpha}(H_{1},H_{2}) \cup  \mathcal{G}_{\text{man}}^{\beta}(H_{1},H_{2})  \cup \mathcal{G}_{\text{man}}^{\text{stab}}(H_{1},H_{2})   \cup (\mathcal{G}_{\text{man}}^{\text{diff}})^{0}(H_{1},H_{2}) $$
	Here $\mathcal{G}_{\text{man}}^{\alpha}$, $\mathcal{G}_{\text{man}}^{\beta}$ and $\mathcal{G}_{\text{man}}^{\text{stab}}$ are the same collections as in the definition of $\mathcal{G}_{\text{man}}$, while $(\mathcal{G}_{\text{man}}^{\text{diff}})^{0}(H_{1},H_{2})$ consists of one edge for each element in the set of diffeomorphisms from $H_{1}$ to $H_{2}$ which are isotopic to the identity in $M$. 
\end{defn}

With these notions in hand, we have a stronger version of Proposition \ref{ReidemeisterSinger2} which applies now to embedded diagrams for some fixed $(Y, z)$:

\begin{prop}\cite[Proposition 2.36]{Naturality}
	Given $(Y, z)$, any two vertices in the graph $(\mathcal{G}_{\text{man}})_{(Y,z)}$ can be connected by an oriented path in $(\mathcal{G}_{\text{man}})_{(Y,z)}$. 
\end{prop}

The salient feature of a strong Heegaard invariant, $F$, is that the the isomorphisms $F(e)$ associated to edges $e$ in $(\mathcal{G}_{\text{man}})_{(Y,z)}$ fit into a transitive system. This follows from the fact that the isomorphism associated to a path depends only on the endpoints:  

\begin{thm}[Theorem 2.38 in \cite{Naturality}]\label{StrongImpliesTransitive}
	Let $F: \mathcal{G}_{\text{man}} \rightarrow \mathcal{C}$ be a strong Heegaard invariant. Given two isotopy diagrams $H, H' \in |(\mathcal{G}_{\text{man}}){(Y,z)}|$ and any two oriented paths $\eta$ and $\nu$ in $(\mathcal{G}_{\text{man}}){(Y,z)}$ from $H$ to $H'$, we have 
	$$F(\eta) = F(\nu)$$
\end{thm}

Now, for any two isotopy diagrams $H,H'$ and an oriented path $\eta$ from $H$ to $H'$, we can define the map $F_{H,H'} = F(\eta)$. 

\begin{cor}[Corollary 2.41 in \cite{Naturality}]
	\label{StrongImpliesNoMonodromy}
	Suppose that $H, H', H'' \in |(\mathcal{G}_{\text{man}})_{(Y,z)}|$. Then 
	$$ F_{H,H''} = F_{H',H''} \circ F_{H,H'}$$
\end{cor}

These results should provide some intuitive justification for the appearance of the notion of a strong Heegaard invariant. At the very least, the notion is enough to ensure such invariants fit into a transitive system. In particular, applying Corollary \ref{StrongImpliesNoMonodromy} to the strong Heegaard invariants
$$CF^{\circ}: \mathcal{G}_{\text{man}} \rightarrow \text{Trans}(P(\text{Kom}(\mathbb{Z}[U]\text{-}\text{Mod})))$$
of Theorem \ref{StrongHeegaard2} immediately yields Corollary \ref{TransitiveSystemResult}. We now show that this transitivity is also enough for the functoriality ends we seek in Theorem \ref{Functoriality}.

\begin{proof}[Proof of Theorem \ref{Functoriality}]
	Assuming Corollary \ref{StrongHeegaard}, the Heegaard Floer invariants $$HF^{\circ}: \mathcal{G}_{\text{man}} \rightarrow P(\mathbb{Z}[U]\text{-}\text{Mod})$$
	are strong Heegaard invariants. Let $\textbf{Man}_{*}$ be the category of closed, connected, oriented, and based $3$-manifolds with based diffeomorphisms. Using the strong Heegaard invariants above, we can obtain functors:
	
	$$HF_{1}^{\circ}: \textbf{Man}_{*} \rightarrow \text{Trans}(P(\mathbb{Z}[U]\text{-}\text{Mod}))$$
	as follows. Given a manifold $(Y, z) \in \text{Ob}(\textbf{Man}_{*})$, Corollary \ref{StrongImpliesNoMonodromy} ensures that the modules $HF^{\circ}(H)$ for isotopy diagrams $H \in |(\mathcal{G}_{\text{man}})_{(Y,z)}|$, along with the isomorphisms $HF^{\circ}_{H, H'}$, form a transitive system. We denote this transitive system by $HF_{1}^{\circ}(Y,z) \in \text{Trans}(P(\mathbb{Z}[U]\text{-}\text{Mod}))$.

	To a pointed diffeomorphism  $\phi: (Y,z) \rightarrow (Y', z')$, the functor $HF_{1}^{\circ}$ will assign a morphism of transitive systems
	$$HF_{1}^{\circ}(\phi): HF_{1}^{\circ}(Y,z) \rightarrow  HF_{1}^{\circ}(Y',z')$$
	defined as follows. Given any isotopy diagram $H=(\Sigma, A, B,z)$ for $(Y,z)$, let $\phi_{H} = \phi |_{\Sigma}$ and $H'$ be the isotopy diagram $\phi(H)$ for $(Y',z')$. By virtue of being a strong Heegaard invariant, $HF^{\circ}$ associates a morphism $HF^{\circ}(\phi_{H}): HF^{\circ}(H) \rightarrow HF^{\circ}(H')$ in $P(\mathbb{Z}[U]\text{-}\text{Mod})$ to any such diffeomorphism of isotopy diagrams $\phi_{H}$. The collection of morphisms $\{\phi_{H}\}$ for $H \in |(\mathcal{G}_{\text{man}})_{(Y,z)}|$ will thus yield a collection of morphisms $\{HF^{\circ}(\phi_{H})\}$. We claim that this collection of morphisms is in fact a morphism of transitive systems 
	$$HF_{1}^{\circ}(\phi): HF_{1}^{\circ}(Y,z) \rightarrow  HF_{1}^{\circ}(Y',z')$$
	as desired. According to Definition \ref{MapOfTransitiveSystems}, we must check that for any path of edges $\gamma$ in $(\mathcal{G}_{\text{man}})_{(Y,z)}$ from $H_{1}$ to $H_{2}$, we have $HF^{\circ}(\phi_{H_{2}}) \circ HF^{\circ}(\gamma) = HF^{\circ}(\gamma ') \circ HF^{\circ}(\phi_{H_{1}})$, for some path $\gamma '$  in $(\mathcal{G}_{\text{man}})_{(Y',z')}$ from $H_{1}'$ to $H_{2}'$. If $\gamma$ is given by the path of edges 
	\begin{equation*}
	\begin{tikzcd}
	D_{0} \arrow{r}{e_{1}}& D_{1} \arrow {r}{e_{2}} & \cdots \arrow{r}{e_{n-1}} & D_{n-1} \arrow{r}{e_{n}} & D_{n} \\
	\end{tikzcd} 
	\end{equation*}
	in $(\mathcal{G}_{\text{man}})_{(Y,z)}$ from $D_{0} = H_{1}$ to $D_{n} = H_{2}$, we pick out a path $\gamma '$ in $(\mathcal{G}_{\text{man}})_{(Y',z')}$  from $H_{1}'$ to $H_{2} '$ given by
	\begin{equation*}
	\begin{tikzcd}
	D_{0}' \arrow{r}{e_{1}'}& D_{1}' \arrow {r}{e_{2}'} & \cdots \arrow{r}{e_{n-1}'} & D_{n-1} \arrow{r}{e_{n}'} & D_{n}' \\
	\end{tikzcd} 
	\end{equation*}
	as follows. We define the intermediate isotopy diagrams in the path $\gamma '$ by $D_{i} ' = \phi(D_{i})$. If the edge $e_{i}$ is given by a strong $\balpha$-equivalence, a strong $\bbeta$-equivalence, or a (de)stabilization, we let $e_{i'}$ denote the corresponding strong $\balpha$-equivalence, strong $\bbeta$-equivalence, or (de)stabilization. If $e_{i}$ corresponds to a diffeomorphism $e_{i}: D_{i-1} \rightarrow D_{i}$ isotopic to the identity, we set $e_{i}' = \phi_{D_{i}} \circ e_{i} \circ \phi_{D_{i-1}}^{-1}$. We then have a subgraph in $\mathcal{G}_{\text{man}}$ given by
	
	\begin{equation*}
	\begin{tikzcd}
	D_{0} \arrow{r}{e_{1}} \arrow{d}{\phi_{H_{1}}}& D_{1} \arrow {r}{e_{2}} \arrow{d}{\phi_{D_{1}}} & \cdots \arrow{r}{e_{n-1}} &  D_{n-1} \arrow{r}{e_{n}} \arrow{d}{\phi_{D_{n-1}}} & D_{n} \arrow{d}{\phi_{H_{2}}} \\
	D_{0}' \arrow{r}{e_{1}'}& D_{1}' \arrow {r}{e_{2}'} & \cdots \arrow{r}{e_{n-1}'} &  D_{n-1} \arrow{r}{e_{n}'} & D_{n}' \\
	\end{tikzcd} 
	\end{equation*}
	The condition that needs to be verified is that the image under $HF^{\circ}$ of the outer rectangle in this subgraph commutes. By construction of the path $\gamma'$, each small square in the diagram is either a distinguished rectangle (recall Definition \ref{StrongHeegaardInvariant}) or a commuting square of diffeomorphisms. Commutativity of the large rectangle now follows by virtue of $HF^{\circ}$ being a strong Heegaard invariant. Since the restriction of $HF^{\circ}$ to $\mathcal{G}_{\text{man}}^{\text{diff}}$ is a functor, the image under $HF^{\circ}$ of the commuting square of diffeomorphisms also commutes. Since the image under $HF^{\circ}$ of any distinguished rectangle also commutes, we thus see that the morphism of transitive systems
	$$HF_{1}^{\circ}(\phi): HF_{1}^{\circ}(Y,z) \rightarrow  HF_{1}^{\circ}(Y',z')$$
	associated to a pointed diffeomorphism $\phi$ is well defined. 
	
	The assignments above thus define the functor $HF_{1}^{\circ}$; we note that composition of morphisms in $\textbf{Man}_{*}$ are respected under $HF_{1}^{\circ}$ because $HF^{\circ}$ is a strong Heegaard invariant, and in particular must be a functor when restricted to $\mathcal{G}_{\text{man}}^{\text{diff}}$ (see Axiom 1 in Definition \ref{StrongHeegaard}).
	
	Finally, we note that isotopic diffeomorphisms in $\textbf{Man}_{*}$ induce identical maps under $HF_{1}^{\circ}$. To see this, suppose $\phi: (Y,z) \rightarrow (Y, z)$ is isotopic to $\text{Id}_{(Y,z)}$, and fix an isotopy diagram $H=(\Sigma, A, B,z)$ for $(Y,z)$. Then $\phi_{H}=\phi |_{H}$ is isotopic to $\text{Id}_{H}$ and $H' = \phi(H) = H$, so by virtue of $HF^{\circ}$ being a strong Heegaard invariant we must have $HF^{\circ}(\phi_{H}) = \text{Id}_{HF^{\circ}(H)}$. Thus $HF_{1}^{\circ}(\phi)$ is the map of transitive systems defined by the data $\{HF^{\circ}(\phi_{H}) = \text{Id}_{HF^{\circ}(H)}\}$ for $H \in (\mathcal{G}_{\text{man}})_{(Y,z)}$, and is thus an identity morphism in $\text{Trans}(P(\mathbb{Z}[U]\text{-}\text{Mod}))$.

\end{proof}

\section{Heegaard Floer Homology as a Weak Heegaard Invariant}\label{WeakHeegaardFloerInvariantsSection}

In this section we very briefly recall numerous maps defined on the Heegaard Floer chain complexes, and then use these maps to define the underlying morphisms of graphs of the strong Heegaard invariants appearing in Theorem \ref{StrongHeegaard2}. For the most part we just seek to establish notation in Sections \ref{SpinStructuressection} - \ref{trianglemapssection}, and refer the reader to \cite{Disks1}, \cite{Cylindrical} and \cite{Naturality} for detailed descriptions of the constructions involved in the definitions appearing there. 

For concreteness and ease of notaton, we will phrase the results in this section in terms of $CF^{-}$, however we note that the definitions vary in a cosmetic way, and analogous results hold, for all of the variants $CF^{\circ}$. In particular, the proof of Theorem \ref{StrongHeegaard2} for $CF^{\circ}$ will follow by the same arguments given here for $CF^{-}$. In fact, one could also obtain the results for the other variants directly from those we prove, as $\widehat{CF}, CF^{+}$ and $CF^{\infty}$ can all be obtained by taking suitable tensor products with $CF^{-}$ and quotients thereof.

Finally, we note at the outset that we will use $\sim$ to indicate homotopic chain maps.

\subsection{$\text{Spin}^{\text{c}}$ Structures and Strong Admissibility}
\label{SpinStructuressection}
We must first address the fact that while the graph $\mathcal{G}_{\text{man}}$ that we have been considering thus far contains arbitrary Heegaard diagrams, the Heegaard Floer chain complexes defined in \cite{Disks1} are defined only with respect to certain \emph{admissible} diagrams. Since we will focus on the case of $CF^{-}$ in this section, the admissibility we will need is given by the notion of \emph{strong admissibility}, which we now summarize.

We begin by recalling the setting of Heegaard Floer homology, and the role of $\text{Spin}^{c}$ structures in the construction of the Heegaard Floer chain complexes. Given a genus $g$ based Heegaard diagram $$\mathcal{H} = (\Sigma, \balpha = (\alpha_{1}, \alpha_{2}, \ldots, \alpha_{g}), \bbeta = (\beta_{1}, \beta_{2}, \ldots, \beta_{g}), z)$$ for a closed, connected, oriented and based $3$-manifold $(Y,z)$, one considers the tori
$$\mathbb{T}_{\balpha} = \alpha_{1} \times \alpha_{2} \times \cdots \times \alpha_{g},\hspace{.2cm} \mathbb{T}_{\bbeta} = \beta_{1} \times \beta_{2} \times \cdots \times \beta_{g}$$
in the symmetric product $\text{Sym}^{g}(\Sigma) := (\Sigma \times \cdots \times \Sigma)/S_{g}$. A choice of complex structure on $\Sigma$ induces an almost complex structure on $\text{Sym}^{g}(\Sigma)$, and with respect to such an induced structure the tori $\mathbb{T}_{\balpha}$ and $\mathbb{T}_{\bbeta}$ are totally real. The Heegaard Floer homology is then defined as a variation of Lagrangian intersection Floer homology applied to these tori. To define the chain complexes one must fix a complex structure $j$ on $\Sigma$, and a choice of generic path $J_{s}$ of almost complex structures on $\text{Sym}^{g}(\Sigma)$ through $\text{Sym}^{g}(j)$ (see \cite{Disks1}). 

The basepoint $z$ induces a map
$$s_{z}: \mathbb{T}_{\balpha} \cap \mathbb{T}_{\bbeta} \rightarrow \text{Spin}^{c}(Y)$$
which associates to each intersection point a $\text{Spin}^{c}$-structure. One first defines a chain complex 
$$CF^{-}(\mathcal{H}, \mathfrak{s})$$ which is freely generated as an abelian group by $[\boldsymbol{x}, i]$, for $\boldsymbol{x} \in \mathbb{T}_{\balpha} \cap \mathbb{T}_{\bbeta}$ with $s_{z}(\boldsymbol{x}) = \mathfrak{s}$ and for $i \in \mathbb{Z}$ with $i <0$. Given two intersection points $\boldsymbol{x}, \boldsymbol{y} \in \mathbb{T}_{\balpha} \cap \mathbb{T}_{\bbeta}$, we let $\pi_{2}(\boldsymbol{x}, \boldsymbol{y})$ denote the set of homtopy classes of Whitney disks connecting $\boldsymbol{x}$ to $\boldsymbol{y}$ in $\text{Sym}^{g}(\Sigma)$, with the usual boundary conditions. Given a homotopy class $\phi \in \pi_{2}(\boldsymbol{x}, \boldsymbol{y})$, we denote by $\mathcal{M}_{J_{s}}(\phi)$ the moduli space of $J_{s}$-holomorphic disks in the class $\phi$, and write $\widehat{\mathcal{M}_{J_{s}}}(\phi) = \mathcal{M}_{J_{s}}(\phi) / \mathbb{R}$ for the quotient with respect to the $\mathbb{R}$-action coming from the translation action on the disks. We let $\mu(\phi)$ denote the Maslov index of the class $\phi$, and let $n_{z}(\phi)$ denote the algebraic intersection number of $\phi$ with $z \times \text{Sym}^{g-1}(\Sigma)$. We then have a well defined relative (in general cyclic) grading on the generators defined above, given by the formula
$$\text{gr}([\boldsymbol{x},i], [\boldsymbol{y},j]) = \mu(\phi) - 2n_{z}(\phi) + 2i - 2j,$$
where $\phi$ is any class $\phi \in \pi_{2}(\boldsymbol{x}, \boldsymbol{y})$. This grading is only integral if $c_{1}(\mathfrak{s})=0$. Finally, the differential 
$$\partial: CF^{-}(\mathcal{H}, \mathfrak{s}) \rightarrow CF^{-}(\mathcal{H}, \mathfrak{s})$$
is defined by the formula
$$ \partial([\boldsymbol{x}, i]) = \sum_{\{\boldsymbol{y} \in \mathbb{T}_{\balpha} \cap \mathbb{T}_{\bbeta} | s_{z}(\boldsymbol{y}) = \mathfrak{s}\} } \sum_{ \{\phi \in \pi_{2}(\boldsymbol{x}, \boldsymbol{y}) | \mu(\phi) =1 \}} \#\widehat{\mathcal{M}_{J_{s}}}(\phi) \cdot [\boldsymbol{y}, i - n_{z}(\phi)].$$
There is an action of the polynomial ring $\mathbb{Z}[U]$ on the complex $CF^{-}(\mathcal{H}, \mathfrak{s})$, where
$$ U \cdot [\boldsymbol{x}, i] = [\boldsymbol{x}, i-1]$$
decreases the relative grading by $2$. We will always consider $CF^{-}(\mathcal{H}, \mathfrak{s})$ as a complex of $\mathbb{Z}[U]$-modules. Finally, the total chain complex associated to $\mathcal{H}$ then splits by definition as 
$$CF^{-}(\mathcal{H}) = \bigoplus_{\mathfrak{s} \in \text{Spin}^{c}(Y)} CF^{-}(\mathcal{H}, \mathfrak{s}).$$

Given a $\text{Spin}^{c}$ structure $\mathfrak{s}$, we call a pointed Heegaard diagram \emph{$\mathfrak{s}$-realized} if there is an intersection point $\boldsymbol{x} \in \mathbb{T}_{\balpha} \cap \mathbb{T}_{\bbeta}$ with $s_{z}(\boldsymbol{x}) = \mathfrak{s}$. We note that for any $\mathfrak{s} \in \text{Spin}^{c}(Y,z)$ there is an $\mathfrak{s}$-realized pointed Heegaard diagram for $(Y,z)$ by \cite[Lemma 5.2]{Disks1}.

The chain complex $CF^{-}(\mathcal{H}, \mathfrak{s})$ can in fact only be defined for Heegaard diagrams $\mathcal{H} = (\Sigma, \balpha, \bbeta, z)$ which satisfy an admissibility hypothesis. Given $\mathfrak{s} \in \text{Spin}^{c}(Y)$, we say the diagram $\mathcal{H}$ is \emph{strongly $\mathfrak{s}$-admissible} if every nontrivial periodic domain $D$ on $\mathcal{H}$ satisfying $\langle c_{1}(\mathfrak{s}), H(D) \rangle = 2n \geq 0$ has some coefficient that is greater than $n$. Here $H(D) \in H_{2}(Y; \mathbb{Z})$ is the homology class naturally associated to the periodic domain $D$. It turns out that this notion of admissibility is enough to ensure that differential $\partial$ given above consists of a finite sum and is well defined on $CF^{-}(\mathcal{H}, \mathfrak{s})$, and to ensure that it in fact yields a chain complex. It is shown in \cite[Lemma 5.4]{Disks1} that given any $\mathfrak{s} \in \text{Spin}^{c}(Y)$, there is an $\mathfrak{s}$-realized, strongly $\mathfrak{s}$-admissible pointed diagram for $(Y,z)$. 

To define triangle maps on the Floer chain complexes, we will need an analogous notion of admissibility for Heegaard triple diagrams. A pointed triple diagram $\mathcal{T} = (\Sigma, \balpha, \bbeta, \bgamma, z)$ specifies a $4$-manifold with boundary, which we denote by $X_{\balpha, \bbeta, \bgamma}$. Given now a $\text{Spin}^{c}$-structure $\mathfrak{s}$ on $X_{\balpha, \bbeta, \bgamma}$, denote by $\mathfrak{s}_{\balpha, \bbeta}$ the restriction of $\mathfrak{s}$ to the boundary component $Y_{\balpha, \bbeta}$. We will say the triple diagram $\mathcal{T}$ is \emph{strongly $\mathfrak{s}$-admissible} if any triply periodic domain $D$ which is the sum of doubly periodic domains, 
$$D = D_{\balpha, \bbeta} + D_{\bbeta, \bgamma} + D_{\balpha, \bgamma}$$
and which furthermore satisfies 
$$\langle c_{1}(\mathfrak{s}_{\balpha, \bbeta}), H(D_{\balpha, \bbeta}) \rangle + \langle c_{1}(\mathfrak{s}_{\bbeta, \bgamma}), H(D_{\bbeta, \bgamma}) \rangle + \langle c_{1}(\mathfrak{s}_{\balpha, \bgamma}), H(D_{\balpha, \bgamma}) \rangle = 2n \geq 0 $$
has some coefficient greaer than $n$. It is shown in \cite[Lemma 8.11]{Disks1} that given any pointed triple diagram $\mathcal{T}$ and a $\text{Spin}^{c}$ structure $\mathfrak{s}$ on $X_{\balpha, \bbeta, \bgamma}$, there is a pointed triple diagram isotopic to $\mathcal{T}$ which is strongly $\mathfrak{s}$-admissible.

\subsection{Orientation Systems}
\subsubsection{Coherent Orientation Systems of Disks}
We recall that to define the differential on the Heegaard Floer chain complexes with coefficients in $\mathbb{Z}$, one must perform signed counts of the points in certain moduli spaces of psuedo-holomorphic disks. To do so, one must ensure that on a pointed Heegaard diagram $\mathcal{H} = (\Sigma, \balpha,\bbeta,z)$ the moduli spaces of holomorphic disks in a homotopy class $A \in \pi_{2}(\boldsymbol{x},\boldsymbol{y})$, which we denote by $\mathcal{M}^{A}$ or $\mathcal{M}(A)$, are orientable. By \cite[Proposition 3.10]{Disks1} (or \cite[Proposition 6.3]{Cylindrical} for the reader more comfortable in the cylindrical setting), these moduli spaces are orientable whenever they are smoothly cut out. There this is shown by trivializing the determinant line bundle $\mathcal{L}$ of the virtual index bundle of the linearized $\bar{\partial }$-equation defining the moduli space in question, so when necessary we will specify our orientations by specifying sections of these determinant line bundles.

In order for these orientations to allow for the structure of a chain complex on the Heegaard Floer chain modules, we actually need somewhat more: we want the moduli spaces for different homotopy classes of disks to be oriented coherently. To make this precise, Ozsv\'ath and Szab\'o used the notion of a \textit{coherent orientation system} for the moduli spaces of holomorphic disks in a Heegaard diagram $\mathcal{H} = (\Sigma, \balpha, \bbeta, z)$. Such an orientation system consists of a collection $\mathfrak{o}_{\mathcal{H}} = \mathfrak{o}_{\balpha, \bbeta} := \{ \mathfrak{o}_{\balpha, \bbeta}^{A}\}$ of sections $\mathfrak{o}_{\balpha, \bbeta}^{A}$ of the determinant line bundle $\mathcal{L}$ over all possible homotopy classes of disks $A \in \pi_{2}(\boldsymbol{x},\boldsymbol{y})$ (ranging over all $\boldsymbol{x}, \boldsymbol{y} \in \mathbb{T}_{\balpha} \cap \mathbb{T}_{\bbeta})$. Roughly, the coherence condition amounts to requiring that these sections are compatible with a process of glueing holomorphic disks together. We refer the reader to \cite{Disks1} for the precise definition of the coherence condition, or to Section \ref{Orientations} where we will formulate a precise version of the notion in the cylindrical setting. For our purposes in this section, we mainly just want to recall the fact that every pointed Heegaard diagram equipped with complex structure data achieving transversality admits a coherent orientation system by \cite[Remarks following Definition 3.12]{Disks1}. We also want to make explicit the folowing equivalence relation on orientation systems.
\begin{defn}\label{orientation equivlance}
	Fix two coherent orientation systems $\mathfrak{o}_{\balpha, \bbeta}$ and $\mathfrak{o}'_{\balpha, \bbeta}$ on a diagram $\mathcal{H} = (\Sigma, \balpha, \bbeta, z)$. We say the orientation systems are \emph{equivalent} if there is a function $$\epsilon: \mathbb{T}_{\balpha} \cap \mathbb{T}_{\bbeta} \rightarrow \{\pm 1\}$$ so that for each $\bx,\by \in \mathbb{T}_{\balpha} \cap \mathbb{T}_{\bbeta}$ we have $$\mathfrak{o}_{\balpha, \bbeta}^{A} = \epsilon(\bx) \cdot \epsilon(\by) \cdot {\mathfrak{o}'}_{\balpha, \bbeta}^{A}$$ for all $A \in \mathcal{M}(\bx, \by)$.
\end{defn}

It follows directly from the definition of the differential on $CF^{-}$ that equivalent orientation systems give rise to isomorphic Heegaard Floer chain complexes. In what follows, we will often be concerned with specifying orientation systems which are unique up to equivalence. For these discussions, it will be useful to explicitly recall one more definition from the literature.

\begin{defn}\cite[Definition 3.12]{Disks1}
	Given a $\text{Spin}^{c}$ structure $\mathfrak{s}$, a strongly $\mathfrak{s}$-admissible diagram $\mathcal{H} = (\Sigma, \balpha, \bbeta, z)$, and an intersection point $\bx_{0} \in \mathbb{T}_{\balpha} \cap \mathbb{T}_{\bbeta}$, we will say a collection of classes $\{A_{\by}\}$ where $A_{\by} \in \pi_{2}(\bx_{0}, \by)$ and $\by$ ranges over the intersection points in $(\mathbb{T}_{\balpha} \cap \mathbb{T}_{\bbeta}) \setminus \{\bx_{0}\}$ which represent $\mathfrak{s}$, is a \emph{complete set of paths (based at $\bx_{0}$)} for $(\mathcal{H}, \mathfrak{s})$.
\end{defn}

\subsubsection{Coherent Orientation Systems of Triangles}
\label{coherentsection}
Given a pointed Heegaard triple diagram $\mathcal{T} = (\Sigma, \balpha, \bbeta, \bgamma, z)$, we also note that moduli spaces of holomorphic triangles in a homotopy class $\psi$, which we denote by $\mathcal{M}^{\psi}$ or $\mathcal{M}(\psi)$, are also orientable when they are smoothly cut out, by \cite[Section 8.2]{Disks1} (or \cite[Proposition 10.3]{Cylindrical}). Given a collection $\mathfrak{o}_{\mathcal{T}} := \{\mathfrak{o}_{\balpha, \bbeta, \bgamma}, \mathfrak{o}_{\balpha, \bbeta}, \mathfrak{o}_{\bbeta, \bgamma}, \mathfrak{o}_{\balpha, \bgamma}\}$, where $\mathfrak{o}_{\balpha, \bbeta, \bgamma}$ is a collection of sections of the determinant line bundle over all homotopy classes of triangles, and $\mathfrak{o}_{\balpha, \bbeta}$, $\mathfrak{o}_{\bbeta, \bgamma}$, and $\mathfrak{o}_{\balpha, \bgamma}$ are collections of sections of the determinant line bundle over all homotopy classes of disks in the respective double diagrams, we will consider a related notion of coherence (see \cite[Definition 8.6]{Disks1}). Roughly, the coherence condition here will amount to the requirement that each collection of orientations of the moduli spaces of strips on the respective double diagrams are coherent, and that all possible pregluings of triangles with strips satisfy the analogous glueing condition (this coherence condition will also be spelled out precisely in Section \ref{Orientations}). The existence of such coherent orientation systems is guaranteed by the following result.

\begin{claim}\cite[Lemma 8.7]{Disks1}
	\label{ExistenceOfCoherentTriangles}
	Fix a pointed Heegaard triple diagram $(\Sigma, \balpha, \bbeta, \bgamma, z)$, and let $\mathfrak{s}$ be a $\text{Spin}^{c}$ structure on $X_{\balpha, \bbeta, \bgamma}$ whose restriction to each boundary component is realized by an intersection point in the corresponding Heegaard diagram. Then for any coherent orientation systems $\mathfrak{o}_{\balpha, \bbeta}$ and $\mathfrak{o}_{\bbeta, \bgamma}$ for two of the boundary components, there exists at least one coherent orientation system $\mathfrak{o}_{\balpha, \bgamma}$ for the remaining boundary component and a coherent orientation system $\mathfrak{o}_{\balpha, \bbeta, \bgamma}$ such that the entire collection of orientations is coherent.
\end{claim}

\begin{remark}\label{uniquenessremark}
	
	We note here that this lemma does not guarantee that the orientation systems $\mathfrak{o}_{\balpha, \bgamma}$ and $\mathfrak{o}_{\balpha, \bbeta, \bgamma}$ are unique, as can be seen from inspection of the proof provided in \cite[Lemma 8.7]{Disks1}. We mainly provide the reference to this lemma as it is stated for background context on the existence of coherent orientation systems. In what follows we will actually be interested in using a strengthened version of this lemma that applies in a particular situation to produce a unique induced coherent orientation system, which we will specify more precisely when the time comes. We note in particular that we only cite Lemma \ref{ExistenceOfCoherentTriangles} in two places in this paper (in Sections \ref{trianglemapssection} and \ref{monodromyoforientationsystems}), and in both cases an additional argument is used to explain why the induced orientation system is unique in the context under consideration.
	
	It will be useful later to have a clear understanding of the indeterminacy in the orientation systems furnished by this lemma, and to have terminology with which we can refer to the sources of indeterminacy. To do so, we will now describe a high level outline of the proof of the above lemma, and point out explicitly where in the proof the indeterminacies arise. For details of the proof, we just point to the original source, since we have no new perspectives or value to add in reproducing them.

	Assume we have fixed $\mathfrak{o}_{\balpha, \bbeta}$ and $\mathfrak{o}_{\bbeta, \bgamma}$ as in the statement of the lemma. The way to produce $\mathfrak{o}_{\balpha, \bgamma}$ and the coherent orientation system $\mathfrak{o}_{\balpha, \bbeta, \bgamma}$ on the triple diagram guaranteed by the lemma can be summarized as follows:
	\begin{enumerate}
		\item First choose an arbitrary orientation over a single class of triangle $\psi_{0} \in \pi_{2}(\bx_{0}, \by_{0}, \bz_{0})$ connecting intersection points $\bx_{0}, \by_{0}$ and $\bz_{0}$.
		\item Next, fix orientations over all periodic classes $\phi_{\balpha, \bgamma}  \in \Pi_{\bz_{0}} \subset \pi_{2}(\bz_{0}, \bz_{0})$ as follows:
		\begin{enumerate}
			\item Define a subgroup $K \subset \Pi_{\bz_{0}}$ by 
			\begin{equation*}
			\begin{aligned}
			K = \{\phi_{\balpha, \bgamma} \in  \pi_{2}(\bz_{0}, \bz_{0}) \hspace{1em} | \hspace{.5em} \exists \hspace{.5em}  &\phi_{\balpha, \bbeta} \in  \pi_{2}(\bx_{0}, \bx_{0}),\\ &  \phi_{\bbeta, \bgamma} \in  \pi_{2}(\by_{0}, \by_{0})\\ &\text{ such that} \\ &\psi_{0} + \phi_{\balpha, \bgamma} = \psi_{0} + \phi_{\balpha, \bbeta} + \phi_{\bbeta, \bgamma} \}
			\end{aligned}
			\end{equation*}
			\item Show that the periodic classes split as:
			$$\Pi_{\bz_{0}} = K \oplus Q$$
			for some free abelian group Q.
			\item Using the defining property of K and a small lemma, extend $\mathfrak{o}_{\balpha, \bgamma}$ over all periodic classes in K such that the resulting orientations are consistent with $\mathfrak{o}_{\balpha, \bbeta}$ and $\mathfrak{o}_{\bbeta, \bgamma}$.
			\item Choose the orienations $\mathfrak{o}_{\balpha, \bgamma}$ arbitrarily over basis a for Q. We will call this collection of choices the \emph{indeterminacy over Q}.
			\item Obtain orientations over all classes of triangles $\psi \in \pi_{2}(\bx_{0}, \by_{0}, \bz_{0})$ by bootstrapping from the above.
		\end{enumerate}
		
		\item Next, choose a complete set of paths for $Y_{\balpha, \bgamma}$, and choose orientations for $\mathfrak{o}_{\balpha, \bgamma}$ over the classes defining the complete set of paths. We will call this collection of choices the \emph{indeterminacy over a complete set of paths}.
		\item The previously defined orientations together uniquely determine a coherent orientation system for the triple diagram. We see that, up to a sign, the indeterminacy in the orientation systems $\mathfrak{o}_{\balpha, \bgamma}$ and $\mathfrak{o}_{\balpha, \bbeta, \bgamma}$ furnished by the lemma is due to the indeterminacy over $Q$ and the indeterminacy over a complete set of paths.
	\end{enumerate}
	Finally, we note that the indeterminacy over a complete set of paths mentioned above does in fact vanish in general, so long as we we consider orientation systems up to equivalence. For given a complete set of paths and two orientation systems $\mathfrak{o}_{\balpha, \bgamma}$, $\mathfrak{o}^{'}_{\balpha, \bgamma}$ which differ on the complete set of paths, it is straightforward to construct a third orientation system $\mathfrak{o}''_{\balpha, \bgamma}$ which is equivalent to $\mathfrak{o}^{'}_{\balpha, \bgamma}$ and which agrees with $\mathfrak{o}_{\balpha, \bgamma}$ on the complete set of paths. Indeed, if the complete set of paths is denoted $\{A_{\by}\}$, construct an equivlance function $\epsilon$ by declaring for each $\by$ that $\epsilon(\by) = -1$ if $\mathfrak{o}_{\balpha, \bgamma}(A_{\by}) \neq \mathfrak{o}^{'}_{\balpha, \bgamma}(A_{\by})$. Altering $\mathfrak{o}^{'}_{\balpha, \bgamma}$ by any such equivalence function will yield an orientation system $\mathfrak{o}''_{\balpha, \bgamma}$ as desired. Thus up to equivalence and sign, we see that the indeterminacy in the orientation systems $\mathfrak{o}_{\balpha, \bgamma}$ and $\mathfrak{o}_{\balpha, \bbeta, \bgamma}$ furnished by the lemma is solely due to the indeterminacy over $Q$ (coming from the indeterminacy in the orientations over the periodic classes).
\end{remark}

\subsection{Change of Almost Complex Structures}
Next, we recall the dependence of the construction of the Heegaard Floer invariants on the choices of almost complex structures involved. The definition of the Heegaard Floer chain complex associated to a pointed Heegaard diagram $(\Sigma, \balpha, \bbeta, z)$ in fact requires a choice of complex structure $j$ on $\Sigma$, and a generic path of almost complex structures $J_{s} \subset \mathcal{U}$ on $\text{Sym}^{g}(\Sigma)$ going through the structure $\text{Sym}^{g}(j)$ induced by $j$. Here $g$ is the genus of $\Sigma$ and $\mathcal{U}$ is a particular contractible set of almost complex structures specified by Ozsv\'ath and Szab\'o in \cite[Theorem 3.15 and Section 4.1]{Disks1}. Given a strongly $\mathfrak{s}$-admissible pointed Heegaard diagram $\mathcal{H} = (\Sigma, \balpha, \bbeta,z)$, a coherent orientation $\mathfrak{o}$ on $\mathcal{H}$, and two choices of such almost complex structure data $(j,J_{s})$ and $(j',J_{s}')$, there is a chain homotopy equivalence
$$ \Phi_{J_{s} \rightarrow J_{s}'}: CF^{-}_{J_{s}}(\Sigma, \balpha, \bbeta, z, \mathfrak{s}, \mathfrak{o}) \rightarrow CF^{-}_{J_{s}'}(\Sigma, \balpha, \bbeta, z, \mathfrak{s}, \mathfrak{o}').$$
Here $\mathfrak{o}'$ is an orientation system uniquely determined  by $\mathfrak{o}$, as described in \cite[Beginning of Section 9]{Cylindrical}. These equivalences fit into a transitive system in the homotopy category of chain complexes of $\mathbb{Z}[U]$-modules, in the sense that $\Phi_{J_{s} \rightarrow J_{s}} \sim \text{id}_{CF^{-}(\Sigma, \balpha, \bbeta)}$ and $\Phi_{J_{s}' \rightarrow J_{s}''} \circ \Phi_{J_{s} \rightarrow J_{s}'} \sim \Phi_{J_{s} \rightarrow J_{s}''}$. This is shown in \cite[Lemma 2.11]{FourManifoldInvariants}. We denote this transitive system in the homotopy category of complexes of $\mathbb{Z}[U]$-modules by
$$CF^{-}(\Sigma, \balpha, \bbeta, z, \mathfrak{s}, \mathfrak{o}).$$

Of course we also obtain from the maps $\Phi_{J_{s} \rightarrow J_{s}'}$ a transitive system of isomorphisms on homology. We will denote the colimit of the $\mathbb{Z}[U]$-modules $HF^{-}_{J_{s}}(\Sigma, \balpha, \bbeta, z, \mathfrak{s}, \mathfrak{o})$ with respect to this transitive system by $$HF^{-}(\Sigma, \balpha, \bbeta, z, \mathfrak{s}, \mathfrak{o}).$$ 


\subsection{Triangle Maps and Continuation Maps}
\label{trianglemapssection}
Given a pointed Heegaard triple diagram $\mathcal{T} = (\Sigma, \balpha, \bbeta, \bgamma, z)$ which is strongly $\mathfrak{s}$-admissible for a $\text{Spin}^{c}$ structure $\mathfrak{s}$ on $X_{\balpha, \bbeta, \bgamma}$, as well as a coherent orientation system $\mathfrak{o}_{\balpha, \bbeta, \bgamma}$ compatible with coherent orientation systems $\mathfrak{o}_{\balpha, \bbeta}$, $\mathfrak{o}_{\bbeta, \bgamma}$ and $\mathfrak{o}_{\balpha, \bgamma}$, there are $\mathbb{Z}[U]$-module chain maps $\mathcal{F}_{\balpha,\bbeta,\bgamma}(\cdot, \mathfrak{s}, \mathfrak{o}_{\balpha, \bbeta, \bgamma})$ of the form:
$$ CF^{-}_{J_{s}}(\Sigma,\balpha,\bbeta, \mathfrak{s}_{\balpha, \bbeta}, \mathfrak{o}_{\balpha, \bbeta}) \otimes_{\mathbb{Z}[U]} CF^{-}_{J_{s}}(\Sigma,\bbeta,\bgamma, \mathfrak{s}_{\bbeta, \bgamma}, \mathfrak{o}_{\bbeta, \bgamma}) \rightarrow CF^{-}_{J_{s}}(\Sigma,\balpha,\bgamma, \mathfrak{s}_{\balpha, \bgamma}, \mathfrak{o}_{\balpha,\bgamma})$$
defined in \cite[Theorem 8.12]{Disks1}. Here, and throughout this section, we sometimes suppress the basepoint $z$ from the chain complex notation for brevity, but the dependence is always implied. Put simply, these chain maps count pseudoholomorpic triangles on the triple diagram. In fact, the homotopy class of the chain map $F_{\balpha, \bbeta, \bgamma}$ does not depend on the choice of almost complex structure data. More precisely, for two choices of almost complex structure data the maps above commute up to homotopy with the change of almost complex structure maps, by \cite[Proposition 8.13]{Disks1}. Thus with respect to the transitive systems $CF^{-}(\Sigma, \balpha, \bbeta, z, \mathfrak{s}, \mathfrak{o})$, the map $F_{\balpha, \bbeta, \bgamma}$ is a morphism in $\text{Trans}(\text{Kom}(\mathbb{Z}[U]\text{-}\text{Mod}))$, ie a morphism between two transitive systems in the homotopy category of $\mathbb{Z}[U]$ modules. We denote this morphism by $\mathcal{F}_{\balpha,\bbeta,\bgamma}(\cdot, \mathfrak{s}, \mathfrak{o}_{\balpha, \bbeta, \bgamma})$ and it takes the form: 
$$CF^{-}(\Sigma,\balpha,\bbeta, \mathfrak{s}_{\balpha, \bbeta}, \mathfrak{o}_{\balpha, \bbeta}) \otimes_{\mathbb{Z}[U]} CF^{-}(\Sigma,\bbeta,\bgamma, \mathfrak{s}_{\bbeta, \bgamma}, \mathfrak{o}_{\bbeta, \bgamma}) \rightarrow CF^{-}(\Sigma,\balpha,\bgamma, \mathfrak{s}_{\balpha, \bgamma}, \mathfrak{o}_{\balpha,\bgamma})$$
We also obtain induced maps of $\mathbb{Z}[U]$-modules $\mathcal{F}_{\balpha,\bbeta,\bgamma}(\cdot, \mathfrak{s}, \mathfrak{o}_{\balpha, \bbeta, \bgamma})$ of the form: 
$$HF^{-}(\Sigma,\balpha,\bbeta, \mathfrak{s}_{\balpha, \bbeta}, \mathfrak{o}_{\balpha, \bbeta}) \otimes_{\mathbb{Z}[U]} HF^{-}(\Sigma,\bbeta,\bgamma, \mathfrak{s}_{\bbeta, \bgamma}, \mathfrak{o}_{\bbeta, \bgamma}) \rightarrow HF^{-}(\Sigma,\balpha,\bgamma,  \mathfrak{s}_{\balpha, \bgamma}, \mathfrak{o}_{\balpha,\bgamma})$$
The triangle maps above allow one to define maps associated to handleslides. To describe the handleslide maps, we first recall the following fact.

\begin{claim}\cite[Lemma 9.4, Remark 9.2 and Section 9.1]{Disks1} (cf. \cite[Lemma 9.2]{Naturality})
	\label{StandardHandleslide}
	Let $(\Sigma, \bbeta, \bgamma', z)$ be a pointed genus $g$ Heegaard diagram such that $\bgamma'$ can be obtained from $\bbeta$ by performing a sequence of handleslides among the curves in $\bbeta$. Then the diagram represents $\#^{g}(S^{1} \times S^{2})$. There is a unique $\text{Spin}^{c}$ structure $\mathfrak{s}_{0} \in \text{Spin}^{c}(\#^{g}(S^{1} \times S^{2}))$ such that $c_{1}(\mathfrak{s}_{0}) = 0$, and upon performing a particular small Hamiltonian isotopy of $\bgamma '$ (specified in \cite{Disks1}) to obtain $(\Sigma, \bbeta, \bgamma, z)$ one can ensure this new diagram is strongly $\mathfrak{s}_{0}$-admissible. Furthermore, there is a choice of coherent orientation system $\mathfrak{o}_{\bbeta, \bgamma}$ on this diagram such that
	$$\widehat{HF}(\Sigma, \bbeta, \bgamma, z,  \mathfrak{s}_{0}, \mathfrak{o}_{\bbeta, \bgamma}) \cong H_{*}(T^{g}; \mathbb{Z}) $$
	$$HF^{-}(\Sigma, \bbeta, \bgamma, z,  \mathfrak{s}_{0}, \mathfrak{o}_{\bbeta, \bgamma}) \cong \mathbb{Z}[U] \otimes H_{*}(T^{g}; \mathbb{Z}) $$
	In this case it follows that in the highest nontrivial relative homological grading $HF^{-}(\Sigma, \bbeta, \bgamma, z,  \mathfrak{s}_{0}, \mathfrak{o}_{\bbeta, \bgamma})$  is isomorphic to $\mathbb{Z} =: \langle \theta_{\bbeta, \bgamma} \rangle $, for a generator we denote $\theta_{\bbeta, \bgamma}$. Finally, there is only one equivalence class of orientation system with these properties.
\end{claim}

\begin{remark}
	For such a diagram, we can also identify a particular intersection point $\theta_{\bbeta, \bgamma} \in CF^{-}(\Sigma, \bbeta, \bgamma, z, \mathfrak{s}_{0}, \mathfrak{o}_{\bbeta, \bgamma})$ representing this element of homology. Indeed, the strongly admissible diagram referred to in the lemma statement yields a chain complex whose rank is the same as that of its homology, and which has a unique intersection point realizing $\mathfrak{s}_{0}$ in the maximal relative grading.
\end{remark}

\begin{remark}
	All of the statements in the lemma other than the last sentence are explicitly proved in the cited references. The last sentence is also contained implicitly in the references cited, but since it is particularly relevant to our arguments we provide a sketch of the proof below.
\end{remark}

The last sentence in Lemma \ref{StandardHandleslide} follows from the next result.
\begin{claim}
	\label{EquivalenceClassesonS1xS2}
	Equivalence classes of coherent orientation systems over the diagram $(\Sigma, \bbeta, \bgamma, z)$ for $(S^{1} \times S^{2})^{\# g}$ from Lemma \ref{StandardHandleslide} are in bijection with morphisms $\pi_{1}(T^{g}) \rightarrow \text{Aut}(\mathbb{Z})$ where $T^{g}$ is a torus. Furthermore, for a corresponding orientation system $\mathfrak{o}$ and morphism $\mathcal{L}$ we have:
	$$ \widehat{HF}((S^{1} \times S^{2})^{\# g}, \mathfrak{o}) \cong H_{*}(T^{g}; \mathcal{L})$$
\end{claim}
\begin{proof}[Proof Sketch]
	Fix a diagram $(\Sigma, \bbeta, \bgamma, z)$ for $(S^{1} \times S^{2})^{\# g}$ as described in Lemma \ref{StandardHandleslide}, an intersection point $\bx_{0} \in \mathbb{T}_{\bbeta} \cap \mathbb{T}_{\bgamma}$, and a complete set of paths based at $\bx_{0}$. As described in Remark \ref{uniquenessremark}, all coherent orientation systems on the diagram agree on the complete set of paths up to equivalence. Thus equivalence classes of coherent orientation systems are determined by their values on a basis for the periodic domains based at $\bx_{0}$. Note that specifying values in $\{\pm 1\}$ for each class in a basis for the periodic domains based at $\bx_{0}$ is the same as specifying a morphism $\pi_{1}(T^{g}) \rightarrow \text{Aut}(\mathbb{Z})$, since the group of periodic classes is identified with $H^{1}((S^{1} \times S^{2})^{\# g})$. This establishes the first sentence in the lemma. 
	
	The second statement in the lemma follows from a direct comparison of the contributions to homology (Heegaard Floer or singular) in the diagrams in question for a given choice of values over a basis for the periodic domains based at $\bx_{0}$. For example, assigning $1$ to each periodic domain corresponds to the isomorphism class of local system over $T^{g}$ specified by the trivial homomorphism $\mathcal{L}:\pi_{1}(T^{g}) \rightarrow \mathbb{Z}/ 2\mathbb{Z}$, and to some equivalence class of coherent orientation system on $(\Sigma, \bbeta, \bgamma, z)$. Using the local picture and calculations developed in \cite[Lemma 9.4]{Disks1}, one can establish an identification between generators of $\pi_{1}$ and generators of $\widehat{HF}$.
\end{proof}

Now to establish the last sentence in Lemma \ref{StandardHandleslide}, just note that there is a single local system $\mathcal{L}$ over the torus $T^{g}$ for which the singular homology is $H_{*}(T^{g}; \mathbb{Z})$ (namely the trivial local system).

Given a strongly  $\mathfrak{s}$-admissible triple diagram $(\Sigma, \balpha, \bbeta, \bgamma, z)$ with $\bgamma$ related to $\bbeta$ as in the statement of Lemma \ref{StandardHandleslide}, we will write 
$$\Psi_{\bbeta \rightarrow \bgamma}^{\balpha}( \cdot, \mathfrak{s}, \mathfrak{o}_{\balpha, \bbeta, \bgamma}) := F_{\balpha,\bbeta,\bgamma}(\cdot \otimes \theta_{\bbeta, \bgamma}, \mathfrak{s}, \mathfrak{o}_{\balpha, \bbeta, \bgamma})$$
where 
 $$F_{\balpha,\bbeta,\bgamma}(\cdot \otimes \theta_{\bbeta, \bgamma}, \mathfrak{s}, \mathfrak{o}_{\balpha, \bbeta, \bgamma}): CF^{-}(\Sigma,\balpha,\bbeta,z, \mathfrak{s}_{\balpha, \bbeta}, \mathfrak{o}_{\balpha, \bbeta}) \rightarrow CF^{-}(\Sigma, \balpha, \bgamma, z, \mathfrak{s}_{\balpha, \bgamma}, \mathfrak{o}_{\balpha, \bgamma})$$ 
Here we have used an arbitrary coherent orientation system $\mathfrak{o}_{\balpha, \bbeta}$ and the coherent orientation system $\mathfrak{o}_{\bbeta, \bgamma}$ of Lemma \ref{StandardHandleslide}, and enlarged them to a coherent orientation system $\mathfrak{o}_{\balpha, \bbeta, \bgamma}$. That this can be done in some way is ensured by Lemma \ref{ExistenceOfCoherentTriangles}; in fact, though, this enlargement is unique up to equivalence in this particular case, as we now explain. Recall we have seen in Remark \ref{uniquenessremark} that the indeterminacy in the equivalence classes of the orientation systems furnished by Lemma \ref{ExistenceOfCoherentTriangles} is due solely to the indeterminacy over the group $Q$. It is shown in the proof of \cite[Lemma 8.7]{Disks1} that this group $Q$ is the image of the composition $q\circ i$:
\begin{equation*}
\begin{tikzcd}
H_{2}(Y_{\balpha, \bgamma}) \arrow{r}{i} &  H_{2}(X_{\balpha, \bbeta, \bgamma}) \arrow{r}{q} & H_{2}(X_{\balpha, \bbeta, \bgamma}, Y_{\balpha, \bbeta} \cup Y_{\bbeta, \bgamma} ) 
\end{tikzcd} 
\end{equation*}
where $i$ is induced by inclusion and $q$ comes from the relative long exact sequence for the relevant pair. In the case at hand, we have $Y_{\balpha, \bbeta} \cong Y_{\balpha, \bgamma}$ are arbitrary $3$ manifolds, and $X_{\balpha, \bbeta, \bgamma}$ is  $Y_{\balpha, \bbeta} \times I$ with a neighborhood of a bouquet of $g$ circles removed. Thus we have $i(H_{2}(Y_{\balpha, \bbeta})) = i(H_{2}(Y_{\balpha, \bgamma}))$, and $Q = 0$. This establishes that the coherent orientation systems used in our definition of the map $\Psi_{\bbeta \rightarrow \bgamma}^{\balpha}$ above are well defined. Similarly if instead $\bbeta$ is related to $\balpha$ as in the statement of Lemma \ref{StandardHandleslide}, we will write 
$$\Psi_{ \bgamma}^{\balpha \rightarrow \bbeta}(\cdot, \mathfrak{s}, \mathfrak{o}_{\bbeta, \balpha, \bgamma}) := F_{\bbeta, \balpha, \bgamma}( \theta_{\bbeta, \balpha} \otimes \cdot, \mathfrak{s}, \mathfrak{o}_{\bbeta, \balpha, \bgamma})$$
where
$$ F_{\bbeta, \balpha, \bgamma}( \theta_{\bbeta, \balpha} \otimes \cdot, \mathfrak{s}, \mathfrak{o}_{\bbeta, \balpha, \bgamma}): CF^{-}(\Sigma,\balpha,\bgamma, z, \mathfrak{s}_{\balpha, \bgamma}, \mathfrak{o}_{\balpha, \bgamma}) \rightarrow CF^{-}(\Sigma, \bbeta, \bgamma, z, \mathfrak{s}_{\bbeta, \bgamma}, \mathfrak{o}_{\bbeta, \bgamma})$$
These can be thought of as maps on the Floer invariants associated to (small variations of) sequences of handleslides on diagrams. These maps are in fact homotopy equivalences according to the following result:
\begin{claim}\cite[Theorem 9.5 and Section 9.1]{Disks1}
	\label{HandleslideMapProperties}
	\begin{enumerate}
		\item If $(\Sigma, \balpha, \bbeta, \bgamma, z)$ is a strongly $\mathfrak{s}$-admissible triple diagram and $\bbeta$ is related to $\bgamma$ as in the statement of Lemma \ref{StandardHandleslide}, then $\Psi_{\bbeta \rightarrow \bgamma}^{\balpha}$ is a chain homotopy equivalence.
		\item Furthermore, such equivalences are transitive: for two triples satisfying the conditions above we have $$\Psi_{\bbeta \rightarrow \bgamma}^{\balpha} \sim \Psi_{\boldsymbol{\delta} \rightarrow \bgamma }^{\balpha} \circ \Psi_{\bbeta \rightarrow \boldsymbol{\delta}}^{\balpha}.$$
		\item The analogous results hold for the maps induced by changing the $\balpha$ curves.
	\end{enumerate}
	
\end{claim}

There are also maps associated to special Hamiltonian isotopies of diagrams \cite[Proof of Theorem 7.3]{Disks1}. Given strongly $\mathfrak{s}$-admissible diagrams $(\Sigma, \balpha, \bbeta,z)$ and $(\Sigma, \balpha', \bbeta ', z)$ and an exact Hamiltonian isotopy on $(\Sigma, \omega)$ taking $\balpha$ to $\balpha '$ and $\bbeta$ to $\bbeta '$, which furthermore never crosses the basepoint, we claim that each coherent orientation system $\mathfrak{o}_{\balpha, \bbeta}$ for the first diagram determines a unique equivalence class of coherent orientation system $\mathfrak{o}_{\balpha ', \bbeta '}$ for the second. This is part of the statement of \cite[Theorem 7.3]{Disks1}, and can be understood as follows. First note that it will suffice to show that there is a correspondence $\pi_{2}(\bx, \bx)_{\mathcal{H}_{1}} \cong \pi_{2}(\by, \by)_{\mathcal{H}_{2}}$ between homotopy classes of periodic disks based at some intersection point $\bx$  on $\mathcal{H}_{1}$ and homotopy classes of periodic disks based at some intersection point $\by$ on $\mathcal{H}_{2}$. With this fact established, a coherent orientation system on the first diagram uniquely determines an equivalence class on the second diagram, since as we have already observed equivlance classes of orientation systems on $\mathcal{H}_{2}$ are determined by their values on the periodic domains based at a single intersection point. The correspondence $\pi_{2}(\bx, \bx)_{\mathcal{H}_{1}} \cong \pi_{2}(\by, \by)_{\mathcal{H}_{2}}$ is realized by a certain concatenation with a homotopy class with varying boundary conditions, as we now explain. 

Following \cite[Proof of Theorem 7.3]{Disks1}, let us denote our isotopy by $\Psi_{t}: \Sigma \rightarrow \Sigma$ and set $\balpha_{t}  = \Psi_{t}(\balpha)$ and $\bbeta_{t}  = \Psi_{t}(\bbeta)$. Define $\pi_{2}^{\Psi_{t}}(\bx, \by)$ to be the set of homotopy classes of Whitney disks which connect $\boldsymbol{x} \in \mathbb{T}_{\balpha} \cap \mathbb{T}_{\bbeta}$ to $\boldsymbol{y} \in \mathbb{T}_{\balpha '} \cap \mathbb{T}_{\bbeta '}$ and have boundary conditions $u(0,t) \in \balpha_{t}$, $u(1,t) \in \bbeta_{t}$. We now explain how a single class $\phi \in \pi_{2}^{\Psi_{t}}(\bx, \by)$ establishes the desired correspondence $\pi_{2}(\bx, \bx)_{\mathcal{H}_{1}} \cong_{\phi} \pi_{2}(\by, \by)_{\mathcal{H}_{2}}$ via a certain conjugation. Given $u$ representing $A \in \pi_{2}(\bx, \bx)_{\mathcal{H}_{1}}$ and a disk $v$ representing the class $\phi \in \pi_{2}^{\Psi_{t}}(\bx, \by)$, we can construct a disk $\bar{v} \natural u \natural v$ by concatenation. Such a disk lies in $\pi_{2}^{\Psi_{1-t}*\text{Id}*\Psi_{t}}(\by, \by)$, which is the set of homotopy classes of Whitney disks which connect $\boldsymbol{y} \in \mathbb{T}_{\balpha'} \cap \mathbb{T}_{\bbeta'}$ to itself, and have boundary conditions matching $\bar{\balpha}* \balpha_{0}* \balpha$ and $\bar{\bbeta}* \bbeta_{0}* \bbeta$ on it's two sides, where $\bar{\balpha}$ and $\bar{\bbeta}$ are the curves traversed in the opposite direction. We now claim two things:
\begin{enumerate}
	\item This correspondence establishes a bijection $$\pi_{2}(\bx, \bx) \cong \pi_{2}^{\Psi_{1-t}*\text{Id}*\Psi_{t}}(\by, \by).$$
	\item There is also a bijection $$\pi_{2}^{\Psi_{1-t}*\text{Id}*\Psi_{t}}(\by, \by) \cong \pi_{2}(\by,\by).$$
\end{enumerate}
We omit the proofs of these facts, but note that both can be understood by thinking of the space of periodic domains at $\bx$ as a subspace of the fundamental group of the path space between the Heegaard curves, based at the constant path $\bx$. In this context, one can show that an isotopy of the Heegaard curves gives rise to an identification between path spaces, and that the class $\phi$ yields an identification between the corresponding loop spaces. This line of reasoning can be used to establish both bijections. For the interested reader, a precise argument explaining related facts in a more general setting can be found in \cite[Section 3.3]{simplicial}. Finally, one should note that a class $\phi \in \pi_{2}^{\Psi_{t}}(\bx, \by)$ does in fact exist for $\by =\Psi_{1}(\bx)$, because given an intersection point $\bx \in \mathbb{T}_{\balpha} \cap \mathbb{T}_{\bbeta}$, we may just follow it with the isotopy to obtain a disk $u(s,t) = \Psi_{t}(\bx)$ which satisfies the requirements for a disk with varying boundary conditions between $\bx$ and $\by = \Psi_{1}(\bx)$. This completes the explanation of the identification between equivalence classes of coherent orientation systems on $(\Sigma, \balpha, \bbeta,z)$ and $(\Sigma, \balpha', \bbeta ', z)$.

With respect to the aforementioned orientation systems there is an induced chain homotopy equivalence
$$\Gamma_{\bbeta \rightarrow \bbeta '}^{\balpha \rightarrow \balpha '}: CF^{-}(\Sigma, \balpha, \bbeta, z, \mathfrak{s}, \mathfrak{o}_{\balpha, \bbeta}) \rightarrow CF^{-}(\Sigma, \balpha', \bbeta ', z, \mathfrak{s}, \mathfrak{o}_{\balpha ', \bbeta '})$$
which we call a continuation map associated to the Hamiltonian isotopy $\phi_{t}$. We will also use the notation
$$\Gamma_{\bbeta}^{\balpha \rightarrow \balpha '} = \Gamma_{\bbeta \rightarrow \bbeta }^{\balpha \rightarrow \balpha '}$$
and
$$\Gamma_{\bbeta \rightarrow \bbeta'}^{\balpha} = \Gamma_{\bbeta \rightarrow \bbeta' }^{\balpha \rightarrow \balpha }$$
By \cite[Lemma 2.12]{FourManifoldInvariants}, these equivalences compose naturally under concatenation of isotopies in the sense that
$$ \Gamma^{\balpha \rightarrow \balpha''}_{\bbeta} \sim  \Gamma^{\balpha' \rightarrow \balpha''}_{\bbeta} \circ \Gamma^{\balpha \rightarrow \balpha'}_{\bbeta}$$
and
$$\Gamma^{\balpha \rightarrow \balpha'}_{\bbeta \rightarrow \bbeta '} \sim \Gamma^{\balpha \rightarrow \balpha'}_{\bbeta '} \circ \Gamma^{\balpha}_{\bbeta \rightarrow \bbeta '} \sim  \Gamma^{\balpha'}_{\bbeta \rightarrow \bbeta '} \circ \Gamma^{\balpha \rightarrow \balpha'}_{\bbeta }.$$
Furthermore, by their definition in \cite[Proof of Theorem 7.3]{Disks1}, they satisfy $\Gamma_{\bbeta \rightarrow \bbeta}^{\balpha \rightarrow \balpha} = \text{id}_{CF^{-}(\Sigma, \balpha, \bbeta, z, \mathfrak{s}, \mathfrak{o}_{\balpha, \bbeta})}$. 

As suggested by the notation, we note that while the continuation map is a priori associated to a Hamiltonian isotopy between the isotopic attaching curves, in the cases of interest for us its chain homotopy class will actually be independent of the choice of isotopy. To see this, we recall:

\begin{claim}\cite[Lemma 9.1 and Section 9.1]{Disks1}
	\label{SmallIsotopy}
	Let $(\Sigma, \bbeta, \bbeta', z)$ be a pointed diagram such that each curve  $\beta_{i} '$ in $\bbeta'$ is obtained from the curve $\beta_{i}$ in $\bbeta$ by performing a small Hamiltonian isotopy which introduces two transverse intersection points between $\beta_{i}$ and $\beta_{i} '$, and no intersection points between $\beta_{i} '$ and $\beta_{j}$ for $j \neq i$.  Then the diagram represents $\#^{g}(S^{1} \times S^{2})$. There is a unique $\text{Spin}^{c}$ structure $\mathfrak{s}_{0} \in \text{Spin}^{c}(\#^{g}(S^{1} \times S^{2}))$ such that $c_{1}(\mathfrak{s}_{0}) = 0$, and the diagram $(\Sigma, \bbeta, \bbeta ', z)$  is strongly $\mathfrak{s}_{0}$-admissible. Furthermore, there is a choice of coherent orientation system $\mathfrak{o}_{\bbeta, \bbeta '}$ on this diagram such that in the highest nontrivial relative homological grading $HF^{-}(\Sigma, \bbeta, \bbeta' , z,  \mathfrak{s}_{0}, \mathfrak{o}_{\bbeta, \bbeta '})$  is isomorphic to $\mathbb{Z} =: \langle \theta_{\bbeta, \bbeta '} \rangle $ for a generator we denote $\theta_{\bbeta, \bbeta '}$.
\end{claim}
Using the generator $\theta_{\bbeta, \bbeta '}$ we have an analogous triangle map to that defined above, which is also shown to be an equivalence:

\begin{claim}\cite[Theorem 9.8 and Section 9.1]{Disks1} 
	\label{IsotopyMapProperties}
	If $(\Sigma, \balpha, \bbeta, \bbeta ', z)$ is a strongly $\mathfrak{s}$-admissible triple diagram and $\bbeta'$ is related to $\bbeta$ as in the statement of Lemma \ref{SmallIsotopy} by a sufficiently small isotopy, then $$F_{\balpha, \bbeta, \bbeta '} ( \cdot \otimes \theta_{\bbeta, \bbeta '} ): CF^{-}(\Sigma, \balpha, \bbeta, z, \mathfrak{s}_{\balpha, \bbeta}, \mathfrak{o}_{\balpha, \bbeta}) \rightarrow CF^{-}(\Sigma, \balpha, \bbeta ', z, \mathfrak{s}_{\balpha, \bbeta '}, \mathfrak{o}_{\balpha, \bbeta '})$$ is a chain homotopy equivalence.

\end{claim}
Furthermore, we have
\begin{claim}\cite[Proposition 11.4]{Cylindrical}
	\label{ContinuationTriangleRelation}
	If the triple diagram $(\Sigma, \balpha, \bbeta, \bbeta ',z)$ is strongly $\mathfrak{s}$-admissible and $\bbeta'$ is related to $\bbeta$ as in the statement of Lemma \ref{SmallIsotopy} by a sufficiently small isotopy, then the continuation map associated to any Hamiltonian isotopy $\phi_{t}$ between $\bbeta$ and $\bbeta '$ satisfies
	$$ \Gamma_{\bbeta \rightarrow \bbeta '} ^{\balpha} \sim  F_{\balpha, \bbeta, \bbeta '} ( \cdot \otimes \theta_{\bbeta, \bbeta '} )$$
\end{claim}
We thus see that the continuation maps associated to small Hamiltonian isotopies of the attaching curves are independent of the choice of isotopy.

Finally, we introduce notation for a composition of triangle maps and continuation maps associated to strong $\balpha$-equivalences and strong $\bbeta$-equivalences.

\begin{defn}\cite[Section 2 and Lemma 2.13]{FourManifoldInvariants} 
	\label{StrongEquivalenceMap}
	Given two strongly $\mathfrak{s}$-admissible diagrams $(\Sigma, \balpha_{1}, \bbeta_{1}, z)$ and $(\Sigma, \balpha_{2}, \bbeta_{2},z)$ which are strongly equivalent, one can construct another pointed diagram $(\Sigma, \balpha_{1}', \bbeta_{1}', z)$ such that:
	\begin{enumerate}
		\item $\balpha_{1}'$ and $\bbeta_{1}'$ are obtained respectively from $\balpha_{1}$ and $\bbeta_{1}$ by special isotopies.
		\item $\balpha_{2}$ and $\bbeta_{2}$ are obtained respectively from $\balpha_{1}'$ and $\bbeta_{1}'$ by (small variations of) sequences of handleslides as in Lemma \ref{StandardHandleslide}.
		\item The quadruple diagram $(\Sigma, \balpha_{1}', \bbeta_{1}', \balpha_{2}, \bbeta_{2})$ is strongly $\mathfrak{s}$-admissible for the unique $\text{Spin}^{c}$-structure on $X_{\balpha_{1}', \bbeta_{1}', \balpha_{2}, \bbeta_{2}}$ which restricts to $\mathfrak{s}$ on $Y_{\balpha_{1}', \bbeta_{2}}$ and $\mathfrak{s}_{0}$ on $Y_{\balpha_{1}', \balpha_{2}}$ and $Y_{\bbeta_{1}', \bbeta_{2}}$.
	\end{enumerate}
	We define a map,
	$$\Phi_{\bbeta_{1} \rightarrow \bbeta_{2}}^{\balpha_{1} \rightarrow \balpha_{2}}(\cdot, \mathfrak{s}): CF^{-}(\Sigma, \balpha_{1}, \bbeta_{1}, z, \mathfrak{s}) \rightarrow CF^{-}(\Sigma, \balpha_{2}, \bbeta_{2}, z, \mathfrak{s})$$
	associated to two such strongly equivalent diagrams by the formula:
	$$ \Phi_{\bbeta_{1} \rightarrow \bbeta_{2}}^{\balpha_{1} \rightarrow \balpha_{2}}(\cdot, \mathfrak{s}) = \Psi_{\bbeta_{1}' \rightarrow \bbeta_{2}}^{\balpha_{2}} \circ \Psi_{\bbeta_{1}'}^{\balpha_{1}' \rightarrow \balpha_{2}} \circ \Gamma_{\bbeta_{1} \rightarrow \bbeta_{1}'}^{\balpha_{1} \rightarrow \balpha_{1}'}.$$
	We will sometimes use the notation
	$$ \Phi_{\bbeta \rightarrow \bbeta '}^{\balpha} =  \Phi_{\bbeta \rightarrow \bbeta '}^{\balpha \rightarrow \balpha}$$
	and 
	$$ \Phi_{\bbeta}^{\balpha \rightarrow \balpha '} =  \Phi_{\bbeta \rightarrow \bbeta }^{\balpha \rightarrow \balpha'}.$$
	
\end{defn}

\subsection{The Weak Heegaard Floer Invariants}
\label{WeakHeegaardFloerInvariants}
Using the previous two subsections, we are now in position to define the value on vertices of the morphism of graphs 
$$CF^{-}: \mathcal{G}_{\text{man}} \rightarrow \text{Trans}(P(\text{Kom}(\mathbb{Z}[U]\text{-}\text{Mod})))$$ which will partially define the weak invariants underlying the maps in Theorem \ref{StrongHeegaard2}. In doing so, we will also define the value on vertices of the morphism of graphs 
$$HF^{-}: \mathcal{G}_{\text{man}} \rightarrow P(\mathbb{Z}[U]\text{-}\text{Mod})$$ appearing in Corollary \ref{StrongHeegaard}.

\begin{defn}
	\label{WeakHeegaardFloer1}
	Fix some pointed isotopy diagram $H = (\Sigma, A, B, z)$ (corresponding to a vertex in $\mathcal{G}_{\text{man}}$) representing the pointed $3$-manifold $(Y,z)$. For $\mathfrak{s} \in \text{Spin}^{c}(Y)$, let
	$$\text{Admiss}_{(\Sigma, A, B,z)}(\mathfrak{s}) = \{ \text{strongly $\mathfrak{s}$-admissible diagrams } (\Sigma, \balpha, \bbeta,z) | [\balpha] = A, [\bbeta] = B \}$$
	be the set of strongly $\mathfrak{s}$-admissible diagrams representing $H$. By \cite[Proofs of Lemma 5.2 and Lemma 5.4]{Disks1}, this is nonempty for all $\mathfrak{s} \in \text{Spin}^{c}(Y)$. Choose any diagram $\mathcal{H} = (\Sigma, \balpha, \bbeta, z) \in \text{Admiss}_{(\Sigma, A, B,z)}(\mathfrak{s})$, and fix a coherent orientation system $\mathfrak{o}_{\balpha, \bbeta}$ on it. By \cite[Lemma 7.3]{Disks1}, the transitive system $CF^{-}(\Sigma, \balpha, \bbeta, z, \mathfrak{s}, \mathfrak{o}_{\balpha, \bbeta})$ can be used along with the continuation maps $\Gamma$ to induce coherent orientation systems for all strongly $\mathfrak{s}$-admissible diagrams representing the isotopy diagram $H$. Then by \cite[Lemma 2.12]{FourManifoldInvariants}, the transitive systems  $CF^{-}(\Sigma, \balpha, \bbeta, z, \mathfrak{s}, \mathfrak{o}_{\balpha, \bbeta})$ ranging over all $(\Sigma, \balpha, \bbeta, z) \in \text{Admiss}_{(\Sigma, A, B,z)}(\mathfrak{s})$ fit into a transitive system (of morphisms between transitive systems) with respect to the continuation maps $\Gamma_{\bbeta \rightarrow \bbeta '}^{\balpha \rightarrow \balpha '}$. We can therefore define a single transitive system (see Section \ref{TransitiveSystemSection}) in $\text{Kom}(\mathbb{Z}[U]\text{-}\text{Mod})$, which we denote by
	$$CF^{-}(H, \mathfrak{s}).$$
	Finally, we define the value of the weak Heegaard invariant $CF^{-}$ on the isotopy diagram $H$ by
	$$CF^{-}(H) = \bigoplus_{\mathfrak{s} \in \text{Spin}^{c}(Y)} CF^{-}(H, \mathfrak{s}).$$
	
	Passing to homology, we obtain instead that the $\mathbb{Z}[U]$-modules $HF^{-}(\Sigma, \balpha, \bbeta, z, \mathfrak{s}, \mathfrak{o}_{\balpha, \bbeta})$ for $(\Sigma, \bbeta, \balpha,z) \in \text{Admiss}_{(\Sigma, A, B,z)}(\mathfrak{s})$ fit into a transitive system of isomorphisms with respect to the  continuation maps. We denote the colimit of this transitive system by
	$$HF^{-}(H, \mathfrak{s})$$
	and define
	$$HF^{-}(H) = \bigoplus_{\mathfrak{s} \in \text{Spin}^{c}(Y)} HF^{-}(H, \mathfrak{s}).$$

	We now proceed to fix the data of the underlying coherent orientation systems we will use to define $CF^{-}(H')$ for all other isotopy diagrams $H'$ in $\mathcal{G}_{\text{man}}$. First consider the path component of $\mathcal{G}_{\text{man}}$ containing the fixed isotopy diagram $H$ chosen above. We note that by Proposition \ref{ReidemeisterSinger2}, the collection of vertices in this path component corresponds to the collection of all isotopy diagrams representing the fixed $3$-manifold $(Y,z)$. Given another isotopy diagram $H'$ in this path component, choose a sequence of edges $\gamma$ in $(\mathcal{G}_{\text{man}})_{(Y,z)}$ from $H$ to $H'$. For any diagrams $\mathcal{H} \in H$ and $\mathcal{H}'$ in $H'$, the constructions described in the previous subsections yield a composition of maps associated to $\gamma$ on the underlying chain complexes: $$CF^{-}(\gamma): CF^{-}(\mathcal{H}) \rightarrow CF^{-}(\mathcal{H}').$$ Here the sequence of maps $CF^{-}(\gamma)$ of course depends on our previously fixed choice of coherent orientation system for $\mathcal{H}$; we described in the previous subsections how each of the possible constituent maps in the composition $CF^{-}(\gamma)$ induces a coherent orientation system on the target given a coherent orientation system on the domain, and it is this induced orientation system that we fix on $\mathcal{H}'$. One can check that this induced orientation on $\mathcal{H}'$ is independent of the choice of path $\gamma$, by verifying the commutativity of the induced orientations occurring in each of the five types of distinguished rectangle, and in a simple handleswap. We will verify this commutativity in Section \ref{monodromyoforientationsystems}. We thus see that our specification of the coherent orientation systems $\mathfrak{o}_{\balpha, \bbeta}$ on all diagrams $\mathcal{H}$ representing $H$ actually yields a choice of coherent orientation systems for all diagrams in the same path component as $H$. Repeating this entire procedure for all path components in $\mathcal{G}_{\text{man}}$, we have thus defined 
	$$CF^{-}(H) = \bigoplus_{\mathfrak{s} \in \text{Spin}^{c}(Y)} CF^{-}(H, \mathfrak{s})$$
	and 
	$$HF^{-}(H) = \bigoplus_{\mathfrak{s} \in \text{Spin}^{c}(Y)} HF^{-}(H, \mathfrak{s})$$
	for all isotopy diagrams $H$ in $\mathcal{G}_{\text{man}}$.
\end{defn}

\begin{remark}
	We interpret the role of coherent orientations in the definition above loosely as follows. If one fixes any Heegaard diagram for a $3$-manifold, there are numerous inequivalent choices of coherent orientation system (in fact there are $2^{b_{1}(Y)}$ such choices, see \cite[Lemma 4.16]{Disks1}). The above definition just says one should fix whichever choice they prefer, and then take care to use the maps induced by the standard Heegaard moves (or diffeomorhisms isotopic to the identity) to carry this choice around when considering different Heegaard diagrams for the same $3$-manifold.   
\end{remark}

To finish defining the weak Heegaard invariants, we need to associate isomorphisms to all edges in $\mathcal{G}_{\text{man}}$. We begin by assigning maps to edges corresponding to strong $\balpha$-equivalences and strong $\bbeta$-equivalences. 

\begin{defn}
	\label{AlphaEquivalenceMap}
	Given two strongly $\balpha$-equivalent isotopy diagrams $H_{1}= (\Sigma, A, B,z)$, $H_{2} = (\Sigma, A', B,z) \in |\mathcal{G}_{\text{man}}|$ representing $(Y,z)$, and $\mathfrak{s} \in \text{Spin}^{c}(Y)$, fix strongly $\mathfrak{s}$-admissible diagrams $(\Sigma, \balpha, \bbeta,z)$ and $(\Sigma, \balpha ', \bbeta ,z)$ representing them. As above, this is possible by \cite[Section 5]{Disks1}. Then by \cite[Theorem 2.3 and Lemma 2.13]{FourManifoldInvariants}, the chain homotopy equivalences $\Phi_{\bbeta}^{\balpha \rightarrow \balpha '}$ fit into a morphism of transitive systems between the transitive systems $CF^{-}(H, \mathfrak{s})$ appearing in Definition \ref{WeakHeegaardFloer1}. Thus for the edge $e \in \mathcal{G}_{\text{man}}^{\balpha}(H_{1}, H_{2})$ corresponding to the strong $\balpha$-equivalence, we can associate this collection of chain homotopy equivalences (or equivalently, this collection of isomorphisms in $\text{Kom}(\mathbb{Z}[U]\text{-}\text{Mod})$) to obtain a morphism
	$$\Phi_{e} := \Phi_{B}^{A \rightarrow A'}: CF^{-}(H_{1}) \rightarrow CF^{-}(H_{2})$$
	We note that such a collection of chain homotopy equivalences is precisely the notion of an isomorphism in $\text{Trans}(\text{Kom}(\mathbb{Z}[U]\text{-}\text{Mod}))$. We define the chain homotopy equivalences associated to a strong $\bbeta$-equivalence analogously.
\end{defn}

To finish defining the weak Heegaard invariants, we assign isomorphisms to stablizations and diffeomorphisms in the next two subsections.

\subsection{Stabilization Maps}
We recall maps on the Heegaard Floer chain complexes which can be associated to stabilizations. Given a strongly $\mathfrak{s}$-admissible diagram $\mathcal{H} = (\Sigma, \balpha, \bbeta, z)$ and a stablization thereof, $\mathcal{H} ' = (\Sigma \# \Sigma_{0}, \balpha ', \bbeta ', z)$, each coherent orientation system $\mathfrak{o}$ on $\mathcal{H}$ induces a coherent orientation system $\mathfrak{o}'$ on $\mathcal{H}'$. With respect to these orientation systems, there is a $\mathbb{Z}[U]\text{-}\text{equivariant}$ chain isomorphism $$\sigma_{\mathcal{H} \rightarrow \mathcal{H}'}: CF^{-}_{J_{s}}(\Sigma, \balpha, \bbeta,z, \mathfrak{s}, \mathfrak{o}) \rightarrow CF^{-}_{J_{s}'(T)}(\Sigma \# \Sigma_{0}, \balpha ', \bbeta ',z, \mathfrak{s}, \mathfrak{o}')$$ defined for sufficiently large values of a parameter $T$.
This is established in \cite[Theorems 10.1 and 10.2]{Disks1}. 

The curves $\balpha ' \cup \bbeta '$ are obtained as the disjoint union of $\balpha \cup \bbeta$ along with a pair of closed curves $\alpha '$, $\beta '$ contained in $\Sigma_{0}$ which intersect transversally in a single point we will denote by $c$. We can identify the intersection points in the two diagrams above by assigning to an intersection point $\boldsymbol{x} \in \mathbb{T}_{\balpha} \cap \mathbb{T}_{\bbeta}$ the intersection point $\sigma_{\mathcal{H} \rightarrow \mathcal{H}'}(\boldsymbol{x}) = \boldsymbol{x} \times c \in \mathbb{T}_{\balpha '} \cap \mathbb{T}_{\bbeta '}$. Fix complex structures $j_{\Sigma}$ on $\Sigma$ and $j_{\Sigma_{0}}$ on $\Sigma_{0}$, and let $j'(T)$ denote the complex structure on $\Sigma \# \Sigma_{0}$ defined by inserting a neck of length $T$ between $(\Sigma,j_{\Sigma})$ and $(\Sigma_{0}, j_{\Sigma_{0}})$. Then one can associate to a perturbation $J_{s}$ of $\text{Sym}^{g}(j_{\Sigma})$ on $\text{Sym}^{g}(\Sigma)$ and a perturbation $J^{0}_{s}$ of $j_{\Sigma_{0}}$,  a perturbation $J_{s}'(T)$ of $\text{Sym}^{g+1}(j'(T))$ on $\text{Sym}^{g+1}(\Sigma \# \Sigma_{0})$. The key argument needed to establish the above chain isomorphism then comes in the form of a neck stretching argument which yields the following glueing result: for sufficiently large values of $T$, a homotopy class of Whitney disk $\phi \in \pi_{2}(\boldsymbol{x}, \boldsymbol{y})$ on $\Sigma$ with Maslov index 1, and the corresponding homotopy class $\phi ' \in \pi_{2}(\boldsymbol{x} \times c, \boldsymbol{y} \times c)$ on $\Sigma \# \Sigma_{0}$ with Maslov index 1, there is an identificaton of moduli spaces $\mathcal{M}_{J_{s}}(\phi) \cong \mathcal{M}_{J_{s}'(T)}(\phi ')$. From this it follows readily that the above map is a $\mathbb{Z}[U]\text{-}\text{equivariant}$ chain isomorphism.

\begin{defn}
	\label{StabilizationMaps}
	Given \textit{isotopy} diagrams $H$ and $H'$, with $H'$ obtained from $H$ via a stabilization, we can associate a morphism of transitive systems $$\sigma_{H \rightarrow H'}: CF^{-}(H) \rightarrow CF^{-}(H')$$
	as follows. Fixing any $\text{Spin}^{c}$-structure $\mathfrak{s}$, strongly $\mathfrak{s}$-admissible representatives $\mathcal{H}$ and $\mathcal{H}'$ which realize the stabilization, and almost complex structure data on $\mathcal{H}$, there is some choice of almost complex structure data on $\mathcal{H}'$ for which the stabilization ismorphism is defined. As described in \cite[Lemma 2.15]{FourManifoldInvariants}, the stabilization maps $\sigma_{\mathcal{H} \rightarrow \mathcal{H}'}$ commute with the change of almost complex structure maps, and with the strong equivalence maps. This implies that the chain isomorphisms $\{\sigma_{\mathcal{H} \rightarrow \mathcal{H}'} \}$, when the complex structures are chosen so that they are defined, satisfy the commutativity requirements required of a morphism of transitive systems as in Definition \ref{MapOfTransitiveSystems}. We can complete this partially defined morphism of transitive systems for other choices of complex structure data by declaring the stabilization map $\sigma_{\mathcal{H} \rightarrow \mathcal{H}'}$ to be computed for allowable complex structure data, followed by the appropriate change of almost complex structure homotopy equivalence $\Phi_{J_{s} \rightarrow J_{s}'}$. We define the morphism of transitive systems associated to the corresponding destablization to be the inverse of $\sigma_{H \rightarrow H'}$.

	On the level of homology, we obtain via the colimit construction in Definition \ref{WeakHeegaardFloer1} canonical isomorphisms $i_{\mathcal{H}}: HF^{-}(\mathcal{H}) \rightarrow HF^{-}(H)$ and $i_{\mathcal{H}'}: HF^{-}(\mathcal{H}') \rightarrow HF^{-}(H')$. We set $\sigma_{H \rightarrow H'} =i_{\mathcal{H}'} \circ \sigma_{\mathcal{H} \rightarrow \mathcal{H}'} \circ  i_{\mathcal{H}}^{-1}$ for any choice of such $\mathcal{H}$, $\mathcal{H}'$. This is independent of the choice of diagrams $\mathcal{H}$ and $\mathcal{H'}$ by the aforementioned result \cite[Lemma 2.15]{FourManifoldInvariants}
\end{defn}
\subsection{Diffeomorphism Maps}
Finally, we need to discuss how diffeomorphisms of Heegaard surfaces lead to maps on the associated chain complexes. We use the following definition:

\begin{defn}\cite[Definition 9.23]{Naturality}
	\label{MapInducedByDiffeo}
	Fix a strongly $\mathfrak{s}$-admissible diagram $(\Sigma, \balpha, \bbeta, z)$, with $|\balpha| = | \bbeta | = k$. Let $j$ be an almost complex structure on $\Sigma$, and $J_{s}$ be a perturbation of the almost complex structure $\text{Sym}^{k}(j)$ on $\text{Sym}^{k}(\Sigma)$. Let $\mathfrak{o}$ be a coherent orientation system on the diagram.  Fix a diffeomorphism $d: \Sigma \rightarrow \Sigma' $, and set $d(\balpha) = \balpha '$, $d(\bbeta) = \bbeta '$. We define an associated map
	as follows. First, the almost complex structure $j$ and perturbation $J_{s}$ can be conjugated via the differential of $d$ to obtain $j' = d_{*}(j)$ on $\Sigma$ and $J_{s}' = d_{*}(J_{s})$ a perturbation of $d_{*}(j)$ on $\text{Sym}^{k}(\Sigma')$. The diffeomorphism $d$ provides an identification between periodic classes $\pi_{2}(\boldsymbol{x}, \boldsymbol{x}) \cong \pi_{2}(\boldsymbol{x}', \boldsymbol{x}')$ for $\boldsymbol{x} \in \mathbb{T}_{\balpha} \cap \mathbb{T}_{\bbeta}$ and $\boldsymbol{x}' \in \mathbb{T}_{\balpha '} \cap \mathbb{T}_{\bbeta '}$. We use this identification to push forward the coherent orientation system $\mathfrak{o}$ to obtain an induced orientation system $\mathfrak{o}'$. This yields a chain isomorphism
	$$d_{J_{s}, J_{s}'}: CF^{-}_{J_{s}}( \Sigma, \balpha, \bbeta, z, \mathfrak{s}, \mathfrak{o}) \rightarrow CF^{-}_{J_{s}'}(\Sigma ', \balpha ' , \bbeta ', z',  d(\mathfrak{s}), \mathfrak{o}')$$ as can be seen easily by a direct argument pushing forward all intersection points, and holomorphic discs connecting two such, via $d$. We note that the change of complex structure maps commute with the maps $d_{J_{s}, J_{s}'}$ (by a direct check), so there is also an induced map of transitive systems
	$$d_{*}: CF^{-}( \Sigma, \balpha, \bbeta, z, \mathfrak{s}) \rightarrow CF^{-}(\Sigma' , \balpha ', \bbeta ', z', d(\mathfrak{s}))$$
	Finally, by Lemma \ref{ContinuationTriangleRelation} and \cite[Lemma 9.24]{Naturality} the maps $d_{*}$ commute with the maps $\Gamma_{\bbeta \rightarrow \bbeta ' }^{\balpha \rightarrow \balpha '}$ appearing in Definition \ref{WeakHeegaardFloer1}. Thus by using the continuation maps the maps $d_{*}$ can be extended to a morphism of the transitive systems in Definition \ref{WeakHeegaardFloer1}
	$$d_{*}: CF^{-}(H, \mathfrak{s}) \rightarrow  CF^{-}(H', d(\mathfrak{s}))$$
	where $H = (\Sigma, [\balpha], [\bbeta], z)$ and $H'  = (\Sigma ', [\balpha '], [\bbeta '], z')$.
	
	On the level of homology, the above definitions give a well defined map of the $\mathbb{Z}[U]$-modules in Definition \ref{WeakHeegaardFloer1},
	$$d_{*}: HF^{-}(H, \mathfrak{s}) \rightarrow HF^{-}(H', d(\mathfrak{s})).$$
	
\end{defn}

\subsection{Monodromy of Orientation Systems}
\label{monodromyoforientationsystems}
We now establish the claim made in Definition \ref{WeakHeegaardFloer1} that there is no monodromy of induced orientations systems around loops of diagrams. This will finish the proof that the Heegaard Floer invariants are weak Heegaard invariants, and will also establish in particular that there is no monodromy of induced orientation systems around the special loops relevant to strong Heegaard invariance. We will show:

\begin{claim}
	\label{monodromy}
	There is no monodromy of coherent orientation systems around loops composed of isomorphisms associated to isotopies, handleslides, stabilizations and diffeomorphisms of diagrams.
\end{claim}

\begin{cor}
	\label{specialmonodromy}
	There is no monodromy of coherent orientation systems around the loops determined by simple handleswaps and distinguished rectangles.
\end{cor}

\begin{rem}
	\label{inducedorientation}
	Note that we have already described how each type of Heegaard move induces a map on the Heegaard Floer chain complex which is defined with respect to an orientation system induced on the codomain from one on the domain. See Section \ref{trianglemapssection}, Definition \ref{StabilizationMaps} and Definition \ref{MapInducedByDiffeo} for the relevant definitions and references.
\end{rem}

To prove Lemma \ref{monodromy}, it will be useful to think first about the canonical orientation systems, introduced in \cite{Disks2}, in particular. We first note the following fact about those orientation systems:

\begin{claim}
	\label{canmonodromy}
	The maps associated to isotopies, handleslides, stabilizations and diffeomorphisms take canonical orientations to canonical orientations.	
\end{claim}
\begin{proof}
	By \cite[Section 8]{Disks2}, each such map induces an isomorphism on the totally twisted module $\underline{HF}^{\infty}$. By \cite[Theorem 10.1]{Disks2}, the canonical orientation system is characterized by the resulting isomorphism type of $\underline{HF}^{\infty}$. \qedhere
\end{proof}

\begin{cor}
	\label{canmonodromy2}
	The maps associated to the loops defining simple handleswaps and distinguished rectangles take canonical orientations to canonical orientations.
\end{cor}

Having established the statement of Lemma \ref{monodromy} for canonical orientation systems, we now turn to proving it for general orientation systems. Fix a Heegaard diagram $\mathcal{H}$ for $Y$, and let
$$ \mathcal{O}_{\mathcal{H}} = \{\text{equivalence classes of coherent orientation systems on } \mathcal{H} \}$$
and $\mathfrak{o}_{c} \in \mathcal{O}_{\mathcal{H}}$ be the canonical orientation system. There is a map
$$\text{diff}_{\mathfrak{o}_{c}}: \mathcal{O}_{\mathcal{H}} \rightarrow \text{Hom}(\Pi_{\bx}^{\mathcal{H}}, \mathbb{Z}/2\mathbb{Z}),$$
where $\Pi_{\bx}^{\mathcal{H}}$ is the group of periodic domains based at $\bx$ in $\mathcal{H}$, defined by measuring the difference between an orientation system and $\mathfrak{o}_{c}$ on each periodic domain. In symbols, for an orientation system $\mathfrak{o}$, $\text{diff}_{\mathfrak{o}_{c}}(\mathfrak{o})$ is a map:
$$ \text{diff}_{\mathfrak{o}_{c}}(\mathfrak{o}): \Pi_{\bx}^{\mathcal{H}} \rightarrow \mathbb{Z}/2\mathbb{Z}$$
satisfying
\[ \text{diff}_{\mathfrak{o}_{c}}(\mathfrak{o})[D] = \begin{cases} 
0 & \text{if } \mathfrak{o}|_{D} = \mathfrak{o}_{c}|_{D} \\
1 & \text{if } \mathfrak{o}|_{D} \neq \mathfrak{o}_{c}|_{D} 
\end{cases}\]
We note that the analogous map $\text{diff}_{\mathfrak{o}}$ can be defined for any coherent orientation system $\mathfrak{o}$ on $\mathcal{H}$.

\begin{claim}
	$$\text{diff}_{\mathfrak{o}_{c}}: \mathcal{O}_{\mathcal{H}} \rightarrow \text{Hom}(\Pi_{\bx}^{\mathcal{H}}, \mathbb{Z}/2\mathbb{Z})$$ is a well-defined bijection.
\end{claim}
\begin{proof}
	We've already seen in the proof of Lemma \ref{EquivalenceClassesonS1xS2} that equivalence classes of orientation systems are determined by their values on a basis for the periodic domains based at $\bx$. Thus if $\text{diff}_{\mathfrak{o}_{c}}(\mathfrak{o}) = \text{diff}_{\mathfrak{o}_{c}}(\mathfrak{o}')$, then $\mathfrak{o}$ and $\mathfrak{o}'$ agree with $\mathfrak{o}_{c}$ on the same set of domains, and $\mathfrak{o} = \mathfrak{o}'$, so $\text{diff}_{\mathfrak{o}_{c}}$ is injective.
	
	Given a morphism $\phi: \Pi_{\bx}^{\mathcal{H}} \rightarrow \mathbb{Z}/2\mathbb{Z}$, define $\mathfrak{o}$ to satisfy
	\[ \mathfrak{o}|_{D} = \begin{cases} 
	\mathfrak{o}_{c}|_{D} & \text{if } \phi(D) = 0 \\
	- \mathfrak{o}_{c}|_{D} & \text{if } \phi(D) = 1 
	\end{cases}\]
	By the comments in Remark \ref{uniquenessremark} , $\mathfrak{o}$ can then be extended over a complete set of paths to obtain a coherent orientation system satisfying $\text{diff}_{\mathfrak{o}_{c}}(\mathfrak{o}) = \phi$.
\end{proof}

Note that by Remark \ref{inducedorientation}, any loop L composed of Heegaard moves induces a map
$$\text{L}: \mathcal{O}_{\mathcal{H}} \rightarrow \mathcal{O}_{\mathcal{H}}.$$
With this notation, proving Lemma \ref{monodromy} amounts to showing that $\text{L}(\mathfrak{o})=\mathfrak{o}$ for each diagram $\mathcal{H}$ and each coherent orientation system $\mathfrak{o} \in \mathcal{O}_{\mathcal{H}}$, while Corollary \ref{canmonodromy2} says $\text{L}(\mathfrak{o}_{c})=\mathfrak{o}_{c}$. To prove Lemma \ref{monodromy}, we will show that the maps on orientation systems induced by Heegaard moves commute with the $\text{diff}$ maps in the following sense. Given a Heegaard move from a diagram $\mathcal{H}_{1}$ to a diagram $\mathcal{H}_{2}$, let $f:\mathcal{O}_{\mathcal{H}_{1}} \rightarrow \mathcal{O}_{\mathcal{H}_{2}}$ be the induced map of coherent orientation systems. Similarly, let $\widetilde{f}: \text{Hom}(\Pi_{\bx}^{\mathcal{H}_{1}}, \mathbb{Z}/2\mathbb{Z}) \rightarrow \text{Hom}(\Pi_{\bx'}^{\mathcal{H}_{2}}, \mathbb{Z}/2\mathbb{Z})$ be the map induced from precomposition with the identifications $ \Pi_{\bx'}^{\mathcal{H}_{2}} \cong H_{2}(Y) \cong \Pi_{\bx}^{\mathcal{H}_{1}}$ described in \cite[Proposition 2.15 and Lemma 2.17]{Disks1}. We then have the following result:

\begin{claim}
	For each of the maps $f$ on coherent orientation systems induced  by Heegaard moves, we have 
	$$\widetilde{f} \circ \text{diff}_{\mathfrak{o}} = \text{diff}_{f(\mathfrak{o})} \circ \text{f}$$
	for all coherent orientation systems $\mathfrak{o}$.
\end{claim}

\begin{proof}
	For each type of Heegaard move, the definition of the map $f$ (which specifies how a coherent orientation on the starting diagram determines one on the target diagram) can be described by an identification $$\phi:  \Pi_{\bx'}^{\mathcal{H}_{2}} \rightarrow \Pi_{\bx}^{\mathcal{H}_{1}}.$$
	If we let $\widetilde{\phi}$ be the map
	$$\widetilde{\phi}: \text{Hom}(\Pi_{\bx}^{\mathcal{H}_{1}}, \mathbb{Z}/2\mathbb{Z}) \rightarrow \text{Hom}(\Pi_{\bx'}^{\mathcal{H}_{2}}, \mathbb{Z}/2\mathbb{Z})$$ 
	induced by precomposition with $\phi$, then for each Heegaard move one can show that 
	\begin{equation}
	\label{obviouscommutativity}
	\widetilde{\phi} \circ \text{diff}_{\mathfrak{o}} = \text{diff}_{f(\mathfrak{o})} \circ \text{f}.
	\end{equation}
	For example, in the case of a handleslide this follows from an inspection of the proof of \cite[Lemma 8.7]{Disks1} (or see for comparison Remark \ref{uniquenessremark}), as we now explain. Let $(\Sigma, \balpha, \bbeta, \bgamma, z)$ be a triple diagram determining a handleslide, as in Lemma \ref{HandleslideMapProperties}, and $\mathcal{H}_{\balpha, \bbeta}$ and $\mathcal{H}_{\balpha, \bgamma}$ be the initial and final diagrams for the handleslide. Fix a homotopy class of triangle $\psi_{0} \in \pi_{2}(\bx, \by, \bz)$ in the triple diagram. The map $f: \mathcal{O}_{\mathcal{H}_{\balpha, \bbeta}} \rightarrow \mathcal{O}_{\mathcal{H}_{\balpha, \bgamma}}$ on orientation systems induced by the handleslide is defined by applying Lemma \ref{ExistenceOfCoherentTriangles} with the orientation $\mathfrak{o}_{\bbeta, \bgamma}$ on the intermediary diagram $\mathcal{H}_{\bbeta, \bgamma}$ chosen to be that of Lemma \ref{StandardHandleslide}. Recall that in this case the indeterminacy of Lemma \ref{ExistenceOfCoherentTriangles} disappears, as the group $Q$ from Remark \ref{uniquenessremark} is zero, so an orientation $\mathfrak{o}_{\balpha, \bgamma}$ is uniquely determined by an orientation $\mathfrak{o}_{\balpha, \bbeta}$. The key property we will need to recall from this particular application of the construction of Lemma \ref{ExistenceOfCoherentTriangles} is that for each periodic class, $A_{\balpha, \bgamma} \in \Pi_{\bx}^{\mathcal{H}_{\balpha, \bgamma}}$, there is a unique pair of classes $A_{\bbeta, \bgamma}$ and $A_{\balpha, \bbeta}$ such that
	\begin{equation} 
	\label{uniquedomain}
	\psi_{0} + A_{\balpha, \bgamma} = \psi_{0} + A_{\bbeta, \bgamma} + A_{\balpha, \bbeta},
	\end{equation}
	and furthermore the induced orientation is constructed such that this relation is respected by the orientations over these domains (see \cite[Proof of Lemma 8.7]{Disks1}). With this in mind, we proceed to establish Equation \eqref{obviouscommutativity} as follows. Fix a periodic domain $P_{\balpha, \bgamma} \in \Pi_{\bx}^{\mathcal{H}_{\balpha, \bgamma}}$, and two orientation systems $\mathfrak{o}_{\balpha, \bbeta}$ and $\mathfrak{o}_{\balpha, \bbeta}'$ on $\mathcal{H}_{\balpha, \bbeta}$. Let $\mathfrak{o}_{\balpha, \bgamma} = f(\mathfrak{o}_{\balpha, \bbeta})$ and $\mathfrak{o}_{\balpha, \bgamma}' = f(\mathfrak{o}_{\balpha, \bbeta}')$ be the corresponding orientations induced by the handleslide, as described above. Finally, let $\phi: \Pi_{\bx}^{\mathcal{H}_{\balpha, \bbeta}} \rightarrow \Pi_{\bx}^{\mathcal{H}_{\balpha, \bgamma}}$ be the map which sends a domain $P_{\balpha, \bbeta}$ to the unique class $P_{\balpha, \bgamma}$ specified by Equation \eqref{uniquedomain}. We then compare the two sides of Equation \eqref{obviouscommutativity}, and find
	\begin{equation*}
	(\widetilde{\phi} \circ \text{diff}_{\mathfrak{o}_{\balpha, \bbeta}}(\mathfrak{o}_{\balpha, \bbeta}'))[P_{\balpha, \bgamma}] = (\mathfrak{o}_{\balpha, \bbeta} - \mathfrak{o}_{\balpha, \bbeta}')[P_{\balpha, \bbeta}]\\
	\end{equation*} 
	while
	\begin{align*}
	(\text{diff}_{f(\mathfrak{o}_{\balpha, \bbeta})} \circ f(\mathfrak{o}_{\balpha, \bbeta}'))[P_{\balpha, \bgamma}] &= (\mathfrak{o}_{\balpha, \bgamma} - \mathfrak{o}_{\balpha, \bgamma}')[P_{\balpha, \bgamma}]\\
	&= (\mathfrak{o}_{\balpha, \bbeta} - \mathfrak{o}_{\balpha, \bbeta}')[P_{\balpha, \bbeta}] + (\mathfrak{o}_{\beta, \bgamma} - \mathfrak{o}_{\bbeta, \bgamma}')[P_{\bbeta, \bgamma}] \\
	&= (\mathfrak{o}_{\balpha, \bbeta} - \mathfrak{o}_{\balpha, \bbeta}')[P_{\balpha, \bbeta}]
	\end{align*}
	where the second equality uses the fact that the induced orientations respect Equation \eqref{uniquedomain} by construction, and the third equality uses the fact that $\mathfrak{o}_{\bbeta, \bgamma} = \mathfrak{o}_{\bbeta, \bgamma}'$. This completes the proof of Equation \eqref{obviouscommutativity} for the case of handleslides.
	
	With Equation \eqref{obviouscommutativity} understood, we further claim that for each Heegaard move the identification
	$$ \phi: \Pi_{\bx'}^{\mathcal{H}_{2}} \rightarrow \Pi_{\bx}^{\mathcal{H}_{1}}$$
	used to define $f$ and $\widetilde{\phi}$ agrees with the identification $$ \Pi_{\bx'}^{\mathcal{H}_{2}} \cong H_{2}(Y) \cong \Pi_{\bx}^{\mathcal{H}_{1}}$$
	used to define $\widetilde{f}$. In particular, this implies that $\widetilde{f} = \widetilde{\phi}$, which together with Equation \eqref{obviouscommutativity} proves the lemma. 
	
	In the case of a handleslide, the aforementioned claim follows from the following observations:
	\begin{enumerate}
		\item Each periodic domain is uniquely determined by the part of it's boundary that lies on one set of attaching curves.
		\item In the case at hand, the identification $\phi$ preserves the part of the boundary of domains that lies on one set of attaching circles, and can be characterized as the unique identification of periodic domains with that property.
		\item The identifications
		$$ \Pi_{\bx}^{\mathcal{H}_{1}} \cong H_{2}(Y)$$
		and
		$$ \Pi_{\bx'}^{\mathcal{H}_{2}} \cong H_{2}(Y)$$
		(described in \cite[Lemma 2.17 and subsequent remarks]{Disks1}) are determined by the part of the boundary of each periodic domain that lies on one set of attaching curves.
	\end{enumerate}
	The first two facts imply that for each periodic domain $D \in \Pi_{\bx}^{\mathcal{H}_{2}}$, $D$ and $\phi(D)$ share the part of their boundary that lies on one set of attaching curves, and the third fact then ensures that $D$ and $\phi(D)$ have the same image in $H_{2}(Y)$, which establishes the claim.
	
	The preceding argument proves that the lemma holds for handleslides. For the other Heegaard moves, Equation \eqref{obviouscommutativity} once again follows directly from the definition of the maps $f$ and $\phi$ (although in these cases the definitions are themselves simpler), and the second claim follows from considerations analogous  to those listed above. We leave the details of these cases to the reader. \qedhere
\end{proof}

\begin{cor}
	\label{comm}
	For any loop L of Heegaard moves, $\widetilde{L} \circ \text{diff}_{\mathfrak{o}} = \text{diff}_{L(\mathfrak{o})} \circ \text{L}$ for all coherent orientation systems $\mathfrak{o}$.
\end{cor}

\begin{proof}[Proof of Lemma \ref{monodromy}]
	Corollary \ref{comm} applied to $\mathfrak{o}_{c}$ yields:
	\[
	\widetilde{\text{L}} \circ \text{diff}_{\mathfrak{o}}(\mathfrak{o}_{c}) = \text{diff}_{L(\mathfrak{o})} \circ \text{L}(\mathfrak{o}_{c}) = \text{diff}_{L(\mathfrak{o})}(\mathfrak{o}_{c})
	\]
	where the last equality comes from Corollary \ref{canmonodromy2}. Since we are considering here a loop of diagrams, the map $\widetilde{L}$ is the identity, and the previous equation yields:
	\[
	\text{diff}_{\mathfrak{o}}(\mathfrak{o}_{c}) =  \text{diff}_{L(\mathfrak{o})}(\mathfrak{o}_{c})
	\]
	or, equivalently,
	\[
	\mathfrak{o} = L(\mathfrak{o})
	\]
	as desired \qedhere
\end{proof}

\section{Heegaard Floer Homology as a Strong Heegaard Invariant}\label{MainProof}

In the previous section we recalled the definition of the weak Heegaard invariants 
$$CF^{-}: \mathcal{G}_{\text{man}} \rightarrow \text{Trans}(P(\text{Kom}(\mathbb{Z}[U]\text{-}\text{Mod})))$$  
and 
$$HF^{-}: \mathcal{G}_{\text{man}} \rightarrow P(\mathbb{Z}[U]\text{-}\text{Mod})$$ 
underlying the strong Heegaard invariants appearing in Theorem \ref{StrongHeegaard2} and Corollary \ref{StrongHeegaard} respectively. To establish Theorem \ref{StrongHeegaard2} we need to check the four axioms required of a strong Heegaard invariant in Definition \ref{StrongHeegaardInvariant}. 

The proofs of axioms 1 and 2 given in \cite[Section 9.2, pg 131]{Naturality} for $\mathbb{F}_{2}[U]\text{-} \text{Mod}$ apply almost directly to establish axioms 1 and 2 for $CF^{-}$ and $HF^{-}$ as Heegaard invariants valued in $\text{Trans}(P(\text{Kom}(\mathbb{Z}[U]\text{-}\text{Mod})))$ and $P(\mathbb{Z}[U]\text{-} \text{Mod})$ respectively, as we now summarize for $CF^{-}$. 

For axiom 1, the functoriality of $CF^{-}$ restricted to $\mathcal{G}_{\text{man}}^{\alpha}$ and $\mathcal{G}_{\text{man}}^{\beta}$ follows from Lemma \ref{HandleslideMapProperties} and \cite[Theorem 2.3]{FourManifoldInvariants}. The functoriality of $CF^{-}$ restricted to $\mathcal{G}_{\text{man}}^{\text{diff}}$ is immediate from Definition \ref{MapInducedByDiffeo}. Finally, for a stabilization $e$ and the corresponding destabilization $e'$,  $CF^{-}(e') = CF^{-}(e)^{-1}$ by Definition \ref{StabilizationMaps}.

For axiom 2, we need to establish that the images under $CF^{-}$ of distinguished rectangles in $\mathcal{G}_{\text{man}}$ (recall Definition \ref{distinguishedrectangle}) form commuting rectangles. For a rectangle of type 1, commutativity follows from Lemma \ref{HandleslideMapProperties} and \cite[Theorem 2.3]{FourManifoldInvariants}. For a rectangle of type 2, commutativity follows from \cite[Lemma 2.15]{FourManifoldInvariants}. For a rectangle of type 3, commutativity follows from \cite[Lemma 9.24]{Naturality}. Finally, rectangles of type 4 and 5 can be seen to commute by directly applying the arguments in \cite[pg. 131]{Naturality}.

We now investigate axiom 3. Let $H = (\Sigma, A, B, z) \in |\mathcal{G}_{\text{man}}|$ be an isotopy diagram, $d: H \rightarrow H$ a diffeomorphism of isotopy diagrams which is isotopic to $\text{Id}_{\Sigma}$, and $d_{*} := CF^{-}(e)$ where $e \in \mathcal{G}^{\text{diff}}_{\text{man}}(H,H)$ is the edge corresponding to $d$. We need to show $d_{*} =  \text{Id}_{CF^{-}(H)}$ as morphisms of transitive systems in $P(\text{Kom}(\mathbb{Z}[U]\text{-}\text{Mod}))$.  We adapt and restate the argument given in \cite[Proposition 9.27]{Naturality} in order to explain why it can be applied to the case of (projective) integral coefficients. We show the following result.

\begin{thm}\label{ContinuityAxiom}
	Let $(\Sigma, \balpha, \bbeta, z)$ be a strongly $\mathfrak{s}$-admissible diagram with $|\balpha| = |\bbeta| = g$. Suppose that $d: \Sigma \rightarrow \Sigma$ is a diffeomorphism isotopic to $\text{Id}_{\Sigma}$, and let $\balpha ' = d(\balpha)$ and $\bbeta ' = d(\bbeta)$. Let $\mathfrak{o}_{\balpha, \bbeta}$ be a coherent orienation system on $(\Sigma, \balpha, \bbeta, z)$ and $\mathfrak{o}_{\balpha ', \bbeta '}$ be the coherent orientation system on $(\Sigma, \balpha ', \bbeta ', z)$ induced by $d$. Then with respect to these orienation systems, we have
	$$ d_{*} = \pm \Gamma^{\balpha \rightarrow \balpha'}_{\bbeta \rightarrow \bbeta'} : HF^{-}(\Sigma, \balpha, \bbeta,z, \mathfrak{s}, \mathfrak{o}_{\balpha, \bbeta}) \rightarrow HF^{-}(\Sigma, \balpha', \bbeta',z', \mathfrak{s}, \mathfrak{o}_{\balpha ', \bbeta '})$$
	Furthermore, as maps 
	$$d_{*}, \pm \Gamma^{\balpha \rightarrow \balpha'}_{\bbeta \rightarrow \bbeta'} : CF^{-}(\Sigma, \balpha, \bbeta,z, \mathfrak{s}, \mathfrak{o}_{\balpha, \bbeta}) \rightarrow CF^{-}(\Sigma, \balpha', \bbeta', z', \mathfrak{s}, \mathfrak{o}_{\balpha ', \bbeta '})$$
	$d_{*}$ is chain homotopic to one of $\pm \Gamma^{\balpha \rightarrow \balpha'}_{\bbeta \rightarrow \bbeta'}$. 
\end{thm}
In fact, this theorem will establish axiom (3) in Definition \ref{StrongHeegaardInvariant} for the weak Heegaard invariants $CF^{-}$ and $HF^{-}$ above. Since $d$ is isotopic to $\text{Id}_{\Sigma}$ by hypothesis, we have $\balpha' $ is isotopic to $\balpha$ and $ \bbeta '$ is isotopic to $\bbeta$, so $H := (\Sigma, [\balpha], [\bbeta], z) = (\Sigma, [\balpha'], [\bbeta'], z')$. The induced map of transitive systems $d_{*}: CF^{-}(H) \rightarrow CF^{-}(H)$ defined in Definition \ref{MapInducedByDiffeo} is then computed by extending the following map by conjugation with the continuation maps:
$$CF^{-}(\Sigma, \balpha, \bbeta, z, \mathfrak{o}_{\balpha, \bbeta}) \xrightarrow{d_{*}} CF^{-}(\Sigma, \balpha ', \bbeta ',  z, \mathfrak{o}_{\balpha ', \bbeta '}) \xrightarrow{\Gamma_{\balpha ' \rightarrow \balpha}^{\bbeta ' \rightarrow \bbeta}} CF^{-}(\Sigma, \balpha, \bbeta, z, \mathfrak{o}_{\balpha, \bbeta}).$$ 
Since $\Gamma_{\balpha ' \rightarrow \balpha}^{\bbeta ' \rightarrow \bbeta} \sim (\Gamma_{\balpha  \rightarrow \balpha'}^{\bbeta  \rightarrow \bbeta '})^{-1}$ and $d_{*} \sim \pm \Gamma_{\balpha  \rightarrow \balpha'}^{\bbeta  \rightarrow \bbeta '}$ by Theorem \ref{ContinuityAxiom}, we see that $d_{*}: CF^{-}(H) \rightarrow CF^{-}(H)$ is the extension of a map $CF^{-}(\Sigma, \balpha, \bbeta, z, \mathfrak{o}_{\balpha, \bbeta}) \rightarrow CF^{-}(\Sigma, \balpha, \bbeta, z, \mathfrak{o}_{\balpha, \bbeta})$ which is homotopic to plus or minus the identity. Thus we see that $d_{*} = \text{Id}_{CF^{-}(H)}$ as morphisms in $\text{Trans}(P(\text{Kom}(\mathbb{Z}[U]\text{-}\text{Mod})))$.

\begin{proof}[Proof of Theorem \ref{ContinuityAxiom}]
	Since $d$ is isotopic to $\text{id}_{\Sigma}$, we may decompose it into a composition of diffeomorphisms $d_{i}$ on some diagrams $\mathcal{H}_{i} = (\Sigma, \balpha_{i}, \bbeta_{i})$, such that each $d_{i}$ is Hamiltonian isotopic to $id_{\Sigma}$ for some symplectic form $\omega_{i}$ on $\Sigma$, and the diagrams satisfy the intersection properties $|\alpha \cap d_{i}(\alpha) | = |\beta \cap d_{i}(\beta) | = 2$ for all $\alpha \in \boldsymbol{\alpha}_{i-1}$ and $\beta \in \boldsymbol{\beta}_{i-1}$. As described in \cite[Proposition 9.27]{Naturality}, it will suffice to prove the result for such a $d_{i}$. So let $d_{t}$ for $t \in \mathbb{R}$ be a Hamiltonian isotopy which is independent of $t$ for $t \in (-\infty, 0]$ and $t \in [1, \infty)$, and which connects $id_{\Sigma}$ to a diffeomorphism $d$ of $\mathcal{H} = (\Sigma, \balpha, \bbeta)$. Throughout the proof, we will use the notation $d_{t}(\balpha)= \balpha_{t}$, $d_{t}(\bbeta) = \bbeta_{t}$, and use primes to indicate the values of various quantities at $t =1$.
	
	Fix the data of a complex structure $j$ on $\Sigma$ and a perturbation $J_{s}$ of $\text{Sym}^{g}(j)$ on $\text{Sym}^{g}(\Sigma)$, and for $t \in \mathbb{R}$ let $j_{t} =(d_{t})_{*}(j)$ and $J_{s,t} = (\text{Sym}^{g}(d_{t}))_{*}(J_{s})$. As described in the sections above, there are numerous chain maps on the Heegaard Floer chain complexes we can associate with the isotopy $d_{t}$ and this induced almost complex structure data. We will be concerned here with the following three:
	
	\begin{enumerate}
		\item We can change the almost complex structure on $\text{Sym}^g(\Sigma)$ from $J_{s} = J_{s,0}$ to $J_{s}' = J_{s,1}$, while leaving the attaching curves unchanged, and consider the induced map
		$$\Phi_{J_{s} \rightarrow J_{s}'}: CF^{-}_{J_{s}}( \Sigma, \balpha, \bbeta, z, \mathfrak{o}_{\balpha, \bbeta}) \rightarrow CF^{-}_{J_{s}'}( \Sigma, \balpha , \bbeta, z, \mathfrak{o}_{\balpha, \bbeta}).$$
		We recall here that this map is defined (in \cite{Disks1}) by counting Maslov index 0 discs $u: [0,1] \times \mathbb{R} \rightarrow \text{Sym}^{g}(\Sigma)$ connecting some $\boldsymbol{x} \in \mathbb{T}_{\balpha} \cap \mathbb{T}_{\bbeta}$ to some $\boldsymbol{y} \in \mathbb{T}_{\balpha } \cap \mathbb{T}_{\bbeta }$, which satisfy $u(0,t) \in \balpha$, $u(1,t) \in \bbeta$ and $du/ds+ J_{s,t}(du/dt) =0$. 
		
		\item We can leave the almost complex structures $(j, J_{s})$ fixed, and consider the effect on the Floer complex of altering only the attaching curves via the map
		$$\Gamma_{\balpha \rightarrow \balpha '}^{\bbeta \rightarrow \bbeta '}: CF^{-}_{J_{s}}( \Sigma, \balpha, \bbeta, z, \mathfrak{o}_{\balpha, \bbeta}) \rightarrow CF^{-}_{J_{s}}( \Sigma, \balpha ', \bbeta ', z, \mathfrak{o}_{\balpha', \bbeta'})$$
		associated to the Hamiltonian isotopy $d_{t}$. In this case, the map is defined by counting Maslov index 0  discs $u$ connecting some $\boldsymbol{x} \in \mathbb{T}_{\balpha} \cap \mathbb{T}_{\bbeta}$ to some $\boldsymbol{y} \in \mathbb{T}_{\balpha '} \cap \mathbb{T}_{\bbeta '}$ as above, but with dynamic boundary conditions $u(0,t) \in \balpha_{t}$, $u(1,t) \in \bbeta_{t}$, and which satisfy $du/ds + J_{s}(du/dt) =0$.
		
		\item We define a new sort of continuation map associated with $d_{t}$,
		$$\Gamma_{d_{t}}: CF^{-}_{J_{s}}( \Sigma, \balpha, \bbeta, z, \mathfrak{o}_{\balpha, \bbeta}) \rightarrow CF^{-}_{J_{s}'}( \Sigma, \balpha ', \bbeta ', z, \mathfrak{o}_{\balpha ' , \bbeta '}) $$
		which combines the ideas from the previous two. This map is defined to count Maslov index 0 discs $u$ which connect some $\boldsymbol{x} \in \mathbb{T}_{\balpha} \cap \mathbb{T}_{\bbeta}$ to some $\boldsymbol{x}' \in \mathbb{T}_{\balpha '} \cap \mathbb{T}_{\bbeta '}$, have dynamic boundary conditions $u(0,t) \in \balpha_{t}$, $u(1,t) \in \bbeta_{t}$, and which satisfy $du/ds + J_{s,t}(du/dt) =0$. We will denote the set of homotopy classes of Whitney disks (not necessarily $J_{s,t}$-holomorphic) satisfying the boundary conditions above by $\pi_{2}^{d_{t}}(\boldsymbol{x},\boldsymbol{x}')$, and for $\phi \in \pi_{2}^{d_{t}}(\boldsymbol{x},\boldsymbol{x}')$ we will denote the moduli space of $J_{s,t}$-holomorphic maps representing $\phi$ by $\mathcal{M}^{d_{t}}(\phi)$.
		
	\end{enumerate}
	
	We claim that the third map in the list above is in fact chain homotopic to the map $d_{J_{s}, J_{s}'}$ from Definition \ref{MapInducedByDiffeo}. To see this, we first explain that if a diffeomorphism (which we also indicate by $d$, as an abuse of notation) $d: \Sigma \rightarrow \Sigma$ isotopic to the identity (via an isotopy $d_{t}$) is sufficiently close to $\text{Id}_{\Sigma}$, then the map defined in case (3) above satisfies $\Gamma_{d_{t}} = d_{J_{s}, J_{s}'}$ as chain maps. Indeed, by taking $d$ to be a sufficiently small perturbation of $\text{Id}_{\Sigma}$, we may ensure the isotopy $d_{t}$ is arbitrarily close to being constant in $t$. For an isotopy which is constant in $t$, the definition of the continuation map in (3) above counts Maslov index $0$ disks with fixed boundary conditions which are $J_{s}$-holomorphic. The only such maps are constant maps. Thus, by Gromov compactness, if the isotopy $d_{t}$ is sufficiently close to being constant, the Maslov index $0$ solutions to the equation appearing in the definition of $\Gamma_{d_{t}}$ will be close enough to constant disks to ensure that $\Gamma_{d_{t}}$ will be a nearest point map. 
	
	Next we note that the definition of $\Gamma_{d_{t}}$ depends on a choice of coherent orientation system for the moduli spaces $\mathcal{M}^{d_{t}}(\phi)$. As explained in \cite[Proof of Proposition 7.3]{Disks1}, when $\pi_{2}^{d_{t}}(\boldsymbol{x}, \boldsymbol{x}') \neq 0$  a single homotopy class $\phi \in \pi_{2}^{d_{t}}(\boldsymbol{x}, \boldsymbol{x}') \cong \mathbb{Z}$ yields via glueing an identification between periodic classes $\pi_{2}(\boldsymbol{x}, \boldsymbol{x}) \cong_{\phi} \pi_{2}(\boldsymbol{x}', \boldsymbol{x}')$ on the two diagrams, and a choice of orientation for $\mathcal{M}^{d_{t}}(\phi)$ then yields an identification between coherent orientation systems on the two diagrams. Thus given a coherent orientation system $\mathfrak{o}_{\balpha, \bbeta}$ on $(\Sigma, \balpha, \bbeta)$, and an orientation on $\mathcal{M}^{d_{t}}(\phi)$, we obtain an induced orientation $\mathfrak{o}_{\balpha ', \bbeta '}$ on $(\Sigma, \balpha ', \bbeta ')$ with respect to which the map is defined. We claim that we may arrange for this induced orientation to agree with that induced by $d_{J_{s}, J_{s}'}$. Indeed, fix for each $\boldsymbol{x} \in \mathbb{T}_{\balpha} \cap \mathbb{T}_{\bbeta}$ a homotopy class $\phi_{\boldsymbol{x}} \in \pi_{2}^{d_{t}}(\boldsymbol{x}, \boldsymbol{x}')$. We can choose orientations on all such $\mathcal{M}^{d_{t}}(\phi_{\boldsymbol{x}})$ freely such that $\Gamma_{d_{t}}$ is the positive nearest point map (with the generator corresponding to an intersection point being taken to the positive generator corresponding to the nearest intersection point after the isotopy is performed), and then extend these choices to a coherent system. The coherent orientation $\mathfrak{o}_{\balpha ', \bbeta '}$ on $( \Sigma, \balpha', \bbeta', z')$ induced by $\Gamma_{d_{t}}$ that results will then be the same as that induced by $d_{J_{s}, J_{s}'}$, as we now explain.  Fix $\boldsymbol{x}, \boldsymbol{y} \in \mathbb{T}_{\balpha} \cap \mathbb{T}_{\bbeta}$ and let $\boldsymbol{x} ' =d(\boldsymbol{x})$ and  $\boldsymbol{y}' =d(\boldsymbol{y})$ be the corresponding interesection points in $\mathbb{T}_{\balpha'} \cap \mathbb{T}_{\bbeta'}$. Given a homotopy class $\psi \in \pi_{2}(\boldsymbol{x}, \boldsymbol{y})$ and a positively oriented Whitney disk $u$ from $\boldsymbol{x}$ to $\boldsymbol{y}$ in the class $\psi$, the orientation system induced by $d_{J_{s}, J_{s}'}$ will positively orient the corresponding disk $d(u)$ representing the class $d(\psi) \in \pi_{2}(\boldsymbol{x}', \boldsymbol{y}')$ (see Definition \ref{MapInducedByDiffeo}). We need to show that the disk $d(u)$ is also positively oriented in the orientation system induced by $\Gamma_{d_{t}}$. As described above, the orientation on $d(u)$ induced by $\Gamma_{d_{t}}$ is specified as follows. We consider representative disks $v_{1}$ and $v_{2}$ for the classes $\phi_{\boldsymbol{x}} \in \pi_{2}^{d_{t}}(\boldsymbol{x}, \boldsymbol{x}')$ and  $\phi_{\boldsymbol{y}} \in \pi_{2}^{d_{t}}(\boldsymbol{y}, \boldsymbol{y}')$, which we may assume are both positively oriented by the choice we made for orientatations on $\mathcal{M}^{d_{t}}(\phi_{\boldsymbol{x}})$ and  $\mathcal{M}^{d_{t}}(\phi_{\boldsymbol{y}})$. We then consider the glued disk $v_{2} \natural u \natural \overline{v_{1}}$. Since an orientation has been specified on each constituent disk and our system is coherent, this glued disk also has a specified orientation, which is positive given our choices. Finally, we note that this disk is identified with $d(u)$ under the identification between coherent orientation systems in the two diagrams, and thus $d(u)$ must also be oriented positively. We thus see that both maps induce the same coherent orientation system on the target and both take the form of the positive nearest point map, so $\Gamma_{\phi_{t}} = \phi_{J_{s}, J_{s}'}$.
	
	Finally, we can decompose our original diffeomorphism $d: (\Sigma, \balpha_{0}, \bbeta_{0}) \rightarrow (\Sigma, \balpha_{1}, \bbeta_{1})$ into a sequence of diffeomorphisms $d^{1}, d^{2}, \cdots, d^{N}$, where $d^{i}: (\Sigma, \balpha_{(i-1)/N}, \bbeta_{(i-1)/N}) \rightarrow (\Sigma, \balpha_{i/N}, \bbeta_{i/N})$ and each $d^{i}$ is isotopic to $\text{Id}_{\Sigma}$ via isotopies $d^{i}_{t}$. For sufficiently large $N$, we can ensure that the continuation map $\Gamma_{d^{i}_{t}}$ associated to each consitituent isotopy satisfies 
	$$\Gamma_{d^{i}_{t}} = (d^{i})_{J_{s,(i-1)/N},J_{s,i/N}}$$ 
	by the argument in the preceding paragraphs. Furthermore, by inserting long necks one can see that the composition of the corresponding continuation maps is homotopic to the original continuation map:
	$$\Gamma_{d_{t}} \sim \left( \Gamma_{d^{N}_{t}} \circ \cdots \circ \Gamma_{d^{1}_{t}} \right).$$
	Since
	$$d_{J_{s}, J_{s}'} = d^{N}_{J_{s,(N-1)/N}, J_{s,1}} \circ \cdots \circ d^{1}_{J_{s,0}, J_{s,1/N}} $$
	we thus see that $d_{J_{s}, J_{s}'} \sim \Gamma_{d_{t}}$, which establishes the claim.
	
	Using Definition \ref{MapInducedByDiffeo} we have $d_{*} = \Phi_{J_{s}' \rightarrow J_{s}} \circ d_{J_{s}, J_{s}'}$. Thus to complete the proof it will in fact suffice to show that $\Phi_{J_{s}' \rightarrow J_{s}} \circ d_{J_{s}, J_{s}'} \sim \pm \Gamma^{\alpha \rightarrow \alpha '}_{\beta \rightarrow \beta'}$, or, since $d_{J_{s}, J_{s}'} \sim \Gamma_{d_{t}}$ and $\Phi_{J_{s}' \rightarrow J_{s}}^{-1} \sim \Phi_{J_{s} \rightarrow J_{s}'}$, to show that
	\begin{equation}\label{SufficientCondition}
	\Gamma_{d_{t}} \sim \pm \Phi_{J_{s} \rightarrow J_{s}'} \circ \Gamma^{\alpha \rightarrow \alpha '}_{\beta \rightarrow \beta'}.
	\end{equation}
	
	To see that Equation \eqref{SufficientCondition} is true, we consider the following generalized notion of a continuation map, of which each of the three maps involved are a special case. Consider a Hamiltonian isotopy $\phi_{t}$ and a generic two parameter family of almost complex structures $K_{s,t}$ on $\text{Sym}^{g}(\Sigma)$ which are perturbations of $\text{Sym}^{g}(k_{t})$ where $k_{t}$ is a one parameter family of complex structures on $\Sigma$. Here we assume for convenience as above that this data is independent of $t$ for $t \in (-\infty, 0]$ and $t \in [1, \infty)$. We set $\balpha_{t} = \phi_{t}(\balpha)$ and $\bbeta_{t}= \phi_{t}(\bbeta)$. Given such data we can associate the \textit{continuation map with respect to  $(\phi_{t}, K_{s,t})$}:
	\begin{equation}\label{ContinuationMapAlongPath}
	\Gamma_{(\phi_{t}, K_{s,t})} : CF^{-}_{K_{s, 0}}(\Sigma, \balpha_{0}, \bbeta_{0}) \rightarrow CF^{-}_{K_{s,1}}(\Sigma, \balpha_{1}, \bbeta_{1})
	\end{equation} 
	by counting Maslov index 0 discs $u$ connecting some $\boldsymbol{x} \in \mathbb{T}_{\balpha_{0}} \cap \mathbb{T}_{\bbeta_{0}}$ to some $\boldsymbol{y} \in \mathbb{T}_{\balpha_{1}} \cap \mathbb{T}_{\bbeta_{1}}$, with dynamic boundary conditions $u(0,t) \in \balpha_{t}$, $u(1,t) \in \bbeta_{t}$, and which satisfy $$\dfrac{du}{ds} + K_{s,t}(\dfrac{du}{dt}) =0.$$ 
	
	The maps $\Gamma_{d_{t}}, \Phi_{J_{s} \rightarrow J_{s}'}$ and $\Gamma^{\alpha \rightarrow \alpha '}_{\beta \rightarrow \beta'}$ above are then the continuation maps with respect to the data $(d_{t}, J_{s,t})$, $(\text{id}_{\Sigma}, J_{s,t})$ and $(d_{t}, J_{s,0})$ respectively. Furthermore, since the homotopy classes of such continuation maps are natural under concatenation and rescaling of the $\phi_{t}$ and $K_{s,t}$ by \cite[Lemma 2.12]{FourManifoldInvariants} (see also the argument below), the composite $ \Phi_{J_{s} \rightarrow J_{s}'} \circ \Gamma^{\alpha \rightarrow \alpha '}_{\beta \rightarrow \beta'}$ is homotopic to the continuation map defined with respect to the data 
	
	\[ (d_{t,1}, J_{s,t,1}) := \begin{cases} 
	(d_{2t}, J_{s,0}) & t \in [0,1/2] \\
	(\text{id}_{\Sigma}, J_{s,2t-1}) & t \in [1/2,1].
	\end{cases}
	\]

	\begin{figure}[h!]
		\centering
		\begin{tikzpicture}
		\node[anchor=south west,inner sep=0] (image) at (0,0) {\includegraphics[width= .7\textwidth]{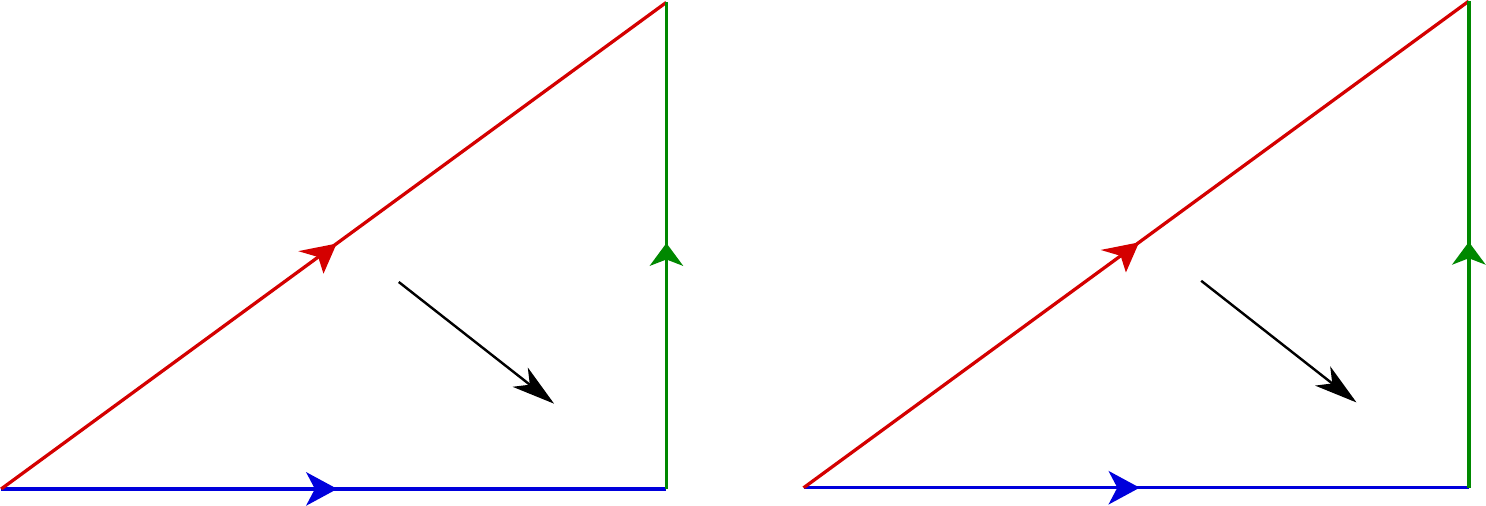}};
		\begin{scope}[x={(image.south east)},y={(image.north west)}]
		\node[text = red] at (.24,.66) {\large $J_{s,t}$};
		\node[text = red] at (.78,.65) {\large $d_{t}$};
		\node[text = blue] at (.27,-.04) {\large $J_{s}$};
		\node[text = blue] at (.8,-.04) {\large $d_{2t}$};
		\node[text = ForestGreen] at (.51,.41) {\large $J_{s,2t-1}$};
		\node[text = ForestGreen] at (1.02,.4) {\large $\text{Id}$};
		\node at (.36,.38) {\large $K_{s,t,\tau}$};
		\node at (.89,.38) {\large $\phi_{t,\tau}$};
		
		\end{scope}
		\end{tikzpicture}
		\caption{A schematic of the complex structure and isotopy data defining the continuation maps $\textcolor{red}{\Gamma_{d_{t}}}$ and (a continuation map homotopic to) $ \textcolor{ForestGreen}{\Phi_{J_{s} \rightarrow J_{s}'}} \circ \textcolor{blue}{\Gamma_{\bbeta \rightarrow \bbeta '}^{\balpha \rightarrow \balpha '}}$, and the homotopies between the two sets of data. The data defining $\Gamma_{d_{t}}$ is represented by the top edges of the two triangles, while the data defining $\Phi_{J_{s} \rightarrow J_{s}'} \circ \Gamma_{\bbeta \rightarrow \bbeta '}^{\balpha \rightarrow \balpha '}$ is represented by the bottom edges followed by the vertical edges.}
		\label{continuationmap}
	\end{figure}

	Consider now two Hamiltonian isotopies $\phi_{t,0}$ and $\phi_{t,1}$ with $\phi_{0,0} = \phi_{0,1} = \text{id}_{\Sigma}$ and $\phi_{1,0} = \phi_{1,1}$, and two generic two parameter families $K_{s,t,0}$ and $K_{s,t,1}$ with $K_{s,0,0} = K_{s,0,1}$ and $K_{s,1,0} = K_{s,1,1}$. We will complete the proof by showing that a generic homotopy $h =(\phi_{t,\tau}, K_{s,t,\tau})$ between $(\phi_{t,0}, K_{s,t,0})$ and $(\phi_{t,1}, K_{s,t,1})$  induces a chain homotopy between $\Gamma_{(\phi_{t,0}, K_{s,t,0})}$ and $\pm \Gamma_{(\phi_{t,1}, K_{s,t,1})}$. In particular, Equation \eqref{SufficientCondition} will follow, as the data $(d_{t}, J_{s,t})$ used to define $\Gamma_{d_{t}, J_{s,t}} =: \Gamma_{d_{t}} $ is homotopic to the data $(d_{t,1}, J_{s,t,1})$ used to define $\Gamma_{d_{t,1}, J_{s,t,1}} \sim \Phi_{J_{s} \rightarrow J_{s}'} \circ \Gamma^{\alpha \rightarrow \alpha '}_{\beta \rightarrow \beta'}$.
	
	Fixing $\tau$, let $\pi_{2}^{\tau}(\boldsymbol{x}, \boldsymbol{y})$ denote the homotopy classes of discs $u$ which connect $\boldsymbol{x}$ to $\boldsymbol{y}$, and which satisfy the boundary conditions  $u(0,t) \in \phi_{t,\tau}(\balpha)$, $u(1,t) \in \phi_{t,\tau}(\bbeta)$. Given a homotopy class $\phi \in \pi_{2}^{\tau}(\boldsymbol{x}, \boldsymbol{y})$, we denote by $\mathcal{M}_{\tau}(\phi)$ the moduli space of discs in the class $\phi$ satisfying 
	$$\dfrac{du}{ds} + K_{s,t,\tau}(\dfrac{du}{dt}) =0 $$
	We note that for fixed $\tau$, the definition of the continuation map with respect to $(\phi_{t,\tau}, K_{s,t,\tau})$ given above can be restated succinctly as counting Maslov index 0 discs in the moduli spaces $\mathcal{M}_{\tau}(\phi)$. For any $\tau$, the homotopy $h$ induces an identification between homotopy classes of discs $\pi_{2}^{0}(\boldsymbol{x}, \boldsymbol{y}) \cong \pi_{2}^{\tau}(\boldsymbol{x}, \boldsymbol{y})$. Using this identification, we may define for each $\phi \in \pi_{2}^{0}(\boldsymbol{x}, \boldsymbol{y})$ the moduli space
	
	\begin{equation}
	\mathcal{M}^{h}(\phi) = \bigcup_{\tau \in I} \mathcal{M}_{\tau}(\phi) \times \{\tau\}
	\end{equation}
	
	For a generic choice of homotopy $h$, this is a manifold of dimension $\mu(\phi) + 1$. We use this moduli space to define a chain homotopy $H^{h}: CF^{-}_{K_{s,0}}(\Sigma, \balpha_{0}, \bbeta_{0}) \rightarrow CF^{-}_{K_{s,1}}(\Sigma, \balpha_{1}, \bbeta_{1})$ between $\Gamma_{(\phi_{t,0}, K_{s,t,0})}$ and $\Gamma_{(\phi_{t,1}, K_{s,t,1})}$ associated with the homotopy $h$. For $\boldsymbol{x} \in \mathbb{T}_{\balpha} \cap \mathbb{T}_{\bbeta}$ we set
	
	$$H^{h}( [\boldsymbol{x}, i ]) = \sum_{\boldsymbol{y} \in \mathbb{T}_{\balpha_{1}} \cap \mathbb{T}_{\bbeta_{1}}}  \sum_{\substack{\phi \in \pi_{2}^{0}(\boldsymbol{x}, \boldsymbol{y}) \\ \mu(\phi) = -1}} \# (\mathcal{M}^{h}(\phi)) [\boldsymbol{y}, i - n_{p}(\phi)].$$
	
	To see that this is a chain homotopy, we will consider the ends of the moduli spaces $\mathcal{M}^{h}(\psi)$ for $\psi$ with Maslov index $\mu(\psi) = 0$. Since such spaces $\mathcal{M}^{h}(\psi)$ are smooth 1 dimensional manifolds for generic choices of almost complex structure data, and since they are orientable, the signed count of the ends is zero for any choice of orientation. 
	
	The ends can be partitioned into three types: those corresponding to $\tau = 0$, those corresponding to $\tau = 1$, and those corresponding to strips breaking off for values $ 0  < \tau <1$. For the ends corresponding to $\tau =0$, the contribution to the count of the ends is given by the count of the zero dimensional moduli space $\# \mathcal{M}_{\tau=0}(\psi)$. Modulo signs, this is precisely the count occurring in the definition of $\Gamma_{(\phi_{t,0}, K_{s,t,0})}$. For $\tau =1$, the contribution to the count of the ends is similarly given by $\# \mathcal{M}_{\tau=1}(\psi)$, which is the count occurring in the definition of $\Gamma_{(\phi_{t,1}, K_{s,t,1})}$, modulo signs. We will discuss the signed contributions below. Finally, the ends corresponding to strip breaking come from the space
	
	$$ \left( \coprod_{\substack{\phi \ast \phi ' = \psi \\ \mu(\phi) = 0 , \mu(\phi ') = 1} } \mathcal{M}^{h}(\phi) \times \widehat{\mathcal{M}}(\phi ') \right) \coprod \left( \coprod_{\substack{\phi ' \ast \phi  = \psi \\ \mu(\phi) = 0 , \mu(\phi ') = 1}}  \widehat{\mathcal{M}}(\phi ') \times \mathcal{M}^{h}(\phi) \right)$$
	
	Supposing the orientations on the moduli spaces $\mathcal{M}^{h}$ are chosen to be coherent with respect to preglueings of strips, the count of the terms in the first parentheses is precisely the count occurring in the composition $\partial_{0}^{-} \circ H^{h}$, while the count of the terms in the second parentheses is precisely the count occurring in $H^{h} \circ (\partial_{1})^{-}$. Here $\partial_{0}^{-}$ indicates the differential on $CF^{-}_{K_{s,0}}(\Sigma, \balpha_{0}, \bbeta_{0})$ and $(\partial_{1})^{-}$ indicates the differential on $CF^{-}_{K_{s,1}}(\Sigma, \balpha_{1}, \bbeta_{1})$.
	
	Finally, we note that we may arrange for the spaces $\mathcal{M}^{h}(\phi)$ to be coherently oriented such that the total signed count of the ends of $\mathcal{M}^{h}(\psi)$ is given by 
	
	$$ 0 = \Gamma_{(\phi_{t,0}, K_{s,t,0})} - \Gamma_{(\phi_{t,1}, K_{s,t,1})} - ( (\partial_{1})^{-} \circ H^{h} + H^{h} \circ \partial_{0}^{-})$$
	
	Indeed, we have
	\begin{equation}
	\label{Moduli2}
	\mathcal{M}^{h}(\psi) = \bigcup_{\tau \in I} \mathcal{M}_{\tau}(\psi) \times \{\tau\} = \{ (u, \tau) \in C^{\infty}(I \times \mathbb{R}, \text{Sym}^{g}(\Sigma)) \times I | u \in \mathcal{M}_{\tau}(\psi) \}
	\end{equation}
	so for each homotopy class $\psi$ we may choose orientations on $\mathcal{M}_{\tau=0}(\psi)$ fitting together coherently, and obtain induced orientations on the spaces $\mathcal{M}^{h}(\psi)$ via the product structure in Equation \eqref{Moduli2}. Such an induced orientation will enjoy the property that the restrictions to the ends at $\tau =0$ and $\tau =1$ yield the counts $- \# \mathcal{M}_{\tau=0}(\psi)$  and $+ \# \mathcal{M}_{\tau=1}(\psi) $ respectively. We omit the technical details of this argument, and refer the interested reader to the proof of Lemma \ref{ParametrizedOrientation}, where an analogous argument dealing with holomorphic triangles is spelled out in detail. We have thus shown that a generic homotopy $h =(\phi_{t,\tau}, K_{s,t,\tau})$ between $(\phi_{t,0}, K_{s,t,0})$ and $(\phi_{t,1}, K_{s,t,1})$ induces a chain homotopy between $\Gamma_{(\phi_{t,0}, K_{s,t,0})}$ and $\pm \Gamma_{(\phi_{t,1}, K_{s,t,1})}$. 
	
	Finally, we note that since the homotopy $h$ is constant in $\tau$ for $t=0$ and $t=1$, the chain homotopy $H^{h}$, defined with respect to the orientations on $\mathcal{M}^{h}(\phi)$ specified above, is a chain homotopy between the continuation maps $\Gamma_{(\phi_{t,0}, K_{s,t,0})}$ and $\Gamma_{(\phi_{t,1}, K_{s,t,1})}$, which both take the form:
	$$CF^{-}_{K_{s, 0,0}=K_{s, 0,1}} (\Sigma, \balpha_{0}, \bbeta_{0}, z, \mathfrak{o}_{\balpha_{0}, \bbeta_{0}}) \rightarrow CF^{-}_{K_{s,1,0}=K_{s,1,1}}(\Sigma, \balpha_{1}, \bbeta_{1}, z, \mathfrak{o}_{\balpha_{1}, \bbeta_{1}})$$ and are defined with respect to the same coherent orientation systems on their domains, and the same coherent orientation systems on their targets. In particular, in the case of interest (ie Equation \eqref{SufficientCondition}) we may choose orientations on $\mathcal{M}_{\tau =0} = \mathcal{M}^{d_{t}} $ so that $d_{J_{s},J_{s}'} \sim \Gamma_{d_{t}}$ (which we established is possible earlier), which together with the above remarks establishes Equation \eqref{SufficientCondition}. This completes the proof of the theorem. \qedhere

\end{proof}

Finally, we relegate the proof of axiom 4, simple handleswap invariance, to Section \ref{Moduli} below. Given a simple handleswap in $\mathcal{G}_{\text{man}}$,
\begin{equation*}
\begin{tikzcd}
H_1  \arrow{rd}{e} &     \\
H_3 \arrow{u}{g} & H_2 \arrow{l}{f}   \\
\end{tikzcd} 
\end{equation*}
we will show that the composition of the induced maps in the category of transitive systems in the projectivized homotopy category yields the identity. We recall from Definition \ref{SimpleHandleswap} that here $H_{i} = (\Sigma \# \Sigma_{0}, \balpha_{i}, \bbeta_{i})$ are isotopy diagrams, $e$ is a strong $\balpha$-equivalence, $f$ is a strong $\bbeta$-equivalence, and $g$ is a diffeomorphism of isotopy diagrams. 

\begin{thm}[cf. Theorem 9.30 in \cite{Naturality}]\label{HandleswapInvariance}
	Let $(\{H_{i}\}, e, f, g)$ be data defining a simple handleswap as above. For the weak Heegaard invariants $CF^{\circ}$ defined in Definition \ref{WeakHeegaardFloer1}, the induced maps $g_{*}:= CF^{\circ}(g)$, $\Phi_{e} := CF^{\circ}(e)$, and $\Phi_{f} := CF^{\circ}(f)$ satisfy
	$$g_{*} \circ \Phi_{f} \circ \Phi_{e} = \text{Id}_{ CF^{-}(H_1)}$$
	Thus the weak Heegaard invariants $CF^{\circ}: \mathcal{G}_{\text{man}} \rightarrow \text{Trans}(P(\text{Kom}(\mathbb{Z}[U]\text{-} \text{Mod}) ))$ satisfy simple handleswap invariance.
\end{thm}

\begin{cor}\label{HandleswapInvarianceHomology}
	The weak Heegaard invariants $HF^{-}: \mathcal{G}_{\text{man}} \rightarrow P(\mathbb{Z}[U]\text{-}\text{Mod})$ satisfy simple handleswap invariance.
\end{cor}
Theorem \ref{HandleswapInvariance} and Corollary \ref{HandleswapInvarianceHomology} will establish Theorem \ref{StrongHeegaard2} and Corollary \ref{StrongHeegaard}, which by Section \ref{ProjectiveNaturalityFromStrong} also establishes Theorem \ref{Functoriality}.

\section{Simple Handleswap Invariance}\label{Moduli}
In this section we prove Theorem \ref{HandleswapInvariance}. The key result which will need to be established is the integral analog of a triangle count proved in \cite[Proposition 9.31]{Naturality}. We will consider the pointed genus two Heegaard triple diagram $\mathcal{T}_{0}$ shown in Figure \ref{triplediagram} (compare the diagrams in Figure \ref{handleswappic}). Given any triple diagram $\mathcal{T}$ we will show that triangle maps on the stabilized diagram $\mathcal{T} \# \mathcal{T}_{0}$, endowed with a sufficiently stretched neck, are determined by triangle maps on the unstabilized diagram $\mathcal{T}$. 

\begin{figure}[h!]
	\centering
	\begin{tikzpicture}
	\node[anchor=south west,inner sep=0] (image) at (0,0) {\includegraphics[width=0.9\textwidth]{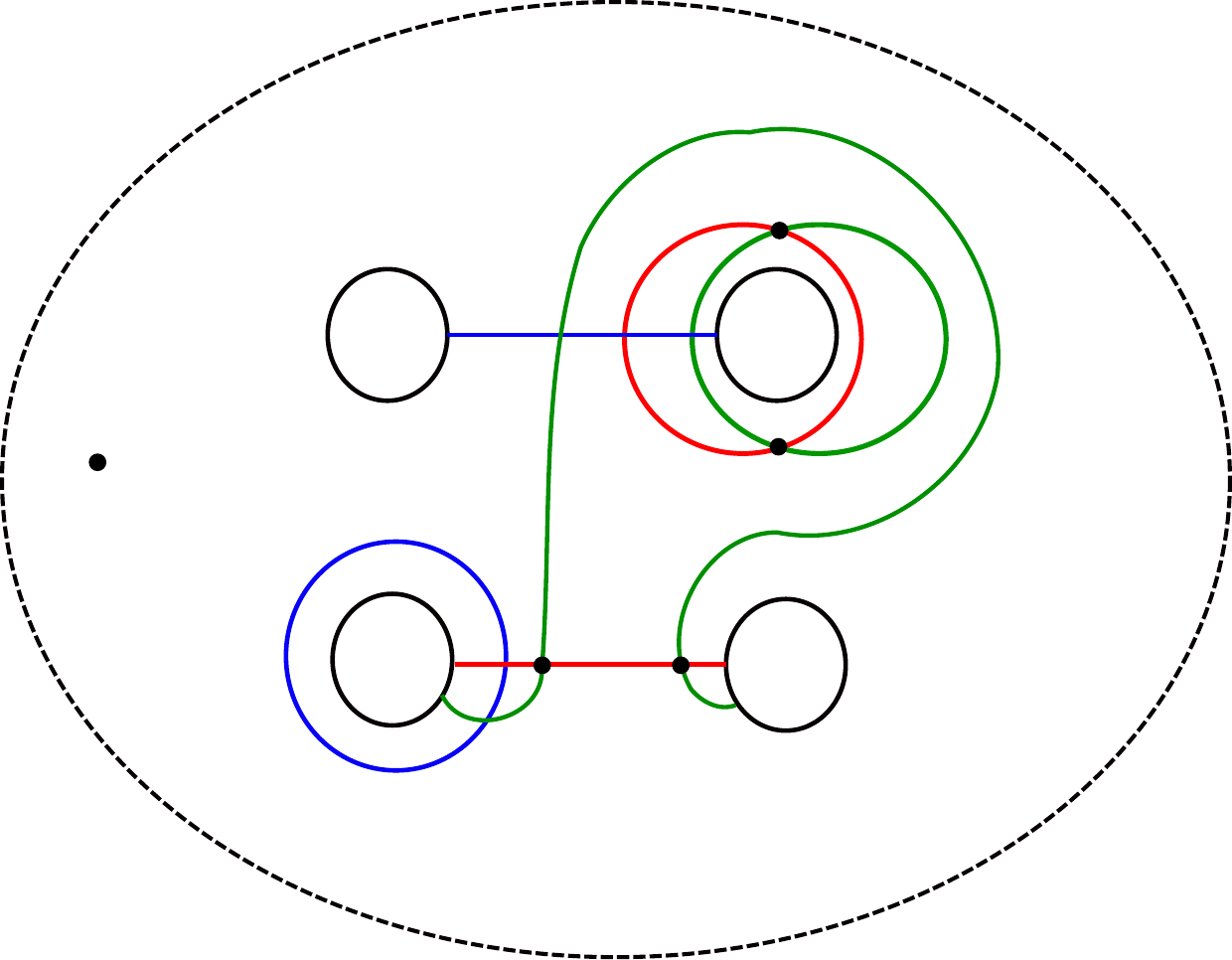}};
	\begin{scope}[x={(image.south east)},y={(image.north west)}]
	\node[text= blue] at (.3,.16) {$\beta_{1}$};
	\node[text= blue] at (.41,.68) {$\beta_{2}$};
	\node[text= OliveGreen] at (.8,.5) {$\alpha_{1} '$};
	\node[text= OliveGreen] at (.7,.795) {$\alpha_{2} '$};
	\node[text= red] at (.52,.27) {$\alpha_{1}$};
	\node[text= red] at (.55,.515) {$\alpha_{2}$};
	\node at (.47,.34) {$\theta_{1}^{-}$};
	\node at (.53,.34) {$\theta_{1}^{+}$};
	\node at (.6305,.49) {$\theta_{2}^{+}$};
	\node at (.6305,.80) {$\theta_{2}^{-}$};
	\node at (.11,.5) {$p_{0}$};
	\node at (.32,.31) {\Huge F};
	\node at (.64,.31) {\reflectbox{ \Huge F}};
	\node at (.32,.65) {\Huge R};
	\node at (.64,.65) {\reflectbox{ \Huge R}};
	\end{scope}
	\end{tikzpicture}
	\caption{The pointed triple diagram $\mathcal{T}_{0}$, with the curves $\balpha_{0}' = (\alpha_{1}', \alpha_{2}')$, $\balpha_{0} = (\alpha_{1}, \alpha_{2})$, $\bbeta_{0} = (\beta_{1}, \beta_{2})$, and the $\theta$ intersection points, labeled. }
	\label{triplediagram}
\end{figure}

We now fix some notation regarding the intersection points in the triple diagram $\mathcal{T}_{0} = (\Sigma, \balpha_{0}', \balpha_{0}, \bbeta_{0}, p_{0})$. We write $\mathbb{T}_{\balpha_{0} } \cap \mathbb{T}_{\bbeta_{0}} = \{\ba\}$ , $\mathbb{T}_{\balpha_{0} '} \cap \mathbb{T}_{\bbeta_{0}} = \{\bb\}$, and $\mathbb{T}_{\balpha_{0} '} \cap \mathbb{T}_{\balpha_{0}} = \{\theta_{1}^{+} \theta_{2}^{+} , \theta_{1}^{+} \theta_{2}^{-}, \theta_{1}^{-} \theta_{2}^{+}, \theta_{1}^{-} \theta_{2}^{-} \}$. Here the intersection points $\theta_{1}^{\pm} \in \alpha_{1}' \cap \alpha_{1}$ and $\theta_{2}^{\pm} \in \alpha_{2}' \cap \alpha_{2}$ are those labeled in Figure \ref{triplediagram}. We write $\btheta := \theta_{1}^{+} \theta_{2}^{+} $. We will show:

\begin{prop}(compare \cite[ Proposition 9.31]{Naturality}) \label{trianglecount}
	Fix a strongly $\mathfrak{s}$-admissible Heegaard triple $\mathcal{T} = (\Sigma, \boldsymbol{\alpha '}, \boldsymbol{\alpha}, \boldsymbol{\beta}, p)$, and consider the diagram $\mathcal{T} \# \mathcal{T}_{0}$, where $\mathcal{T}_{0} = (\Sigma, \balpha_{0}', \balpha_{0}, \bbeta_{0}, p_{0})$ is the diagram in Figure \ref{triplediagram} and the connect sum is taken at the basepoints $p$ and $p_{0}$. Then for a generic and sufficiently stretched almost complex structure there is a coherent orientation system $\mathfrak{o}_{\mathcal{T}_{0}}$ on $\mathcal{T}_{0}$, which together with any coherent orientation system $\mathfrak{o}_{\mathcal{T}}$ on $\mathcal{T}$ induces a coherent orientation system $\mathfrak{o}_{\mathcal{T} \# \mathcal{T}_{0}}$ on $\mathcal{T} \# \mathcal{T}_{0}$. Furthermore, with respect to these orientations,
	$$ \mathcal{F}_{\mathcal{T} \# \mathcal{T}_{0}}((\boldsymbol{x} \times \boldsymbol{\Theta}) \otimes (\boldsymbol{y} \times \boldsymbol{a}), \mathfrak{s}) = \pm \mathcal{F}_{\mathcal{T}}(\boldsymbol{x} \otimes \boldsymbol{y}, \mathfrak{s}) \times \boldsymbol{b} $$
	for any $\boldsymbol{x} \in \mathbb{T}_{\balpha '} \cap \mathbb{T}_{\balpha}$ and $\boldsymbol{y} \in \mathbb{T}_{\balpha} \cap \mathbb{T}_{\bbeta}$.
\end{prop}

In fact when we prove handleswap invariance the diagram $\mathcal{T}_{0}$ and the triangle count just stated will be relevant only to the consideration of the strong $\balpha$-equivalence involved in the statement. We will need an analogous result which pertains to the strong $\bbeta$-equivalence map occurring in the statement. We now state the precise result we will need for this. Let $\mathcal{T}_{0} '= (\Sigma_{0}, \balpha_{0}', \bbeta_{0},\bbeta_{0} ', p_{0})$ denote the pointed genus two triple diagram shown in Figure \ref{triplediagram2}, where $\balpha_{0} ' = \{\alpha_{1}', \alpha_{2}' \}$, $\bbeta_{0} ' = \{\beta_{1}, \beta_{2} \}$ and $\bbeta_{0} ' = \{\beta_{1}', \beta_{2}' \}$ (again compare the diagrams in Figure \ref{handleswappic}). 

\begin{figure}[h!]
	\centering
	\begin{tikzpicture}
	\node[anchor=south west,inner sep=0] (image) at (0,0) {\includegraphics[width=0.9\textwidth]{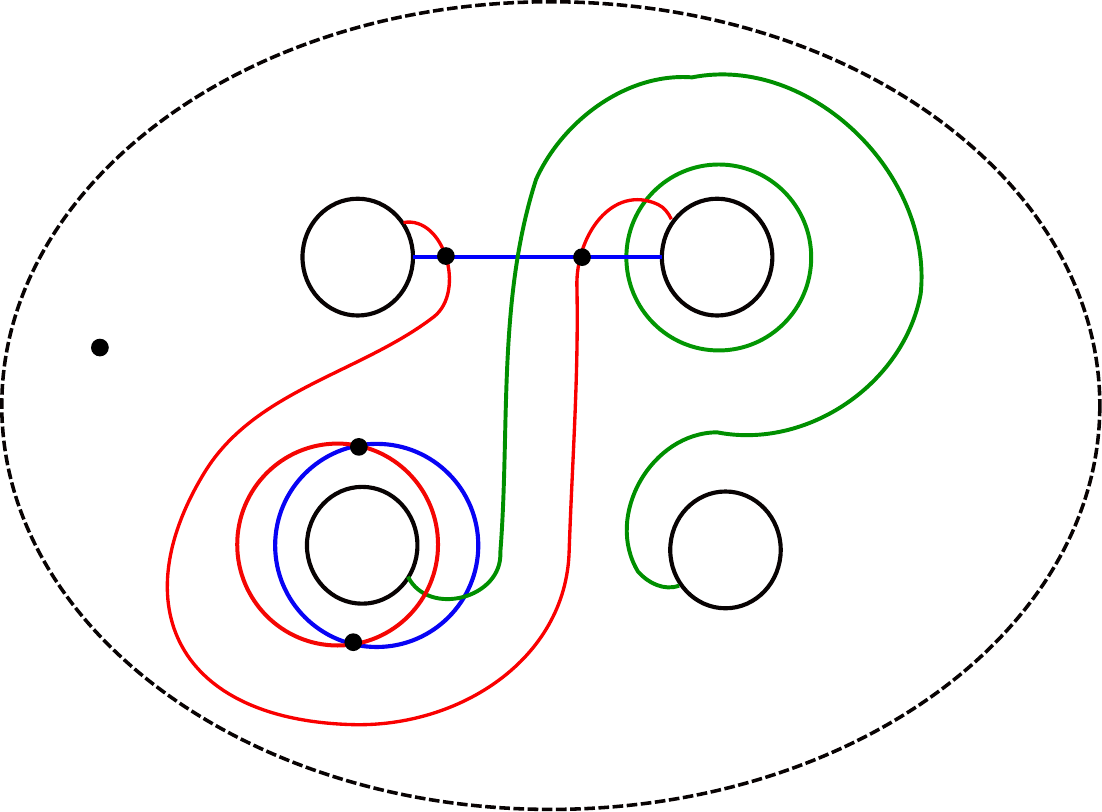}};
	\begin{scope}[x={(image.south east)},y={(image.north west)}]
	\node[text= blue] at (.42,.19) {$\beta_{1}$};
	\node[text= blue] at (.44,.65) {$\beta_{2}$};
	\node[text= OliveGreen] at (.81,.5) {$\alpha_{1} '$};
	\node[text= OliveGreen] at (.72,.79) {$\alpha_{2} '$};
	\node[text= red] at (.23,.19) {$\beta_{1} '$};
	\node[text= red] at (.12,.28) {$\beta_{2} '$};
	\node at (.34,.50) {$(\theta_{1}^{-})'$};
	\node at (.33,.15) {$(\theta_{1}^{+})'$};
	\node at (.516,.73) {$(\theta_{2}^{+})'$};
	\node at (.43,.75) {$(\theta_{2}^{-})'$};
	\node at (.12,.55) {$p_{0}$};
	\node at (.33,.33) {\Huge F};
	\node at (.66,.33) {\reflectbox{ \Huge F}};
	\node at (.33,.685) {\Huge R};
	\node at (.66,.685) {\reflectbox{ \Huge R}};
	\end{scope}
	\end{tikzpicture}
	\caption{The pointed triple diagram $\mathcal{T}_{0}'$, with the curves $\balpha_{0}' = (\alpha_{1}', \alpha_{2}')$, $\bbeta_{0} = (\beta_{1}, \beta_{2})$, and $\bbeta_{0} ' = (\beta_{1} ', \beta_{2} ')$, and the $\theta '$ intersection points, labeled.}
	\label{triplediagram2}
\end{figure}

We further fix the following notation for intersection points in the diagram: we let $\mathbb{T}_{\balpha_{0}'} \cap \mathbb{T}_{\bbeta_{0}} = \{\textbf{b}\}$, $\mathbb{T}_{\balpha_{0}'} \cap \mathbb{T}_{\bbeta_{0}'} = \{\textbf{c}\}$, and $\btheta '$ denote the generator in $\mathbb{T}_{\bbeta_{0}} \cap \mathbb{T}_{\bbeta_{0} '}$ with the highest relative grading. Let $\mathcal{T} ' = (\Sigma, \balpha', \bbeta, \bbeta ',p)$ be another pointed Heegaard triple, and consider the diagram $\mathcal{T}' \# \mathcal{T}_{0}'$, where the connect sum is taken at the basepoints $p$ and $p_{0}$. Then we will have an analogous triangle count:

\begin{prop}(compare \cite[ Proposition 9.32]{Naturality}) \label{trianglecount2}
	Fix a strongly $\mathfrak{s}$-admissible Heegaard triple $\mathcal{T} ' = (\Sigma, \boldsymbol{\alpha '}, \boldsymbol{\beta}, \boldsymbol{\beta '}, p)$, and consider the diagram $\mathcal{T}' \# \mathcal{T}_{0}'$ as above. Then for a generic and sufficiently stretched almost complex structure there is a coherent orientation system $\mathfrak{o}_{\mathcal{T}_{0}'}$ on $\mathcal{T}_{0}'$, which together with any coherent orientation system $\mathfrak{o}_{\mathcal{T}'}$ on $\mathcal{T}'$ induces a coherent orientation system $\mathfrak{o}_{\mathcal{T}' \# \mathcal{T}_{0}'}$ on $\mathcal{T} ' \# \mathcal{T}_{0} '$. Furthermore, with respect to these orientations, 
	$$ \mathcal{F}_{\mathcal{T} ' \# \mathcal{T}_{0}'}((\boldsymbol{x} \times \boldsymbol{b}) \otimes (\boldsymbol{y} \times \btheta '), \mathfrak{s}) = \pm \mathcal{F}_{\mathcal{T}'}(\boldsymbol{x} \otimes \boldsymbol{y}, \mathfrak{s}) \times \boldsymbol{c} $$
	for any $\boldsymbol{x} \in \mathbb{T}_{\balpha '} \cap \mathbb{T}_{\bbeta}$ and $\boldsymbol{y} \in \mathbb{T}_{\bbeta} \cap \mathbb{T}_{\bbeta '}$.
\end{prop}
We will prove Proposition \ref{trianglecount} in the following subsection. Since a nearly identical proof can be used to establish Proposition \ref{trianglecount2}, we omit the proof of that result. We now assume Propositions \ref{trianglecount} and \ref{trianglecount2} and use them to establish Theorem \ref{HandleswapInvariance}.

\begin{proof}[Proof of Theorem \ref{HandleswapInvariance}]
	We consider a simple handleswap $(H_{1}, H_{2}, H_{3}, e, f, g)$ as in Definition \ref{SimpleHandleswap}. We first note that to prove the statement about transitive systems appearing in Theorem \ref{HandleswapInvariance}, it will suffice to find representatives $\mathcal{H}_{1}$, $\mathcal{H}_{2}$, and $\mathcal{H}_{3}$ for the isotopy diagrams, and show that for these representatives we have
	$$g_{*} \circ \Phi_{f} \circ \Phi_{e} = \pm \text{Id}_{CF^{-}(\mathcal{H}_{1})}$$
	in $\text{Kom}(\mathbb{Z}[U]\text{-}\text{Mod})$, or equivalently
	$$g_{*} \circ \Phi_{f} \circ \Phi_{e} =  \text{Id}_{CF^{-}(\mathcal{H}_{1})}$$
	in $P(\text{Kom}(\mathbb{Z}[U]\text{-}\text{Mod}))$.
	Indeed, since each of the maps $\Phi_{e}$, $\Phi_{f}$, and $g_{*}$ above are contained in the morphisms $\Phi_{e}, \Phi_{f}$ and $g_{*}$ of the transitive systems $CF^{-}(H)$, by the results in Sections \ref{WeakHeegaardFloerInvariantsSection} and \ref{MainProof}, this monodromy relation will automatically yield corresponding monodromy relation for all such triangles.
	
	Let $\mathcal{H}_{1} = (\Sigma \# \Sigma_{0}, \balpha_{1}, \bbeta_{2})$ be a representative for the first isotopy diagram in the collection of data specifying the simple handleswap. By definition, $\mathcal{H}_{1}$ decomposes as $\mathcal{H} \# \mathcal{H}_{0}$, where $\mathcal{H} =(\Sigma, \balpha, \bbeta)$ and $\mathcal{H}_{0} = (\Sigma_{0}, \balpha_{0}, \bbeta_{0})$ are as in Figure \ref{handleswappic} ($\mathcal{H}_{0}$ here is what we were denoting by $P \cap \mathcal{H}_{1}$ in Definition \ref{SimpleHandleswap}). 
	
	Fix two new curves $\balpha_{0}'$ on $\Sigma_{0}$ which are related to $\balpha_{0}$ as in the diagram $\mathcal{T}_{0}$ in the statement of Proposition \ref{trianglecount}. Fix also a collection of curves $\balpha ' $ on $\Sigma$ which are obtained by performing a small Hamiltonian isotopy on the curves in $\balpha$. The second isotopy diagram $H_{2}$ can then be represented as $H_{2} = (\Sigma \# \Sigma_{0}, \balpha ' \cup \balpha_{0} ', \bbeta \cup \bbeta_{0} )$, and the morphism associated to the strong $\balpha$-equivalence $e$ is given by the triangle map $\Phi_{e} := \Psi_{\bbeta \cup \bbeta_{0}}^{\balpha \cup \balpha_{0} \rightarrow \balpha ' \cup \balpha_{0} '}$. We note that our choices of representatives for the isotopy diagrams $H_{1}$ and $H_{2}$ ensure that the strong equivalence map of Definition \ref{StrongEquivalenceMap} applied to these representatives is computed using only a single triangle map, as opposed to a composition of triangle maps and continuation maps. As in the notation of Proposition \ref{trianglecount}, we set $\mathbb{T}_{\balpha_{0}} \cap \mathbb{T}_{\bbeta_{0}} = \{\boldsymbol{a}\}$ and $\mathbb{T}_{\balpha_{0} '} \cap \mathbb{T}_{\bbeta_{0}} = \{\boldsymbol{b}\}$. We then have for any $\boldsymbol{y} \times \boldsymbol{a} \in \mathbb{T}_{\balpha \cup \balpha_{0}} \cap \mathbb{T}_{\bbeta \cup \bbeta_{0}}$:
	\begin{align*}
	\Phi_{e}(\boldsymbol{y} \times \boldsymbol{a}) &= \Psi_{\bbeta \cup \bbeta_{0}}^{\balpha \cup \balpha_{0} \rightarrow \balpha ' \cup \balpha_{0} '}(\boldsymbol{y} \times \boldsymbol{a}) \\
	&= \mathcal{F}_{\balpha ' \cup \balpha_{0} ', \balpha \cup \balpha_{0}, \bbeta \cup \bbeta_{0}}(\btheta_{\balpha ' \cup \balpha_{0}', \balpha \cup \balpha_{0}} \otimes (\boldsymbol{y} \times \boldsymbol{a})) \\
	&= \mathcal{F}_{\balpha ' \cup \balpha_{0} ', \balpha \cup \balpha_{0}, \bbeta \cup \bbeta_{0}}((\btheta_{\balpha ', \balpha} \times \btheta) \otimes (\boldsymbol{y} \times \boldsymbol{a})) \\
	&= \pm \mathcal{F}_{\balpha ', \balpha, \bbeta}(\btheta_{\balpha ', \balpha} \times \boldsymbol{y}) \times \boldsymbol{b} \\
	&= \pm \Gamma_{\bbeta}^{\balpha \rightarrow \balpha'}(\boldsymbol{y}) \times \boldsymbol{b}
	\end{align*}
	Here we have used Proposition \ref{trianglecount} in the second to last equality, and Lemma \ref{ContinuationTriangleRelation} in the last equality. 
	
	We perform the analogous calculation for the strong $\bbeta$-equivalence. Fix two new curves $\bbeta_{0} '$ on $\Sigma_{0}$ which are related to $\bbeta_{0}$ as in the diagram $\mathcal{T}_{0}'$ in the statement of Proposition \ref{trianglecount2}.  Fix also a collection of curves $\bbeta ' $ on $\Sigma$ which are obtained by performing a small Hamiltonian isotopy on the curves in $\bbeta$. The third isotopy diagram $H_{3}$ can then be represented as $H_{3} = (\Sigma \# \Sigma_{0}, \balpha ' \cup \balpha_{0} ', \bbeta ' \cup \bbeta_{0} ')$, and the morphism associated to the strong $\bbeta$-equivalence $f$ is given by the triangle map $\Phi_{f} := \Psi_{\bbeta \cup \bbeta_{0} \rightarrow \bbeta ' \cup \bbeta_{0} '}^{\balpha ' \cup \balpha_{0} '}$. As in the notation of Proposition \ref{trianglecount2}, we set $\mathbb{T}_{\balpha_{0}'} \cap \mathbb{T}_{\bbeta_{0} '} = \{\textbf{c}\}$. By the same sequence of computations as in the previous case we then have for any $\boldsymbol{x} \times \boldsymbol{b} \in \mathbb{T}_{\balpha ' \cup \balpha_{0}'} \cap \mathbb{T}_{\bbeta \cup \bbeta_{0}}$:
	\begin{align*}
	\Phi_{f}(\boldsymbol{x} \times \boldsymbol{b}) &= \Psi_{\bbeta \cup \bbeta_{0} \rightarrow \bbeta ' \cup \bbeta_{0} '}^{ \balpha ' \cup \balpha_{0} '}(\boldsymbol{x} \times \boldsymbol{b}) \\
	&= \mathcal{F}_{\balpha ' \cup \balpha_{0} ', \bbeta \cup \bbeta_{0}, \bbeta ' \cup \bbeta_{0}'}((\boldsymbol{x} \times \boldsymbol{b}) \otimes \btheta_{\bbeta \cup \bbeta_{0}, \bbeta ' \cup \bbeta_{0}'}) \\
	&= \mathcal{F}_{\balpha ' \cup \balpha_{0} ', \bbeta \cup \bbeta_{0}, \bbeta ' \cup \bbeta_{0}'}( (\boldsymbol{x} \times \boldsymbol{b}) \otimes (\btheta_{\bbeta , \bbeta '} \times \btheta) ) \\
	&= \pm \mathcal{F}_{\balpha ', \bbeta, \bbeta '}(\boldsymbol{x} \times \btheta_{\bbeta, \bbeta'}) \times \boldsymbol{c} \\
	&= \pm \Gamma_{\bbeta \rightarrow \bbeta '}^{\balpha'}(\boldsymbol{x}) \times \boldsymbol{c}
	\end{align*}
	This time we have used Proposition \ref{trianglecount2} in the second to last equality, and again used Lemma \ref{ContinuationTriangleRelation} in the last equality.
	
	We note that in the collection of representatives for the isotopy diagrams in a simple handleswap one could leave the $\balpha$ and $\bbeta$ curves unchanged throughout the handleswap, which would necessitate the diffeomorphism $g$ restricting to the identity on $\Sigma$. Here we have altered $\balpha$ and $\bbeta$ slightly, so that the strong $\balpha$-equivalence and strong $\bbeta$-equivalence maps could each be computed via a single triangle map $\Psi$. Since our alteration of the curves $\balpha$ and $\bbeta$ on $\Sigma$ came from small Hamiltonian isotopies, we can however still ensure that for our representatives for the handleswap the diffeomorphism $g$ is isotopic to the identity when restricted to $\Sigma$. Furthermore, since $g$ is part of a simple handleswap it must satisfy $g(\balpha ') = g(\balpha)$ and $g(\bbeta ') = g(\bbeta)$. Thus, by definition of the maps induced by diffeomorphisms of diagrams, we have
	
	$$g_{*}(\boldsymbol{z} \times \boldsymbol{c}) = (g|_{\Sigma})_{*}(\boldsymbol{z}) \times \boldsymbol{a}$$
	for all $(\boldsymbol{z} \times \boldsymbol{c}) \in \mathbb{T}_{\balpha ' \cup \balpha_{0} '} \cap \mathbb{T}_{\bbeta ' \cup \bbeta_{0} '}$. 
	
	Putting these formulas for each of the induced maps together, we find that

	\begin{align*}
	g_{*} \circ \Phi_{f} \circ \Phi_{e}(\boldsymbol{y} \times \boldsymbol{a}) &= \left( g_{*} \circ \Psi_{\bbeta \cup \bbeta_{0} \rightarrow \bbeta ' \cup \bbeta_{0} '}^{ \balpha ' \cup \balpha_{0} '} \circ \Psi_{\bbeta \cup \bbeta_{0}}^{\balpha \cup \balpha_{0} \rightarrow \balpha ' \cup \balpha_{0} '} \right) (\boldsymbol{y} \times \boldsymbol{a})  \\
	&= \pm \left( (g|_{\Sigma})_{*} \circ \Gamma_{\bbeta \rightarrow \bbeta '}^{\balpha '} \circ \Gamma_{\bbeta}^{\balpha \rightarrow \balpha '} \right) (\boldsymbol{y}) \times \boldsymbol{a}
	\end{align*}
	
	Since the restiction of $g$ to $\Sigma$ is isotopic to the identity, Theorem \ref{ContinuityAxiom} ensures 
	
	$$(g|_{\Sigma})_{*} \circ \Gamma_{\bbeta \rightarrow \bbeta '}^{\balpha '} \circ \Gamma_{\bbeta}^{\balpha \rightarrow \balpha '} \sim \pm \text{Id}_{CF^{-}(\mathcal{H})}$$
	
	We thus have
	\begin{align*}
	g_{*} \circ \Phi_{f} \circ \Phi_{e} &= \pm \left( (g|_{\Sigma})_{*} \circ \Gamma_{\bbeta \rightarrow \bbeta '}^{\balpha '} \circ \Gamma_{\bbeta}^{\balpha \rightarrow \balpha '} \right)  \otimes \text{Id}_{CF^{-}(\mathcal{H}_{0})} \\
	&\sim \pm \text{Id}_{CF^{-}(\mathcal{H})} \otimes \text{Id}_{CF^{-}(\mathcal{H}_{0})} \\
	&\sim \pm \text{Id}_{CF^{-}(\mathcal{H}_{1})},
	\end{align*}
	which by the remarks at the beginning of the proof completes the argument.
\end{proof}
Having established the implication (Proposition \ref{trianglecount} and Proposition \ref{trianglecount2} $\implies$ Theorem \ref{HandleswapInvariance}), we now turn towards proving Proposition \ref{trianglecount}.

We employ the strategy used in \cite{Naturality} for proving the analog of Proposition \ref{trianglecount} appearing there. We import many results exactly as they are stated there, while in a few cases we make small modifications in order to be able to apply their results. For the reader's convenience we provide statements of some results from \cite{Naturality}, and provide proofs of any imported results which must be modified slightly for our purposes. We also provide sketches of proofs of certain statements from \cite{Naturality} which we do not need to modify, but whose exposition we hope will aid in the readibility of this paper. 

In the remainder of this section we work in the cylindrical formulation of Heegaard Floer homology introduced by Lipshitz in \cite{Cylindrical}.

\subsection{Moduli Spaces of Triangles}
We begin by recalling some notation and terminology regarding holomormphic triangles in the cylindrical setting of Heegaard Floer homology (see \cite{Cylindrical}). We denote by $\Delta$ the subset of $\mathbb{C}$ shown in Figure \ref{Delta} below, which has three cylindrical ends modeled on $[0,1] \times [0, \infty)$. We will think of this region as a triangle with its vertices removed. We also introduce in the figure notation we will use to indicate the boundary components and ends of this region.

\begin{figure}[h!]
	\centering
	\begin{tikzpicture}
	\node[anchor=south west,inner sep=0] (image) at (0,0) {\includegraphics[width=0.3\textwidth]{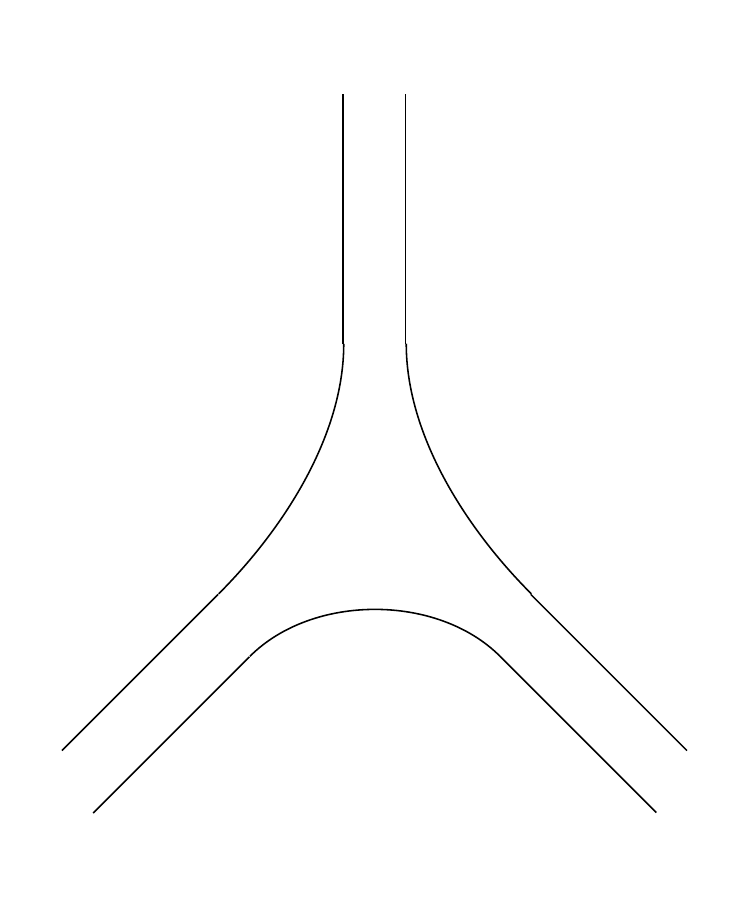}};
	\begin{scope}[x={(image.south east)},y={(image.north west)}]
	\node at (.5,.94) {\large $\nu_{\balpha ' \bbeta}$};
	\node at (.03,.12) {\large $\nu_{\balpha \bbeta}$};
	\node at (.99,.12) {\large $\nu_{\balpha ' \balpha}$};
	\node at (.5,.27) {\large $e_{\balpha}$};
	\node at (.32,.5) {\large $e_{\bbeta}$};
	\node at (.68,.5) {\large $e_{\balpha '}$};
	\end{scope}
	\end{tikzpicture}
	\caption{The region $\Delta$.}
	\label{Delta}
\end{figure}
We will consider almost complex structures $J$ on $\Sigma \times \Delta$ which satisfy the following conditions:
\begin{enumerate}[($J'1'$)]
	\item $J$ is tamed by the split symplectic form on $\Sigma \times \Delta$.
	\item On each component of $\Sigma \setminus (\boldsymbol{\alpha}' \cup \boldsymbol{\alpha} \cup \boldsymbol{\beta})$ there is at least one point at which $J = j_{\Sigma} \times j_{\Delta}$.
	\item On each cylindrical end $\Sigma \times [0,1] \times \mathbb{R}$ of $\Sigma \times \Delta$, there is a $2$-plane distribution $\eta$ on $\Sigma \times [0,1] \times \{0\}$ such that the restriction of $\omega$ to $\eta$ is non-degenerate, $J$ preserves $\eta$, and the restriction of $J$ to $\eta$ is compatible with $\omega$. Furthermore, $\eta$ is tangent to $\Sigma$ near $(\Sigma \times \{0,1\} \times \{0\}) \cup (\Sigma \times [0,1] \times \{0\})$.
	\item The planes $T_{d}(\{p\} \times \Delta)$ are complex lines of $J$ for all $(p,d) \in \Sigma \times \Delta$.
	\item There is an open set $U \subset \Delta$ containing $\partial \Delta \setminus \{ \nu_{\alpha ' \alpha}, \nu_{\alpha \beta}, \nu_{\alpha ' \beta} \}$ such that the planes $T_{p}( \Sigma \times \{d\})$ are complex lines of $J$ for all $(p,d)$ near $(\boldsymbol{\alpha}' \cup \boldsymbol{\alpha} \cup \boldsymbol \beta) \times \Delta$ and for all $(p,d) \in \Sigma \times U$.

\end{enumerate}
$J$-holomorphic curves in $\Sigma \times \Delta$ for almost complex structures $J$ of this sort enjoy the following property.

\begin{claim}[Lemma 3.1 in \cite{Cylindrical}] \label{Lemma9.37}
	Let $J$ be an almost complex structure on $\Sigma \times \Delta$ that satisfies the axioms $(J'1') - (J'5')$. If $u: S \rightarrow \Sigma \times \Delta$ is $J$-holomorphic and $\pi_{\Sigma} \circ u$ is nonconstant on a component $S_{0}$ of S, then $\pi_{\Sigma} \circ u |_{S_{0}}$ is an open map. Furthermore, there are coordinates near any critical point of $\pi_{\Sigma} \circ u |_{S_{0}}$  where $\pi_{\Sigma} \circ u$ takes the form $z \mapsto z^{k}$ for some $k > 0$. 
\end{claim} 
\noindent In fact, this result follows immediately from \cite[Theorem 7.1]{PositivityofIntersections}.

To understand Proposition \ref{trianglecount}, we will need to investigate the nature of triangle maps on the diagram $\mathcal{T} \# \mathcal{T}_{0}$. In the cylindrical setting, the notion of a holomorphic triangle in a Heegaard triple diagram takes the following form. 

\begin{defn}\label{Holomorphic Triangle}
	Let $\mathcal{T} =(\Sigma, \boldsymbol{\alpha ' }, \boldsymbol{ \alpha}, \boldsymbol{\beta})$ be a triple diagram, and set $d = | \boldsymbol{\alpha' } | = | \boldsymbol{\alpha} | = | \boldsymbol{\beta} |$. By a \emph{holomorphic triangle in the triple diagram $\mathcal{T}$} we will mean a $(j,J)$-holomorphic map $u: S \rightarrow \Sigma \times \Delta$ satsifying:
	
	\begin{enumerate}[(M1)]
		\item $(S,j)$ is a (possibly nodal) Riemann surface with boundary and $3d$ punctures on $\partial S$. 
		\item $u$ is locally nonconstant. 
		\item $u(\partial S) \subset ( \boldsymbol{\alpha '} \times e_{\balpha '}) \cup ( \boldsymbol{\alpha } \times e_{\balpha }) \cup ( \boldsymbol{\beta } \times e_{\bbeta})$.
		\item $u$ has finite energy.
		\item For each $i \in \{1, \ldots, d  \} $ and $\sigma \in \{ \balpha ' , \balpha, \bbeta  \}$, the preimage $u^{-1}( \sigma_{i} \times e_{\sigma})$ consists of exactly one component of the punctured boundary of $S$.
		\item As one approaches the punctures of $\partial S$, the map $u$ converges to a collection of intersection points on the Heegaard triple in the cylindrical ends of $\Sigma \times \Delta$.
	\end{enumerate}
	We will often ask holomorphic triangles to satisfy the following additional two requirements:
	
	\begin{enumerate}[(M1), resume]
		\item $\pi_{\Delta} \circ u$ is nonconstant on each component of $S$.
		\item $S$ is smooth, and $u$ is an embedding. 
	\end{enumerate}
\end{defn}

Unless otherwise specified, we will use the term holomorphic triangle to refer to maps satisfying axioms $(M1) - (M6)$, and explicitly note when we are considering curves satisfying the additional axioms $(M7)$ and $(M8)$.

For any homology class $\psi$ of triangles on a Heegaard triple diagram $\mathcal{T}$, we will denote by $\mathcal{M}(\psi)$ the moduli space of holomorphic triangles on $\mathcal{T}$ in the homology class $\psi$. Given a Riemann surface $S$, we will indicate by $\mathcal{M}(\psi, S)$ the subspace of $\mathcal{M}(\psi)$ consisting of holomorphic triangles with source $S$.

To obtain the triangle count we are after on a sufficiently stretched copy of $\mathcal{T} \# \mathcal{T}_{0}$, we will need to understand compactifications of these moduli spaces of triangles. These compactifications allow for a weaker notion of triangle which we refer to as broken:

\begin{defn}\label{Broken Holomorphic Triangle}
	Let $\mathcal{T} =(\Sigma, \boldsymbol{\alpha ' }, \boldsymbol{ \alpha}, \boldsymbol{\beta})$ and $d$ be as above. We say that a collection of $(j,J)$-holomorphic curves $BT = ( u_{1}, v_{1}, \ldots, v_{n}, w_{1}, \ldots, w_{m})$ is a \emph{broken holomorphic triangle on $\mathcal{T}$ representing the homology class $\psi$} if
	\begin{enumerate}[(BT1)]
		\item $u_{1}$ is a curve mapping to $\Sigma \times \Delta$ satisfying $(M1)$ and $(M3) - (M6)$.
		\item $v_{i}$ are curves mapping to $\Sigma \times I \times \mathbb{R}$ which satisfy the analogs of $(M1)$ and $(M3) - (M6)$, each representing some homology class of strips in one of the diagrams $(\Sigma, \boldsymbol{\alpha}, \boldsymbol{\alpha}')$, $(\Sigma, \boldsymbol{\alpha}', \boldsymbol{\beta})$ or  $(\Sigma, \boldsymbol{\alpha}, \boldsymbol{\beta})$.
		\item The $w_{i}$ are curves from Riemann surfaces with $d$ boundary components and a single puncture on each boundary component, and which map to $\Sigma \times I \times \mathbb{R} \coprod \Sigma \times \Delta$. For each $i$, the  boundary components of the curve $w_{i}$ all map to a single set of attaching curves. 
		\item The total homology class of the curves in $BT$ is equal to $\psi$.
	\end{enumerate}
\end{defn}

With this notion in hand, we can state the following compactness result which describes the behavior of triangles on $\mathcal{T} \# \mathcal{T}_{0}$ as we stretch the neck:

\begin{prop}[Proposition 9.40 in \cite{Naturality}] \label{Prop9.40}
	Let $\psi \# \psi_{0}$ be a homology class of triangles on $(\Sigma \# \Sigma_{0}) \times \Delta$, and $u_{T_{i}}$ be a sequence of holomorphic triangle representatives for $\psi \# \psi_{0}$  on $(\Sigma \# \Sigma_{0}) \times \Delta$, with respect to almost complex structures $J(T_{i})$ for neck lengths $T_{i} \rightarrow \infty$. Then there is a subsequence which converges to a triple $(U, V, U_{0})$ where $U$ and $U_{0}$ are broken holomorphic triangles on $\Sigma \times \Delta$ and $\Sigma_{0} \times \Delta$  representing $\psi$ and $\psi_{0}$ respectively, and $V$ is a collection of holomorphic curves on the neck regions $S^{1} \times \mathbb{R} \times \Delta$ or $S^{1} \times \mathbb{R} \times [0,1] \times \mathbb{R}$ which are asymptotic to (possibly multiply covered) Reeb orbits $S^{1} \times \{d \}$  for $d \in \Delta$ or $d \in [0,1] \times \mathbb{R}$.
	
\end{prop}

\begin{rem}
	More precisely, the asymptotic condition on the curves appearing in $V$ in Proposition \ref{Prop9.40} above has the following meaning. By a ``Reeb orbit" in this context, we mean a periodic orbit $\gamma$ of the vector field $\dfrac{d}{d \theta}$ on $S^{1} \times \mathbb{R} \times \Delta$ or $S^{1} \times \mathbb{R} \times I \times \mathbb{R}$, where $\theta$ is the coordinate on $S^{1}$. The curves $v$ in $V$ have as sources punctured Riemann surfaces. Let $S$ be a connected component of such a source, $q$ a puncture of $S$, and $v: S \rightarrow S^{1} \times \mathbb{R} \times \Delta$. Write $(\theta, r,z)$ for coordinates on the target. Then $v$ is asymptotic to $\gamma$ at $q$ if:
	
	\begin{enumerate}
		\item There is a neighborhood $U$ of $q$ in $S$ and a biholomorphic diffeomorphism $ \phi: U \cong S^{1} \times (0, \infty)$. Write $(x,y)$ for coordinates on $S^{1} \times (0,\infty)$.
		\item $r \circ v \circ \phi^{-1} \rightarrow \infty$ as $y \rightarrow \infty$
		\item $ (\theta, z) \circ v \circ \phi^{-1}(x,y) \rightarrow \gamma(x)$ as $y \rightarrow \infty$ as maps $S^{1} \rightarrow S^{1} \times \Delta$ in $\mathcal{C}^{\infty}_{\text{loc}}$.
	\end{enumerate}

\end{rem}

\subsection{Matched Moduli Spaces and Orientations}
\label{Orientations}

Fix a triple diagram $\mathcal{T} = (\Sigma, \boldsymbol{\alpha '}, \boldsymbol{\alpha}, \boldsymbol{\beta})$ and a point $p \in \Sigma \setminus (\boldsymbol{\alpha '} \cup \boldsymbol{\alpha} \cup \boldsymbol{\beta}))$. Let $u: S \rightarrow \Sigma \times \Delta$ be a $J$-holomorphic curve satisfying (M1)-(M6), for some almost complex structure $J$ on $\Sigma \times \Delta$ satisfying $(J'1')$-$(J'5')$. Then $u$ is locally non-constant by condition (M2), so by Lemma \ref{Lemma9.37} $\pi_{\Sigma} \circ u$ is an open map on each component of $S$, and takes the form $ z \mapsto z^{k}$ near any critical point. Thus $(\pi_{\Sigma} \circ u)^{-1}(p)$ is a finite set of points. Furthermore, using property $(J'4')$ of the almost complex structure $J$, positivity of complex intersections for $J$-holomorphic curves (See eg \cite{PositivityofIntersections} or \cite{MS}) ensures that all intersections between $p \times \Delta$ and the image of $u$ are positive.

We will write $(\pi_{\Sigma} \circ u)^{-1}(p) = \{x_{1}, \dots, x_{n_{p}(u)}\} \in \text{Sym}^{n_{p}(u)}(S)$, and define 
\begin{equation*}\label{MatchingCondition}
\rho^{p}(u) := \{\pi_{\Delta} \circ u(x_{1}), \dots, \pi_{\Delta} \circ u(x_{n_{p}(u)})\} \in \text{Sym}^{n_{p}(u)}(\Delta)
\end{equation*}
We remark that our notation involving set braces is somewhat misleading, as there may of course be repetitions among the points $x_{i}$ in the symmetric product, corresponding to intersection points occuring with positive multiplicity greater than 1. 

To understand the triangle count, we will be concerned with holomorphic triangles $u$ for which $\rho^{p}(u)$ takes prescribed values. As a first step towards understanding the moduli spaces of such triangles, Juh\'asz, Thurston and Zemke show that, for any prescribed value outside the fat diagonal, such a triangle is somewhere injective.

\begin{claim}[Lemma 9.45 in \cite{Naturality}] \label{Lemma9.45}
	Let $(\Sigma, \boldsymbol{\alpha'}, \boldsymbol{\alpha}, \boldsymbol{\beta},p)$ be a triple diagram, and $\boldsymbol{d} \in \text{Sym}^{k}(\Delta) \setminus \text{Diag}(\Delta)$. If $u: S \rightarrow \Sigma \times \Delta$ is a $J$-holomorphic curve satsifying $(M1)-(M6)$ for an almost complex structure satisfying $(J'1')-(J'5')$, which furthermore has $\rho^{p}(u) = \boldsymbol{d}$, then every component of $u$ is somewhere injective. 
\end{claim}
Fix a Heegaard triple diagram $\mathcal{T} = (\Sigma, \balpha ', \balpha, \bbeta, p)$ and a homology class of triangle $\psi$, with $n_{p}(\psi) =k$. Given a subset $X \subset \text{Sym}^{k}(\Delta)$, we let
$$\mathcal{M}(\psi, S, X) = \{u \in \mathcal{M}(\psi,S) | \rho^{p}(u) \in X \}$$

and

$$\mathcal{M}(\psi, X) = \{u \in \mathcal{M}(\psi) | \rho^{p}(u) \in X \}.$$

Using techniques similar to those used in the standard setting, Juh\'asz, Thurston and Zemke prove the following result, which shows that generically these matched moduli spaces are smooth manifolds.

\begin{prop}[Proposition 9.47 in \cite{Naturality}] \label{Prop9.47}
	Let $(\Sigma, \boldsymbol{\alpha'}, \boldsymbol{\alpha}, \boldsymbol{\beta})$ be a triple diagram, and fix a point $p \in \Sigma \setminus (\boldsymbol{\alpha'} \cup \boldsymbol{\alpha} \cup \boldsymbol{\beta})$. Suppose $X \subset \text{Sym}^{k}(\Delta)$ for some $k \in \mathbb{N}$ is a nonempty submanifold that does not intersect the fat diagonal. Furthermore, suppose that for every $x \in X$, the $k$-tuple $x$ has no coordinate in the open set $U \subset \Delta$ from $(J'5')$. Then, for a generic choice of almost complex structure $J$, the set $\mathcal{M} ( \psi, S, X)$ is a smooth manifold of dimension 
	
	$$ \text{ind}(\psi, S) - \text{codim}(X) $$
	where $\text{ind}(\psi,S)$ denotes the Fredholm index of the linearized $\bar{\partial}$ operator at any representative $u: S \rightarrow \Sigma \times \Delta$ for $\psi$. For $X = \text{Sym}^{k}(\Delta)$, the same statement holds near any curve $u$ that has no component $T$ on which $\pi_{\Delta} \circ u |_{T}$ is constant and has image in $U$, and such that all components of $u$ are somewhere injective. 
	
\end{prop}

It will be important for our purposes to note that these moduli spaces are also orientable when they are smoothly cut out, which follows in a straightforward manner from the framework in which the proof of the previous proposition is carried out. We now provide a sketch of the argument.

\begin{claim}\label{OrientabilityofMatchedModuliSpace}
	For $J$ and $X$ satisfying the hypotheses of Proposition \ref{Prop9.47}, with $X \subset \text{Sym}^{k}(\Delta)$ furthermore assumed to be an orientable submanifold, $\mathcal{M} ( \psi, S, X)$ is orientable. 
\end{claim}

\begin{proof}
	Forgetting the matching condition (ie\ taking $X = \text{Sym}^{k}(\Delta)$) we consider $\mathcal{M} ( \psi, S, \text{Sym}^{k}(\Delta) )= \mathcal{M}(\psi, S)$. By \cite[Proposition 6.3 and Section 10.3]{Cylindrical}, whenever this space is transversely cut out it is an orientable smooth manifold. 
	
	For the case when $X \neq  \text{Sym}^{k}(\Delta)$, we briefly recall how one can establish the existence of a smooth manifold structure on $\mathcal{M}(\psi, S, X)$, as in the proof of \cite[Proposition 9.47]{Naturality}. Consider the map $\rho^{p}: \mathcal{M}(\psi, S) \rightarrow \text{Sym}^{k}(\Delta)$. To obtain the smooth manifold structure on $\mathcal{M}(\psi, S, X)$, one considers the universal moduli space $\mathcal{M}_{\text{univ}}^{\ell}(\psi, S)$. This consists of triples $(u, j, J)$, where $j$ is a $C^{\ell}$ complex structure on $S$, $J$ is a $C^{\ell}$ almost complex structure on $\Sigma \times \Delta$ satisying conditions $(J'1') - (J'5')$, and $u$ is a $(j,J)$-holomorphic map $u: S \rightarrow \Sigma \times \Delta$ in the homology class $\psi$, which furthermore satisfies certain regularity conditions (see \cite[pg 968]{Cylindrical}). It is shown in the proof of Proposition \ref{Prop9.47}, using the technique of \cite[Proposition 3.7]{Cylindrical}, that the universal moduli space $\mathcal{M}_{\text{univ}}^{\ell}(\psi, S)$ is a Banach manifold and the evaluation map $\rho^{p}: \mathcal{M}_{\text{univ}}^{\ell}(\psi, S) \rightarrow \text{Sym}^{k}(\Delta)$ is a submersion at all triples $(u, j, J)$ for which $\rho^{p}(u)$ is not in the fat diagonal. Thus for $X$ missing the fat diagonal, the universal matched moduli space $\mathcal{M}_{\text{univ}}^{\ell}(\psi, S,X) := (\rho^{p})^{-1}(X)$ is a Banach manifold. One can then apply the Sard-Smale theorem to the Fredholm map $\pi: \mathcal{M}_{\text{univ}}^{\ell}(\psi, S,X) \rightarrow \mathcal{J^{\ell}}$ to obtain a regular value $J \in J^{\ell}$ so that $\mathcal{M}^{\ell}(\psi, S, X) = \pi^{-1}(J)$ is a smooth manifold. Finally, one uses an approximating bootstrapping argument to obtain the same result for $C^{\infty}$ complex structures. More precisely, one obtains that for a generic choice of $J$ the space $\mathcal{M}(\psi, S)$ is a smooth manifold and the map
	$$\rho^{p}: \mathcal{M}(\psi, S) \rightarrow \text{Sym}^{k}(\Delta)$$
	is transverse to $X$. Thus for $X$ missing the fat diagonal $\mathcal{M}(\psi, S, X):= (\rho^{p})^{-1}(X)$ is a smooth manifold.
	
	Fixing $u \in \mathcal{M}(\psi, S, X)$ we have
	$$T_{u} \mathcal{M}(\psi, S) \cong T_{u} \mathcal{M}(\psi, S, X) \oplus N_{u}$$
	where $N$ is any choice of orthogonal complement.  Since $\mathcal{M}(\psi, S)$ is orientable, it will suffice to show $N$ is orientable to establish that $\mathcal{M}(\psi, S, X)$ is orientable. Since $\rho^{p}$ is transverse to $X$, we have 
	$$ d\rho^{p}( T_{u} \mathcal{M}(\psi, S)) + T_{\rho^{p}(u)} X = T_{\rho^{p}(u)} \text{Sym}^{k}(\Delta).$$
	Since $(d \rho^{p})^{-1}(TX) = T\mathcal{M}(\psi, S, X)$, the two equations above yield a direct sum decomposition
	$$ d \rho^{p}(N_{u}) \oplus T_{\rho^{p}(u)} X \cong T_{\rho^{p}(u)} \text{Sym}^{k}(\Delta).$$
	Finally, since $X$ and $\text{Sym}^{k}(\Delta)$ are orientable, and $d \rho^{p} |_{N}$ is an ismorphism on each fiber, the last equation establishes orientability of the complement $N$. Thus $\mathcal{M}(\psi, S, X)$ is orientable, as desired.
\end{proof}

We now turn to an investigation of the behavior of orientations on these moduli spaces. We recall again the notion of coherent orientation systems, and now provide the precise definitions in the cylindrical setting, as we will need them in some of our computations. We begin with the moduli space of holomorphic \textit{strips} in a homology class $A \in \pi_{2}(\boldsymbol{x},\boldsymbol{y})$, denoted $\mathcal{M}^{A}$, on some Heegaard (double) diagram $\mathcal{H} = (\Sigma, \balpha,\bbeta)$. We set  $\widehat{\mathcal{M}}^{A} =\mathcal{M}^{A} / \mathbb{R}$. As noted above, these moduli spaces are orientable whenever they are smoothly cut out by \cite[Proposition 6.3]{Cylindrical}. There this is shown by trivializing the determinant line bundle of the virtual index bundle of the linearized $\bar{\partial }$-equation. In fact, this line bundle is trivialized over a larger auxiliary space of curves which are not necessarily holomorphic, which we denote by $\mathcal{B}^{A}$, rather than over $\mathcal{M}^{A}$. We ask for the trivializations of these determinant lines $\mathcal{L}$ over $\mathcal{B}^{A}$ to satisfy the following compatibility under glueing.

\begin{defn}
	\label{coherentstrips}
	Given a Heegaard diagram $\mathcal{H}$, homology classes of strips $A,A'$ which are adjacent on the diagram (ie\ $A \in \pi_{2}(\boldsymbol{x},\boldsymbol{y})$, $A' \in \pi_{2}(\boldsymbol{y},\boldsymbol{z})$), and maps $u:S \rightarrow \Sigma \times I \times \mathbb{R}$ and $u':S' \rightarrow \Sigma \times I \times \mathbb{R}$ representing $A$ and $A'$ respectively, one can preglue the positive corners of $u$ to the negative corners of $u'$ (see \cite[Appendix A]{Cylindrical} for one such construction). In fact, there is a 1 parameter family of such preglueings $(u \natural_{r} u' : S \natural_{r} S' \rightarrow  \Sigma \times I \times \mathbb{R})$ in the class  $A + A'$, defined for sufficiently large values of the parameter $r$ . One can show that this map preserves the analogs of $(M1)$, $(M3)$ and $(M4)$ for strips, and the asymptotic conditions one asks of the strips. Denote the collection of maps of the form  $S \rightarrow \Sigma \times I \times \mathbb{R}$ in a given homology class $A$ which furthermore satisfy $(M1)$, $(M3)$, $(M4)$, and the asymptotic conditions by $\mathcal{B}^{A}(S)$.  We say a choice of orientations for all $\widehat{\mathcal{M}}^{A}$, specified by a collection of nonvanishing sections $\mathfrak{o}_{\mathcal{H}} = \mathfrak{o}_{\balpha, \bbeta} = \{\mathfrak{o}^{A}\}$ of $\mathcal{L}$ over all of the $\widehat{\mathcal{M}}^{A}$, is a \emph{coherent orientation system on $\mathcal{H}$} if the induced map of determinant lines covering the map $\natural_{r}: \mathcal{B}^{A}(S) \times \mathcal{B}^{A'}(S') \times (R, \infty) \rightarrow \mathcal{B}^{A+A'}(S \natural_{r} S')$  satisfies $(\natural_{r})_{*} ( \mathfrak{o}^{A} \times \mathfrak{o}^{A'}) = + \mathfrak{o}^{A +A'}$.

\end{defn}

That such coherent orientation systems exist is shown in numerous places.  One construction sufficient for our purposes can be found in \cite[Section 6]{Cylindrical}.

In the case of holomorphic triangles, the moduli spaces $\mathcal{M}(\psi)$ are also orientable. For a collection of orientations on $\mathcal{M}(\psi)$ for all homology classes $\psi$ of triangles in a triple diagram, we will consider a related notion of coherence.

\begin{defn}\label{coherenttriangles}
	Given a Heegaard triple diagram $\mathcal{T}$, we will say a choice of orientations for $\mathcal{M}^{\psi_{\balpha, \bbeta}}$, $\mathcal{M}^{\psi_{\bbeta, \bgamma}}$, $\mathcal{M}^{\psi_{\balpha, \bgamma}}$, and $\mathcal{M}(\psi)$ (for $\psi_{\balpha, \bbeta}$, $\psi_{\bbeta, \bgamma}$ and $\psi_{\balpha, \bgamma}$ ranging over all classes of strips in the respective double diagrams, and $\psi$ ranging over all classes of triangles in the triple diagram) specified by a collection of sections $\mathfrak{o}_{\mathcal{T}} = \{ \mathfrak{o}_{\balpha, \bbeta, \bgamma}, \mathfrak{o}_{\balpha,\bbeta}, \mathfrak{o}_{\bbeta, \bgamma}, \mathfrak{o}_{\balpha, \bgamma} \}$ is a \emph{coherent orientation system of triangles}, if each collection of orientations of the moduli spaces of strips on the respective double diagrams are coherent, and all possible pregluings of triangles with strips satisfy the analogous glueing condition. 
\end{defn}

Following \cite[Section 6]{Cylindrical}, given a homology class of triangles $\psi$ on the triple diagram $\mathcal{T}$, let $T(\psi)$ denote the space of pairs $(u,j)$, where $u: S \rightarrow \Sigma \times \Delta$ is a curve in the class $\psi$ satisfying $(M1)$, $(M3)$ and $(M4)$, and $j$ is a complex structure on $S$. We declare two such pairs $(u: S \rightarrow \Sigma \times \Delta, j)$ and $(u': S' \rightarrow \Sigma \times \Delta, j')$ to be equivalent if there is a biholomorphism $\phi: (S, j) \rightarrow (S', j')$ such that the diagram

\begin{equation}\label{teichmuller}
\begin{tikzcd}
S \arrow{rd}{u} \arrow{rr}{\phi} &   & S ' \arrow{ld}{u'}  \\
& \Sigma \times \Delta &  \\
\end{tikzcd} 
\end{equation}
commutes. We denote the quotient of $T(\psi)$ by this equivalence relation by $\mathcal{B}(\psi)$.

Let $p: I \rightarrow \text{Sym}^{k}(\Delta)$ be an embedded path missing the fat diagonal. We consider the following moduli spaces of holomorphic triangles associated to homology classes $\psi_{0} \in \pi_{2}(\btheta,\ba,\bb)$ in the triple diagram $\mathcal{T}_{0}$ from Proposition \ref{trianglecount}:

\begin{equation}
\mathcal{M}_{I}^{\psi_{0}} = \mathcal{M}(\psi_{0}, p(I)) = \{  (u,t) | u \in \mathcal{M}(\psi_{0}) \text{ such that } \rho^{p}(u) \in p(t) \text{ for some } t \in I \}
\label{MI}
\end{equation}
and 
\begin{equation}
\mathcal{M}_{t}^{\psi_{0}} = \mathcal{M}(\psi_{0}, p(t)) = \{ u \in \mathcal{M}(\psi_{0}) \text{ such that }  \rho^{p}(u) \in p(t)  \}
\end{equation}

By Proposition \ref{Prop9.47}, for a generic choice of almost complex structure on $\Sigma_{0} \times \Delta$ the moduli spaces $\mathcal{M}_{I}^{\psi_{0}}$ are smooth manifolds of dimension $\mu(\psi_{0}) - \text{codim}(p(I))$. By Lemma \ref{Lemma950}, we have $\mu(\psi_{0}) = 2n_{p_{0}}(\psi_{0})$, so the expected dimension becomes $2n_{p_{0}}(\psi_{0}) - (2k-1)$. In particular, when $k=n_{p_{0}}(\psi_{0})$ the moduli space $\mathcal{M}_{I}^{\psi_{0}}$ is a smooth 1 manifold when it is transversely cut out. Similarly, the expected dimension of $\mathcal{M}_{t}^{\psi_{0}}$ is 0 when $k=n_{p_{0}}(\psi_{0})$. Finally, we define the spaces

$$\mathcal{M} = \coprod_{\substack{\psi_{0} \in \pi_{2}(\btheta, \ba,\bb) \\ n_{p_{0}}(\psi_{0})=k}} \mathcal{M}^{\psi_{0}}$$

$$\mathcal{M}_{I} = \coprod_{\substack{\psi_{0} \in \pi_{2}(\btheta, \ba,\bb) \\ n_{p_{0}}(\psi_{0})=k}} \mathcal{M}_{I}^{\psi_{0}}$$

$$\mathcal{M}_{t} = \coprod_{\substack{\psi_{0} \in \pi_{2}(\btheta, \ba,\bb) \\ n_{p_{0}}(\psi_{0})=k}} \mathcal{M}_{t}^{\psi_{0}} $$

We provide a schematic of these spaces and their relationships in Figure \ref{MatchedModuliSchematic}.

We note for the following arguments that by the remarks above $\mathcal{M}_{I}$ is a smooth manifold of dimension 1 for a generic choice of almost complex structure, and for each $t$ a (potentially different) generic choice of almost complex structure will ensure $\mathcal{M}_{t}$ is a smooth manifold of dimension $0$. We will denote by $\mathfrak{o}_{\mathcal{M}_{I}}$ and $\mathfrak{o}_{\mathcal{M}_{t}}$ nowhere zero sections of the bundles $\mathcal{L}_{I}$ and $\mathcal{L}_{t}$ respectively, which are the determinant line bundles of the virtual index bundles of the linearized equations defining these moduli spaces. We recall that such sections determine orientations of the moduli spaces.

\begin{figure}[h!]
	\centering
	\begin{tikzpicture}
	\node[anchor=south west,inner sep=0] (image) at (0,0) {\includegraphics[width=0.8\textwidth]{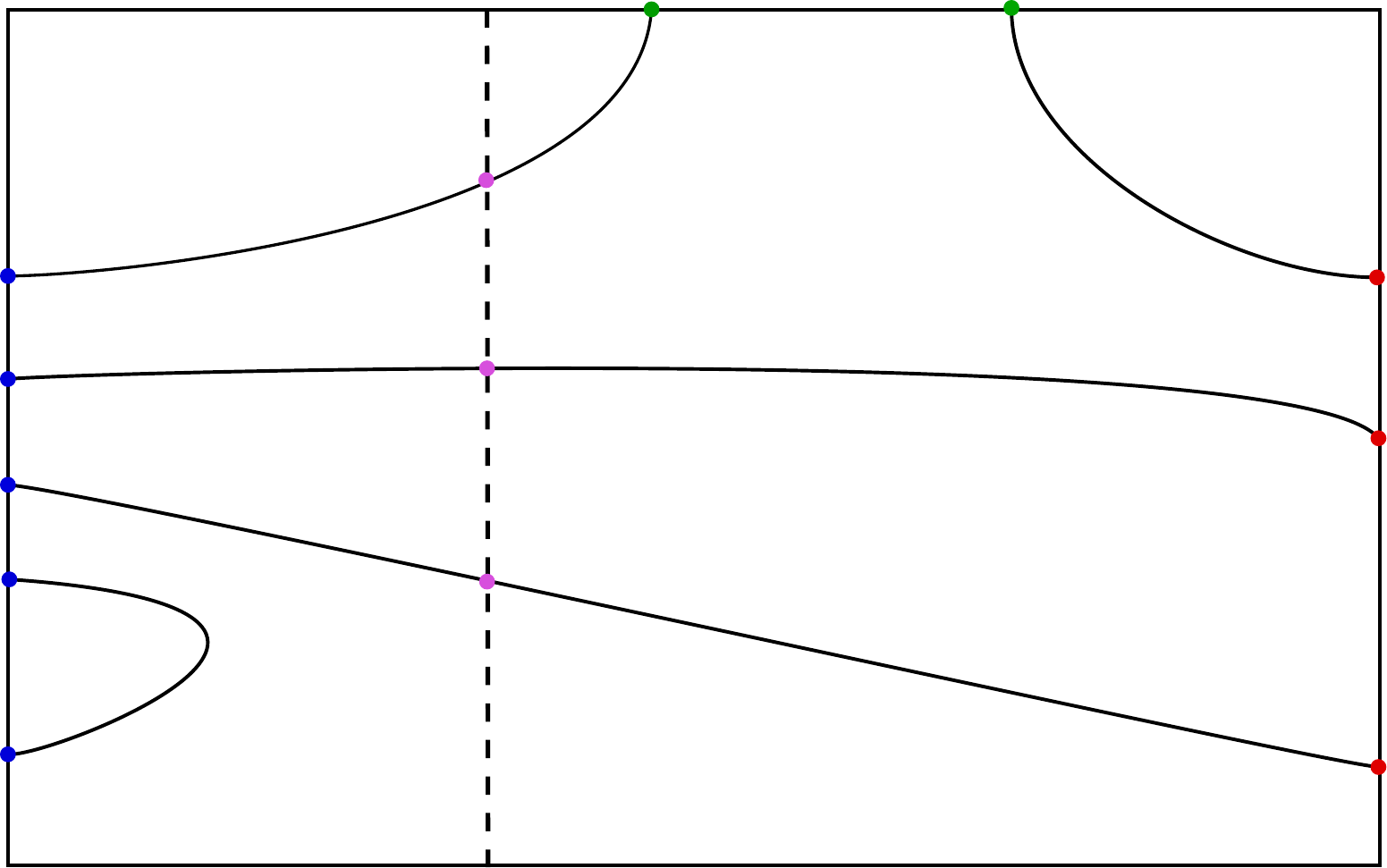}};
	\begin{scope}[x={(image.south east)},y={(image.north west)}]
	\node[text = blue] at (-.04,.56) {\large $\mathcal{M}_{0}$};
	\node[text = red] at (1.04,.5) {\large $\mathcal{M}_{1}$};
	\node[text = Purple] at (.31,.62) {\large $\mathcal{M}_{t}$};
	\node at (.6,.61) {\large $\mathcal{M}_{I}$};
	\node at (.28,.9) {\large $\mathcal{M}\times{t}$};
	\node at (-.07,.9) {\large $\mathcal{M}\times{0}$};
	\node at (1.07,.9) {\large $\mathcal{M}\times{1}$};
	
	\end{scope}
	\end{tikzpicture}
	\caption{A schematic of the space $\mathcal{M} \times I$ with $\mathcal{M}_{I}$ inside it. Vertical slices of the picture such as the vertical dashed line represent the spaces $\mathcal{M} \times {t}$, while the solid curves collectively represent the smooth moduli space $\mathcal{M}_{I}$. The left and right endpoints on $\mathcal{M}_{I}$ represent $\mathcal{M}_{0} \subset \mathcal{M} \times {0}$ and $\mathcal{M}_{1} \subset \mathcal{M} \times {1}$ respectively, while the endpoints of $\mathcal{M}_{I}$ on the top and bottom of the figure represent degenerations of triangles into broken triangles in the compactification.}
	\label{MatchedModuliSchematic}
\end{figure}

For arguments appearing later, we want to ensure we can achieve the following intuitively achievable constraints on our orientations:

\begin{claim}\label{ParametrizedOrientation}
	Let $\mathcal{M}_{I}$ and $\mathcal{M}_{t}$ be as above. Then there are coherent orientation systems $\mathfrak{o}_{\mathcal{M}_0}$ on $\mathcal{M}_{0}$, $\mathfrak{o}_{\mathcal{M}_1}$ on $\mathcal{M}_{1}$, and $\mathfrak{o}_{\mathcal{M}_I}$ on $\mathcal{M}_{I}$ such that $(\mathfrak{o}_{\mathcal{M}_I}) |_{\mathcal{M}_{0}} \cong - \mathfrak{o}_{\mathcal{M}_0}$ and $(\mathfrak{o}_{\mathcal{M}_I}) |_{\mathcal{M}_{1}} \cong \mathfrak{o}_{\mathcal{M}_1}$. 
\end{claim}

\begin{proof}
	The proof is an elaboration on that of Lemma \ref{OrientabilityofMatchedModuliSpace}. Consider again the universal moduli space of holomorphic maps $\mathcal{M}_{\text{univ}}$ consisting of triples $(u, j, J)$ satisfying the conditions as in the proof of Lemma \ref{OrientabilityofMatchedModuliSpace}. We consider the map
	$$\rho^{p} \times \text{id}: \mathcal{M}_{\text{univ}} \times I \rightarrow \text{Sym}^{k}(\Delta) \times I$$
	given by $(\rho^{p} \times \text{id})(u, j, J, t) = (\rho^{p}(u), t)$. This map is again a submersion when $\rho^{p}(u)$ is not in the fat diagonal, by \cite[Proposition 9.47]{Naturality}. Let 
	$$P = \{(p(t), t) | \text{t} \in I \}$$
	and note that $(\rho^{p}\times \text{id})^{-1}(P) = (\mathcal{M}_{\text{univ}})_{I}$, where here we have used the notation $(\mathcal{M}_{\text{univ}})_{I}$ to indicate the universal moduli space matched to $p(I)$ (as the notation is used in Equation \eqref{MI}). Since we are working with a path $p$ missing the fat diagonal, $\rho^{p} \times \text{id}$ is a submersion, and we have, as a consequence of the Sard Smale Theorem, parametric transversality: denoting by $\mathcal{M}_{J}$ the moduli space of holomorphic curves with respect to the almost complex structure $J$, $\rho^{p} \times \text{id}: \mathcal{M}_{J} \times I \rightarrow \text{Sym}^{k}(\Delta) \times I$ is transverse to $P$ for generic $J$. By \cite[Section 6]{Cylindrical} we may orient such $\mathcal{M}_{J}$, which in turn specifies a product orientation on $\mathcal{M}_{J} \times I$. Since $P$ is also oriented (by a fixed orientation for $I$), we then have for such $J$  that $(\rho^{p}\times \text{id})^{-1}(P) = \mathcal{M}_{I}$ inherits an orientation $\mathfrak{o}_{\mathcal{M}_I}$; furthermore, this orientation satisfies the boundary conditions
	
	\begin{equation}\label{orientationconvention}
	(\mathfrak{o}_{\mathcal{M}_I}) |_{\mathcal{M}_{0}} =  -\mathfrak{o}_{\mathcal{M}_{0}}
	\end{equation} 
	and 
	
	\begin{equation}
	(\mathfrak{o}_{\mathcal{M}_I}) |_{\mathcal{M}_{1}} = \mathfrak{o}_{\mathcal{M}_{1}} ,
	\end{equation}  
	where $\mathfrak{o}_{\mathcal{M}_{0}}$ and $\mathfrak{o}_{\mathcal{M}_{1}}$ are the orientations coming from the previously fixed choice of orientation on $\mathcal{M}_{J}$, as desired. Finally, we note that by the same argument used to prove \cite[Lemma 10.10]{Cylindrical}, we may arrange for the orientation systems $\mathfrak{o}_{\mathcal{M}_{I}}$, $\mathfrak{o}_{\mathcal{M}_{0}}$ and $\mathfrak{o}_{\mathcal{M}_{1}}$ in the preceding paragraph to be enlarged to coherent systems in the sense of Definition \ref{coherenttriangles}. \qedhere 
	
\end{proof}

Having discussed the smooth manifold structure and a particular construction of coherent orientations on the matched moduli spaces of triangles on a triple diagram, we now state a glueing result from \cite{Naturality} which will allow us to relate these matched moduli spaces of triangles on the diagram $\mathcal{T}_{0}$ to the triangles on $\mathcal{T} \# \mathcal{T}_{0}$ we seek to count. We consider homology classes of triangles $\psi$ on an arbitrary pointed triple diagram $\mathcal{T}=(\Sigma, \balpha ', \balpha, \bbeta, p)$ and $\psi_{0}$ on the pointed diagram $\mathcal{T}_{0}= (\Sigma_{0}, \balpha_{0} ', \balpha_{0}, \bbeta_{0}, p_{0})$. We form the connected sum of the diagrams at the points $p$ and $p_{0}$, and consider the resulting homology class $\psi \# \psi_{0}$:

\begin{prop}[Proposition 9.49 in \cite{Naturality}] \label{Prop9.49}
	Let $u$ and $u_{0}$ be holomorphic triangles representing homology classes $\psi$ and $\psi_{0}$ in $\Sigma \times \Delta$ and $\Sigma_{0} \times \Delta$ respectively. Let $k = n_{p}(\psi) = n_{p_{0}}(\psi_{0})$, and suppose $\mu(u) =0$, $\mu(u_{0}) = 2k$, and $\rho^{p}(u) = \rho^{p_{0}}(u_{0}) \in \text{Sym}^{k}(\Delta) \setminus \text{Diag}^{k}(\Delta)$. Suppose further that $\mathcal{M}(\psi)$ and $\mathcal{M}(\psi_{0}, \rho^{p}(u))$ are transversely cut out near $u$ and $u_{0}$. Then there is a homeomorphism $h$ between $[0,1)$ and a neighborhood of $(u,u_{0})$ in the compactified 1-dimensional moduli space
	$$\overline{\bigcup_{T} \mathcal{M}_{J(T)}(\psi \# \psi_{0})}$$
	such that $h(u,u_{0}) = \{0 \}$
\end{prop}

Finally, the following three facts will also be useful in the proof of the triangle count of Proposition \ref{trianglecount}, so we state them here as lemmas for convenience in referencing.

\begin{claim}[Lemma 9.50 in \cite{Naturality}] \label{Lemma950}
	Consider the triple diagram $\mathcal{T}_{0}= (\Sigma_{0}, \boldsymbol{\alpha_{0} '}, \boldsymbol{\alpha_{0}}, \boldsymbol{\beta_{0}})$. If $\boldsymbol{x} \in \mathbb{T}_{\alpha_{0}'} \cap \mathbb{T}_{\alpha_{0}}$ and $\psi_{0} \in \pi_{2}(\boldsymbol{x}, \boldsymbol{a}, \boldsymbol{b})$, then 
	
	\begin{equation}
	\mu(\psi_{0}) = 2n_{p_{0}}(\psi_{0}) + \mu(\boldsymbol{x}, \boldsymbol{\Theta})
	\end{equation}
\end{claim}

\begin{claim}\label{hatdifferential}
	The differential on $\widehat{CF}(\Sigma_{0}, \boldsymbol{\alpha_{0} '}, \boldsymbol{\alpha_{0}}, p_{0},\mathfrak{o}_{\balpha_{0}', \balpha_{0}})$, defined with respect to the coherent orientation system $\mathfrak{o}_{\balpha_{0}', \balpha_{0}}$ specified in Lemma \ref{StandardHandleslide}, vanishes.
\end{claim}

\begin{proof}
	By \cite[Lemma 9.4]{Disks1} $\text{rank}_{\mathbb{Z}} (\widehat{HF}(\Sigma_{0}, \boldsymbol{\alpha_{0} '}, \boldsymbol{\alpha_{0}},p_{0},\mathfrak{o}_{\balpha_{0}', \balpha_{0}}))= 4$. By inspection $\text{rank}_{\mathbb{Z}} (\widehat{CF})=4$, so the differential must vanish. 
\end{proof}

\begin{claim}
	\label{quasi}
	The map 
	
	$$\Psi_{\beta_{0}} ^{\alpha_{0} \rightarrow \alpha_{0}'} : \widehat{CF}(\Sigma_{0}, \boldsymbol{\alpha_{0} }, \boldsymbol{\beta_{0}}, p_{0}) \rightarrow \widehat{CF}(\Sigma_{0}, \boldsymbol{\alpha_{0} '}, \boldsymbol{\beta_{0}}, p_{0})$$
	
	satisfies $\Psi_{\beta_{0}} ^{\alpha_{0} \rightarrow \alpha_{0}'}(\boldsymbol{a}) = \pm \boldsymbol{b}$.
	
\end{claim}

\begin{proof}
	By Lemma \ref{HandleslideMapProperties}, $\Psi_{\beta_{0}} ^{\alpha_{0} \rightarrow \alpha_{0}'}$ is a quasi-isomorphism. Since the two complexes in question are trivial of rank one over $\mathbb{Z}$, the quasi-isomorphism must be an isomorphism between trivial, rank one complexes over $\mathbb{Z}$, of which there are precisely two.
\end{proof}

\subsection{Counting Triangles}
We are now in position to prove the main triangle count, and conclude the proof of handleswap invariance.

\begin{proof}[Proof of Proposition \ref{trianglecount}]
	As we did in Sections \ref{WeakHeegaardFloerInvariantsSection} and \ref{MainProof}, we will consider the case of the chain complexes $CF^{-}$ in what follows in order to fix definitions, however we note that the proof carries over equally well for all variants $CF^{\circ}$.
	
	For an almost complex structure J which achieves transversality we have, by definition,
	
	$$ \mathcal{F}_{\mathcal{T} \# \mathcal{T}_{0}}((\boldsymbol{x} \times \boldsymbol{\Theta}) \otimes (\boldsymbol{y} \times \boldsymbol{a})) = \sum_{\bz} \sum_{\substack{A \in \pi_{2}(\boldsymbol{x} \times \boldsymbol{\Theta},\boldsymbol{y} \times \boldsymbol{a}, \boldsymbol{z} \times \bb ) \\ \mu(A) = 0}} (\#\mathcal{M}_{J}(A)) U^{n_{p}(A)} \cdot \bz  \times \bb$$
	
	and 
	
	$$ \mathcal{F}_{\mathcal{T}}(\boldsymbol{x} \otimes \boldsymbol{y}) \times \boldsymbol{b} = \left( \sum_{\bz} \sum_{\substack{A \in \pi_{2}(\bx,\by,\bz) \\ \mu(A) = 0}} (\#\mathcal{M}_{J}(A)) U^{n_{p}(A)} \cdot \bz \right) \times \bb$$

	To obtain the result we will count Maslov index 0 holomorphic triangles in the homology class $A$, for each generator $\bz \in \mathbb{T}_{\boldsymbol{\alpha '}} \cap \mathbb{T}_{\boldsymbol{\beta}}$ and class $A \in \pi_{2}(\bx \times \btheta, \by \times \ba, \bz \times \bb)$.
	
	Consider two homology classes of triangles $\psi \in \pi_{2}(\boldsymbol{x}, \boldsymbol{y}, \boldsymbol{z})$ on $\mathcal{T} = (\Sigma, \boldsymbol{\alpha '}, \boldsymbol{\alpha}, \boldsymbol{\beta},p)$ and $\psi_{0} \in \pi_{2}(\boldsymbol{\Theta}, \boldsymbol{a}, \boldsymbol{b})$ on $\mathcal{T}_{0}= (\Sigma_{0}, \boldsymbol{\alpha_{0} '}, \boldsymbol{\alpha_{0}}, \boldsymbol{\beta_{0}}, p_{0})$. If $n_{p}(\psi) = n_{p_{0}}(\psi_{0})$, so the classes match across the connect sum point, then the homology classes can be combined to give a class $\psi \# \psi_{0} \in \pi_{2}(\bx \times \btheta, \by \times \ba, \bz \times \bb)$.  Conversely, it is clear that any class $A \in \pi_{2}(\bx \times \btheta, \by \times \ba, \bz \times \bb)$ can be written uniquely as a connect sum of suitable classes with this matching condition. 
	
	So for any such homology class $A = \psi \# \psi_{0}$ with $\mu(A) = 0$, we aim to count Maslov index zero holomorphic representatives as we stretch the neck, i.e to count $ \# \mathcal{M}_{J(T_{i})}(\psi \# \psi_{0})$, where $J(T_{i})$ is a sequence of almost complex structures being stretched along the neck. To do so, suppose $u_{T_{i}}$ is a sequence of $J(T_{i})$-holomorphic triangles representing $\psi \# \psi_{0}$, where $\mu(\psi \# \psi_{0})=0$. We note here that by \cite[Theorem 4.1]{Sarkar} and Lemma \ref{Lemma950} we have $\mu(\psi \# \psi_{0}) = \mu(\psi) + \mu(\psi_{0}) -2n_{p}(\psi_{0}) = \mu(\psi) + \mu(\boldsymbol{\theta}, \boldsymbol{\theta}) = \mu(\psi)$. Hence $\mu(\psi) = 0$, and $\mu(\psi_{0}) = 2n_{p_{0}}(\psi_{0})$.
	
	By Proposition \ref{Prop9.40}, there is a subsequence of $u_{T_{i}}$ which converges to a triple $(U, V, U_{0})$ where $U$ is a broken holomorphic triangle in $\Sigma \times \Delta$ representing $\psi$, $U_{0}$ is a broken holomorphic triangle in $\Sigma_{0} \times \Delta$ representing $\psi_{0}$, and $V$ is a collection of holomorphic curves mapping into the neck regions that are asymptotic to (possibly multiply covered) Reeb orbits of the form $S^{1} \times \{d\}$.
	
	The proof will now proceed in steps as follows:
	
	\begin{enumerate}
		\item We will show U consists of a single holomorphic triangle $u$ with Maslov index zero, with $u$ satisfying (M1)-(M8), and potentially some number of constant holomorphic curves.
		\item We then show that $U_{0}$ consists of a single Maslov index $2n_{p_{0}}(\psi_{0})$ triangle $u_{0}'$, with $u_{0}'$ satsisfying (M1)-(M8) and $\rho^{p}(u) = \rho^{p_{0}}(u_{0})$, and potentially some number of constant holomorphic curves.
		\item We rule out the possibility of constant curves occurring in steps 1 and 2, and show that $V$ consists of a collection of trivial holomorphic cylinders.
		\item Using this knowledge of $(U,V,U_{0})$ and the glueing result, we reduce the proof to showing Lemma \ref{SingleDivisorModuliCount} below.
	\end{enumerate}
	
	In fact, the proofs of steps (1) through (3) given in \cite{Naturality} carry over exactly as they are stated there, so we will only carry out step (4).

	\textbf{Step 4}
	By steps (1)-(3), a sequence $u_{T_{i}}$ of $J(T_{i})$-holomorphic triangles representing $\psi \# \psi_{0}$  converges to a broken holomorphic triangle $(U,V,U_{0})$, where $U=u$ is a single holomomorphic triangle satisfying $\mu(u)=0$, $V$ is a collection of trivial holomorphic cylinders, $U_{0}$ is a single holomorphic triangle $u_{0}$ satisfying $\mu(u_{0}) = 2n_{p}(\psi)$, and $\rho^{p}(u)  =\rho^{p_{0}}(u_{0})$. By Proposition \ref{Prop9.49}, there is therefore a homeomorphic identification $h$ between a neighborhood of $(u,u_{0})$ in the compactified 1 dimensional moduli space
	
	$$ \overline{\bigcup_{T_{i}}\mathcal{M}_{J(T_{i})}(\psi \# \psi_{0})}$$
	and the interval $[0,1)$, such that $h(u,u_{0}) = \{0\}$. This yields an identification
	$$\mathcal{M}_{J(T_{i})}(\psi \# \psi_{0}) \cong \{ (u,u_{0}) \in \mathcal{M}(\psi) \times \mathcal{M}(\psi_{0}) | \rho^{p}(u) = \rho^{p}(u_{0})      \}$$ for sufficiently large $T_{i}$. We now fix $J_{T_{i}}$ for such a sufficiently large value of $T_{i}$, and drop this choice of almost complex structure from our notation. 
	
	Given coherent orientation systems $\mathfrak{o}_{\mathcal{T}}$ over $\mathcal{T}$ and $\mathfrak{o}_{\mathcal{T}_{0}}$ over $\mathcal{T}_{0}$, there is a coherent orientation system $\mathfrak{o}_{\mathcal{T} \# \mathcal{T}_{0}}$ with respect to which the signed count of the $0$ dimensional moduli space $\mathcal{M}(\psi \# \psi_{0})$ is given by 
	$$\# \mathcal{M}(\psi \# \psi_{0}) = \# \{ (u,u_{0}) \in \mathcal{M}(\psi) \times \mathcal{M}(\psi_{0}) | \rho^{p}(u) = \rho^{p}(u_{0})      \}.$$
	Indeed, given two homology classes of triangles $\psi$ on $\mathcal{T}$ and $\psi_{0}$ on $\mathcal{T}_{0}$, the glueing map $\natural$ (see \cite[Appendix A, page 1082]{Cylindrical} for the definition) used to identify the two moduli spaces is covered by a map of determinant lines $(\natural)_{\#}$ which can be used to produce an orientation $\mathfrak{o}_{\mathcal{T}\# \mathcal{T}_{0}}^{\psi \# \psi_{0}}$ over $\mathcal{M}(\psi \# \psi_{0})$ from orientations $\mathfrak{o}_{\mathcal{T}}^{\psi}$ over $\mathcal{M}(\psi)$ and $\mathfrak{o}_{\mathcal{T}_{0}}^{\psi_{0}}$ over $\mathcal{M}(\psi_{0})$. Similarly, for two homology classes of strips $A$ on $\mathcal{T}$ and $A_{0}$ on $\mathcal{T}_{0}$, the same procedure can be used to determine an orientation $\mathfrak{o}_{\mathcal{T} \# \mathcal{T}_{0}}^{A \# A_{0}}$ from $\mathfrak{o}_{\mathcal{T}}^{A}$ and $\mathfrak{o}_{\mathcal{T}_{0}}^{A_{0}}$. The fact that homology classes of strips and triangles on $\mathcal{T} \# \mathcal{T}_{0}$ are in bijective correspondence to pairs of homology classes of strips on $\mathcal{T}$ and $\mathcal{T}_{0}$ ensures that the coherent orientation systems $\mathfrak{o}_{\mathcal{T}}$ and $\mathfrak{o}_{\mathcal{T}_{0}}$ thus determine a single orientation system $\mathfrak{o}_{\mathcal{T} \# \mathcal{T}_{0}}$ over all classes of strips and triangles in the connect summed diagram  (ie the determinations for a particular class of triangle or strip on the summed diagram are not overspecified). That this induced orientation is coherent follows from the coherence of the two constituent orientations, along with the fact that glueing map $(\natural)_{\#}$ above commutes with the map $(\natural)_{*}$ appearing in Definition \ref{coherenttriangles}. More precisely, the coherence follows from these facts as
	\begin{align*}
	\mathfrak{o}_{\mathcal{T} \# \mathcal{T}_{0}}^{(\psi + A) \# (\psi_{0} + A_{0})} &:= (\natural)_{\#}(\mathfrak{o}_{\mathcal{T}}^{\psi + A} \times \mathfrak{o}_{\mathcal{T}_{0}}^{\psi_{0} + A_{0}})\\
	&= (\natural)_{\#}((\natural)_{*}(\mathfrak{o}_{\mathcal{T}}^{\psi} \times \mathfrak{o}_{\mathcal{T}}^{A}) \times (\natural)_{*}(\mathfrak{o}_{\mathcal{T}_{0}}^{\psi_{0}} \times \mathfrak{o}_{\mathcal{T}_{0}}^{A_{0}}))\\
	&= (\natural)_{*}((\natural)_{\#}(\mathfrak{o}_{\mathcal{T}}^{\psi} \times \mathfrak{o}_{\mathcal{T}_{0}}^{\psi_{0}}) \times (\natural)_{\#}(\mathfrak{o}_{\mathcal{T}}^{A} \times \mathfrak{o}_{\mathcal{T}_{0}}^{A_{0}})) \\
	& =:  (\natural)_{*}(\mathfrak{o}_{\mathcal{T} \# \mathcal{T}_{0}}^{\psi \# \psi_{0}} \times \mathfrak{o}_{\mathcal{T} \# \mathcal{T}_{0}}^{A \# A_{0}} ) \\
	\end{align*}
	where the second equality is the definition of coherence for the orientation systems $\mathfrak{o}_{\mathcal{T}}$ and $\mathfrak{o}_{\mathcal{T}_{0}}$, and the third equality is the statement of the commutativity of the two induced glueing maps referenced above. This commutativity follows from the fact that the two glueing maps can be viewed as taking place in a small neighborhood of the curves being glued, and can thus be performed in either order, or simultaneously, via the construction in \cite[Appendix A]{Cylindrical}. This establishes coherence of the system $\mathfrak{o}_{\mathcal{T} \# \mathcal{T}_{0}}$. 
	
	For $u \in \mathcal{M}(\psi)$ let $$\mathcal{M}_{(\btheta, \ba,\bb)}(\rho^{p}(u)) = \coprod_{\substack{\psi_{0} \in \pi_{2}(\btheta, \ba, \bb) \\ \mu(\psi_{0}) = 2n_{p}(\psi) }} \mathcal{M}(\psi_{0}, \rho^{p}(u)).$$ With respect to a coherent orientation system $\mathfrak{o}_{\mathcal{T} \# \mathcal{T}_{0}}$ on $\mathcal{T} \# \mathcal{T}_{0}$ determined from any coherent systems $\mathfrak{o}_{\mathcal{T}}$ and $\mathfrak{o}_{\mathcal{T}_{0}}$ as above, the triangle map in question can then be written as
	
	\begin{align*}
	\mathcal{F} &= \mathcal{F}_{\mathcal{T} \# \mathcal{T}_{0}}((\boldsymbol{x} \times \boldsymbol{\Theta}) \otimes (\boldsymbol{y} \times \boldsymbol{a})) \\
	&= \sum_{\bz} \sum_{\substack{\psi \in \pi_{2}(\boldsymbol{x}, \boldsymbol{y}, \boldsymbol{z}) \\ \psi_{0} \in \pi_{2}(\boldsymbol{\Theta},\boldsymbol{a}, \bb ) \\ \mu(\psi \# \psi_{0}) = 0}} \hspace{-1em} \# \{ (u,u_{0}) \in \mathcal{M}(\psi) \times \mathcal{M}(\psi_{0}) | \rho^{p}(u) = \rho^{p}(u_{0})      \} U^{n_{p} (\psi \# \psi_{0})} \cdot \bz  \times \bb \\
	&= \sum_{\bz} \sum_{\substack{ \mu(\psi) = 0 \\ \psi \in \pi_{2}(\boldsymbol{x}, \boldsymbol{y}, \boldsymbol{z}) }} \sum_{\substack{ \psi_{0} \in \pi_{2}(\boldsymbol{\Theta},\boldsymbol{a}, \bb ) \\ \mu(\psi_{0}) = 2n_{p}(\psi)}} \hspace{-1em} \# \{ (u,u_{0}) \in \mathcal{M}(\psi) \times \mathcal{M}(\psi_{0}) | \rho^{p}(u) = \rho^{p}(u_{0})      \} U^{n_{p}  (\psi \# \psi_{0})} \cdot \bz  \times \bb \\
	&=\sum_{\bz} \sum_{\substack{\psi \in \pi_{2}(\boldsymbol{x}, \boldsymbol{y}, \boldsymbol{z}) \\ \mu(\psi) = 0}} \sum_{\substack{ \psi_{0} \in \pi_{2}(\boldsymbol{\Theta},\boldsymbol{a}, \bb ) \\ \mu(\psi_{0}) = 2n_{p}(\psi)}} \sum_{u \in \mathcal{M}(\psi)} \# \left( u \times \mathcal{M}(\psi_{0}, \rho^{p}(u))\right) U^{n_{p} (\psi \# \psi_{0})} \cdot \bz  \times \bb \\
	&= \sum_{\bz} \sum_{\substack{\psi \in \pi_{2}(\boldsymbol{x}, \boldsymbol{y}, \boldsymbol{z}) \\ \mu(\psi) = 0}} \sum_{u \in \mathcal{M}(\psi)} \# \left( u \times \mathcal{M}_{(\btheta, \ba,\bb)}(\rho^{p}(u))\right) U^{n_{p}  (\psi \# \psi_{0})} \cdot \bz  \times \bb 
	\end{align*}
	We will show in Lemma \ref{SingleDivisorModuliCount} below that there is a coherent orientation system $\mathfrak{o}_{\mathcal{T}_{0}}$ on $\mathcal{T}_{0}$ for which either $$\# \mathcal{M}_{(\btheta, \ba,\bb)}(\rho^{p}(u)) = 1$$ for all $\psi$ with $\mu(\psi)=0$ and all $u \in \mathcal{M}(\psi)$, or $$\# \mathcal{M}_{(\btheta, \ba,\bb)}(\rho^{p}(u)) = -1$$ for all $\psi$ with $\mu(\psi)=0$ and all $u \in \mathcal{M}(\psi)$. Then we will have
	
	\begin{align*}
	\mathcal{F} &= \mathcal{F}^{-}_{\mathcal{T} \# \mathcal{T}_{0}}((\boldsymbol{x} \times \boldsymbol{\Theta}) \otimes (\boldsymbol{y} \times \boldsymbol{a})) \\
	 &= \sum_{\bz} \sum_{\substack{\psi \in \pi_{2}(\boldsymbol{x}, \boldsymbol{y}, \boldsymbol{z}) \\ \mu(\psi) = 0}} \sum_{u \in \mathcal{M}(\psi)} \# \left( u \times \mathcal{M}_{(\btheta, \ba,\bb)}(\rho^{p}(u))\right) U^{n_{p}(\psi \# \psi_{0})} \cdot \bz  \times \bb\\
	&= \pm \sum_{\bz} \sum_{\substack{\psi \in \pi_{2}(\boldsymbol{x}, \boldsymbol{y}, \boldsymbol{z}) \\ \mu(\psi) = 0}} \# \mathcal{M}(\psi) U^{n_{p}(\psi \# \psi_{0})} \cdot \bz  \times \bb \\
	&=  \pm \left( \sum_{\bz} \sum_{\psi \in \pi_{2}(\bx,\by,\bz), \mu(\psi) = 0} (\#\mathcal{M}(\psi)) U^{n_{p}(\psi)} \cdot \bz \right) \times \bb \\
	&= \pm \mathcal{F}^{-}_{\mathcal{T}}(\boldsymbol{x} \otimes \boldsymbol{y}) \times \boldsymbol{b}
	\end{align*}
	This completes the proof of the proposition, modulo Lemma \ref{SingleDivisorModuliCount}. \qedhere
	
\end{proof}

\begin{claim}\label{SingleDivisorModuliCount}
	For $\boldsymbol{d} \in Sym^{k}(\Delta) \setminus Diag(\Delta)$ and a generic choice of almost complex structure $J$, the moduli space $\mathcal{M}_{(\btheta, \ba, \bb)}(\boldsymbol{d})$ is a smoothly cut out 0-manifold. For such $J$, there is a coherent orientation system $\mathfrak{o}_{\mathcal{T}_{0}}$ on $\mathcal{T}_{0}$ for which the signed count of points in the moduli space is
	
	$$ \# \mathcal{M}_{(\btheta, \ba, \bb)}(\boldsymbol{d})= \pm 1$$
	where the constant is independent of $\boldsymbol{d}$.
\end{claim}

\begin{proof}
	The proof is again carried out in steps:
	
	\begin{enumerate}
		\item We show the moduli space is transversely cut out for generic $J$.
		\item We show that for generic $\boldsymbol{d} \in \text{Sym}^{k}(\Delta) \setminus \text{Diag}(\Delta)$, the signed count $ \# \mathcal{M}_{(\btheta, \ba, \bb)}(\boldsymbol{d})$ is independent of $\boldsymbol{d}$.
		\item We find one choice of $\boldsymbol{d}$ giving the desired count. 
	\end{enumerate}
	
	In fact, the proof of step (1) given in \cite{Naturality} carries over exactly as it is stated there, so we will only prove steps (2) and (3).
	
	\textbf{Step 2}
	Let $p: I \rightarrow \text{Sym}^{k}(\Delta)$ be a path from $\boldsymbol{d_{0}}$ to $\boldsymbol{d_{1}}$, where the image of $p$ satisfies the conditions of Lemma \ref{Prop9.47}. We consider the moduli space
	$$\bigcup_{t \in I} \mathcal{M}_{(\btheta, \ba, \bb)}(p(t))$$
	which by Proposition \ref{Prop9.47} and Lemma \ref{OrientabilityofMatchedModuliSpace} is a smooth, orientable 1 manifold. From orientability, we know that the signed count of the ends of the moduli space above is zero. We now describe all contributions to the count of the ends.  We begin by making considerations which will hold for any choice of coherent orientation system satisfying the property appearing in Lemma \ref{ParametrizedOrientation}.
	
	The ends of $\bigcup_{t \in I} \mathcal{M}_{(\btheta, \ba, \bb)}(p(t))$ fall into three classes. They arise from $\mathcal{M}_{(\btheta, \ba, \bb)}(\boldsymbol{d_{0}})$, $\mathcal{M}_{(\btheta, \ba, \bb)}(\boldsymbol{d_{1}})$, and degenerations of holomorphic triangles to broken holomorphic triangles in the compactification. Let $u_{i}: S_{0} \rightarrow \Sigma_{0} \times \Delta$ be a sequence of holomorphic triangles in $\bigcup_{t \in I} \mathcal{M}_{(\btheta, \ba, \bb)}(p(t))$. As shown in \cite[ Lemma 9.58]{Naturality}, the only degenerations that can occur correspond to ``strip breaking". In particular, if $u_{i}$ converges to a  broken holomorphic triangle $$ U = (u_{1}, v_{1}, \ldots, v_{n}, w_{1}, \ldots, w_{m})$$ (in the sense of Definition \ref{Broken Holomorphic Triangle}), then in fact $U = (u_{1}, v_{1}, \ldots, v_{n})$ where the $v_{i}$ are holomorphic strips. We note that the argument used to rule out other types of degenerations has nothing to do with orientations.  Furthermore, we will see presently that among degenerations corresponding to strip breaking, the only ones which can occur yield broken triangles $U$ consisting of a triangle $u_{1}$ of index $2k -1$ which matches a divisor $p(t)$ for some $t \in I$, as well as a single curve $v_{1}: S \rightarrow \Sigma_{0} \times I \times \mathbb{R}$ with index 1. 
	
	To see this, note that if $U$ is genuinely broken then $U=(u_{1}, v_{1}, \ldots, v_{n})$ with $u_{1}$ a holomorphic triangle representing a class in $\pi_{2}(\boldsymbol{x}, \ba,\bb)$ and $v_{i}$  holomorphic curves in $\pi_{2}(\boldsymbol{y}_{i}, \boldsymbol{z}_{i})$ for some $\boldsymbol{y}_{i}, \boldsymbol{z}_{i} \in \mathbb{T}_{\alpha '} \cap \mathbb{T}_{\alpha}$.We now analyze what contributions to the ends can occur for the four possible intersection points $\bx \in \mathbb{T}_{\alpha '} \cap \mathbb{T}_{\alpha}$.

	Suppose $\bx = \btheta$. Then by applying Lemma \ref{Lemma950} to $u_{1}$ we obtain $\mu(u_{1}) = 2n_{p_{0}}(u_{1})$. Since $u_{1}$ satisfies a matching condition with $p(t)$ for some $t \in I$, we have $2n_{p_{0}}(u_{1}) = |\rho^{p}(p(t))| = k = 2n_{p_{0}}(\psi_{0})= \mu(\psi_{0})$. Thus $\mu(u_{1}) = \mu(\psi_{0})$. Since the total homology class of $U$ must be $\psi_{0}$, we therefore must have $\mu(v_{i}) = 0$ and $n_{p_{0}}(v_{I}) = 0$ for all $i$. Since the $v_{i}$ satisfy $(M1)$ and $(M3)$-$(M6)$, the only possibility for such curves is that each is a collection of constant components. Indeed, if any $v_{i}$ were locally nonconstant, it would satisfy $(M2)$, hence by \cite[Corollary 7.2]{Naturality} the dimension of the relevant moduli space containing it would be negative. Thus $U = (u_{1})$ (plus potentially some constant curves) is in the interior of $\bigcup_{t \in I} \mathcal{M}_{(\btheta, \ba, \bb)}(p(t))$, and so contributes nothing to the signed count of the ends. 
	
	Next we consider the cases  $\bx = \theta_{1}^{+} \theta_{2}^{-} , \theta_{1}^{-} \theta_{2}^{+} $. In these cases Lemma \ref{Lemma950} yields that the index of the triangle must be $\mu(u_{1}) = 2n_{p_{0}}(u_{1})-1 = 2n_{p_{0}}(\psi_{0})-1$, so the remaining curves must have indices which sum to 1. Similarly, $0= n_{p_{0}}(\psi_{0}) - n_{p_{0}}(u_{1})= \sum_{i} n_{p_{0}}(v_{i})$, so $v_{i}$ must have multiplicity 0 at the basepoint for each $i$. The only possibility in this case is that there is a single Maslov index 1 strip $v_{1}$. Thus in this case, we have additional contributions to the ends coming from:
	
	$$\bigcup_{\substack{t \in I \\ \boldsymbol{x} \in \{\theta_{1}^{+} \theta_{2}^{-} , \theta_{1}^{-} \theta_{2}^{+}\} }} \bigcup_{\substack{\phi \in \pi_{2}(\btheta, \boldsymbol{x}) \\ n_{p_{0}}(\phi) = 0}} \mathcal{M}_{(\boldsymbol{x}, \ba, \bb)}(p(t)) \times \widehat{\mathcal{M}}(\phi)$$
	Fix $\boldsymbol{x} \in \{\theta_{1}^{+} \theta_{2}^{-} , \theta_{1}^{-} \theta_{2}^{+}\}$. Then by Lemma \ref{hatdifferential} we know that
	
	$$\sum_{\substack{\phi \in \pi_{2}(\btheta, \boldsymbol{x}) \\ n_{p_{0}}(\phi) = 0}} \#  \widehat{\mathcal{M}}(\phi) = 0$$
	Thus 
	\begin{align*}
	&\#(\bigcup_{\substack{t \in I \\ \boldsymbol{x} \in \{\theta_{1}^{+} \theta_{2}^{-} , \theta_{1}^{-} \theta_{2}^{+}\} }} \bigcup_{\substack{\phi \in \pi_{2}(\btheta, \boldsymbol{x}) \\ n_{p_{0}}(\phi) = 0}} \mathcal{M}_{(\boldsymbol{x}, \ba, \bb)}(p(t)) \times \widehat{\mathcal{M}}(\phi)) \\
	&= \sum_{\substack{t \in I \\ \boldsymbol{x} \in \{\theta_{1}^{+} \theta_{2}^{-} , \theta_{1}^{-} \theta_{2}^{+}\} }} \sum_{\substack{\phi \in \pi_{2}(\btheta, \boldsymbol{x}) \\ n_{p_{0}}(\phi) = 0}} \#(\mathcal{M}_{(\boldsymbol{x}, \ba, \bb)}(p(t)) \times \widehat{\mathcal{M}}(\phi))  \\
	&= \sum_{\substack{t \in I \\ \boldsymbol{x} \in \{\theta_{1}^{+} \theta_{2}^{-} , \theta_{1}^{-} \theta_{2}^{+}\} }} \sum_{\substack{\phi \in \pi_{2}(\btheta, \boldsymbol{x}) \\ n_{p_{0}}(\phi) = 0}} (\# \mathcal{M}_{(\boldsymbol{x}, \ba, \bb)}(p(t))) \cdot (\#  \widehat{\mathcal{M}}(\phi)) \\
	&= 0
	\end{align*}
	Here we have used in the last equality the fact that we have endowed the orientable manifold $\bigcup_{t \in I} \mathcal{M}_{(\btheta, \ba, \bb)}(p(t))$ with some coherent orientation system. This implies in particular that the orientation induced on the compactification agrees with the product orientation at ends such as those above. So we see these cases also contribute nothing to the count of signed ends of the moduli space. 
	
	Lastly, we consider the case $\boldsymbol{x} = \theta_{1}^{-} \theta_{2}^{-}$. For any $\psi_{0} \in \pi_{2}(\theta_{1}^{-} \theta_{2}^{-}, \boldsymbol{a}, \boldsymbol{b})$ we have by lemma \ref{Lemma950} $\mu(psi_{0}) = 2n_{p_{0}}(\psi_{0}) - 2 =2k -2$. By proposition \ref{Prop9.47}, for a generic choice of almost complex structure $J$, and a fixed source $S$, the matched moduli space $\mathcal{M}(\psi_{0}, S, p(I))$ is a smooth manifold of dimension
	
	$$\text{ind}(\psi_{0},S) - \text{codim}(p(I))= \text{ind}(\psi_{0},S) - (2k-1) \leq \mu(\psi_{0}) - (2k-1) = -1$$
	Here the fact being used to establish the inequality is that for any holomorphic triangle $u$ in the homology class $A$ (not necessarily embedded), the index of the linearized $\bar{\partial}$ operator at $u$ satisfies $\text{ind}(A, S) = \mu(A) -2\text{sing}(u)$, and in particular $ \text{ind}(A, S) \leq \mu(A)$. This is \cite[Equation 9.46]{Naturality}, which comes from adapting \cite[Proposition 5.69]{Bordered}. This shows that for a generic choice of $J$, the broken triangle $U$ can not in fact contain a triangle $u_1$ in such a class $\psi_{0}$. 
	
	To summarize, we have shown that the ends of $\bigcup_{t \in I} \mathcal{M}_{(\btheta, \ba, \bb)}(p(t))$ correspond to $\mathcal{M}_{(\btheta, \ba, \bb)}(\boldsymbol{d_{0}})$, $\mathcal{M}_{(\btheta, \ba, \bb)}(\boldsymbol{d_{1}})$, and to degenerations of triangles into broken triangles containing one triangle and one strip, and that the last types of ends contribute nothing to the total signed count of the ends. Since we have chosen a collection of orientation systems satisfying the conclusion of Lemma \ref{ParametrizedOrientation}, we see that the signed count of the ends of  $\bigcup_{t \in I} \mathcal{M}_{(\btheta, \ba, \bb)}(p(t))$ is given by:
	
	$$\# \mathcal{M}_{(\btheta, \ba, \bb)}(\boldsymbol{d_{1}}) -  \# \mathcal{M}_{(\btheta, \ba, \bb)}(\boldsymbol{d_{0}}) = 0.$$
	This concludes step 2.
	
	We note that by Lemma \ref{ParametrizedOrientation}, a coherent orientation system on $\mathcal{M}_{(\btheta, \ba, \bb)}(p(0))$ induces a coherent orientation system over $\bigcup_{t \in I} \mathcal{M}_{(\btheta, \ba, \bb)}(p(t))$ and $\mathcal{M}_{(\btheta, \ba, \bb)}(p(1))$ satisfying the conclusion of the lemma. We thus see that if we can find a single divisor $\textbf{d}$ and a coherent orientation system $\mathfrak{o}$ over $\mathcal{M}_{(\btheta, \ba, \bb)}(\textbf{d})$ giving the desired count, then the argument of step 2 shows that there are induced coherent orientations over all divisors $\textbf{d}'$ in the same path component as $\textbf{d}$ for which the counts are the same. We will construct such a divisor in step 3 below.
	
	\textbf{Step 3}
	To construct a divisor $\textbf{d} \in \text{Sym}^{k}(\Delta) \setminus \text{Diag}(\Delta)$ giving the desired count, we consider a path of divisors subject to constraints, and evaluate the asymptotics of the moduli spaces of triangles matched to divisors in this path. Our argument is an explication of that in \cite{Naturality}, which is in turn based on an analogous argument  in \cite[pg. 653]{DisksLinks}  which deals with holomorphic strips. Our goal in summarizing these proofs is to make explicit the dependence of all statements on signs and orientations.
	
	We consider any path $p: [1,\infty) \rightarrow \text{Sym}^{k}(\Delta) \setminus \text{Diag}(\Delta)$ for which each point in $p(t)$ is at least a distance of $t$ away from all other points in $p(t)$, with respect to a metric on $\Delta$ for which the corners are infinite strips in $\mathbb{C}$ (see Figure \ref{Delta}). We further require that the points in $p(t)$ smoothly approach the vertex $v_{\alpha_{0} \beta_{0}}$ of $\Delta$ as $t \rightarrow \infty$. For such a path of divisors, we have as before a matched moduli space
	
	$$\mathcal{M}_{(\btheta, \ba, \bb)}(p) = \bigcup_{t \in [1,\infty]} \mathcal{M}_{(\btheta, \ba, \bb)}(p(t)).$$
	
	By the same arguments used in step 2, the ends of this moduli space corresponding to degenerations of triangles at finite values of $t$, with $t \neq 1$, will contribute nothing to the signed count of the ends, for any choice of coherent orientation system. Consider any coherent orientation system $\mathfrak{o}$ satisfying the properties of that furnished by Lemma \ref{ParametrizedOrientation}; then with respect to such an orientation system the signed count $\# \mathcal{M}_{(\btheta, \ba, \bb)}(p(1))$ must agree with the signed count of the ends of $\mathcal{M}_{(\btheta, \ba, \bb)}(p)$ coming from degenerations of triangles as $t \rightarrow \infty$. So we now count these ends.
	
	We claim that as $t \rightarrow \infty$, the only broken triangles which can occur in the limit consist of a single genuine triangle $\tau$ of index 0 on $(\Sigma_{0}, \balpha_{0}', \balpha_{0}, \bbeta_{0})$, along with $k$ index 2 curves on $(\Sigma_{0}, \balpha_{0}, \bbeta_{0})$ which satisfy matching conditions with some collection of divisors $c_{i} \in [0,1] \times \mathbb{R}$. To see this, we note that each point in the path $p$ consists of $k$ distinct points in $\Delta$, and the fact that these $k$ points separate and approach the vertex $v_{\alpha_{0} \beta_{0}}$ in the limit necessitates that the limiting broken triangle must contain $k$ strips satisfying matching conditions. To see the index of each of these curves must be 2, we make some simple observations about the diagram $(\Sigma_{0}, \balpha_{0}, \bbeta_{0})$ for $S^{3}$.

	First, note that the only homology classes of discs supporting holomorphic representatives are $\{\boldsymbol{e}_{\ba} + s[\Sigma_{0}] \}$ for nonnegative integers $s$, where $\boldsymbol{e}_{\ba}$ is the constant disk at $\ba$. The Maslov indices for such classes are $\mu(\boldsymbol{e}_{\ba} + s[\Sigma_{0}]) = 2s$. The fact that each strip satisfies a matching condition implies we must have $s \geq 1$ for each homology class. Since the total index of each holomorphic triangle in the moduli space $\mathcal{M}_{(\btheta, \ba, \bb)}(p)$ is $2k$, the limiting broken holomorphic triangle must have index $2k$, so the only possibility is that each of the k curves has index 2 (ie\ has s=1), and the triangle $\tau$ has index 0. By counting multiplicities and noting positivity of intersections, we see that the triangle $\tau$ must satisfy $n_{p_{0}}(\tau) = 0$. Using the same arguments as in the preceding proposition, we have that all of the curves in the broken triangle must satisfy $(M1) - (M8)$. 
	
	Applying the glueing result of Lipshitz \cite[Appendix A, Proposition A.1]{Cylindrical}, we see that we can obtain the signed count of the ends ocurring as degenerations as $t \rightarrow \infty$, or equivalently the count $\# \mathcal{M}_{(\btheta, \ba, \bb)}(p(1))$, as:
	
	$$\# \mathcal{M}_{(\btheta, \ba, \bb)}(p(1)) = (\# \mathcal{M}_{(\ba, \ba)}(c))^{k} \cdot \sum_{\substack{ \psi \in \pi_{2}(\btheta, \ba,\bb) \\ n_{p_{0}}(\psi) =0}} \# \mathcal{M}(\psi)$$
	where $c$ is a divisor in $[0,1] \times \mathbb{R}$ and $\mathcal{M}_{(\ba, \ba)}(c)$ is the moduli space of index 2 strips on $(\Sigma_{0}, \balpha_{0}, \bbeta_{0})$ with $\rho^{p}(u) = c$. Here the counts are occurring with respect to any coherent orientation system $\mathfrak{o}_{\mathcal{T}_{0}} = \{\mathfrak{o}_{\balpha_{0}', \balpha_{0}, \bbeta_{0}},\mathfrak{o}_{\balpha_{0}, \bbeta_{0}}, \mathfrak{o}_{\balpha_{0}', \balpha_{0}}, \mathfrak{o}_{\balpha_{0}', \bbeta_{0}} \}$ on $\mathcal{T}_{0}$ and the compatible orientation system $\mathfrak{o}_{\balpha_{0}, \bbeta_{0}}$ included in the data $\mathfrak{o}_{\mathcal{T}_{0}}$. The sum on the right hand side is precisely the count occurring in the triangle map in Lemma \ref{quasi}, and is thus $\pm 1$. Thus to finish this step it suffices to show that there is a coherent orientation system $\mathfrak{o}_{\mathcal{T}_{0}}$ for which
	
	$$ \# \mathcal{M}_{(\ba, \ba)}(c) = \pm 1.$$
	Consider the standard diagram $\mathcal{H}_{S^{1} \times S^{2}}$ for $S^{1} \times S^{2}$, twice stabilized via the diagram $(\Sigma_{0}, \balpha_{0}, \bbeta_{0})$ as shown in Figure \ref{TwiceStabilizedDiagram}. The figure depicts this genus 3 diagram for $S^{1} \times S^{2}$, along with a choice of basepoint $z$. Both bigons in $\mathcal{H}_{S^{1} \times S^{2}}$ for $S^{1} \times S^{2}$ admit a single holomorphic representative. We consider a choice of coherent orientation system on $\mathcal{H}_{S^{1} \times S^{2}}$ for which the the bigons cancel, and the resulting Floer homology is $\widehat{HF} \cong \mathbb{Z}^{2}$. By invariance of $\widehat{HF}$, the twice stabilized bigon in the twice stabilized diagram must also have a single holomorphic representative. As in the proof of stabilization invariance in \cite{Cylindrical}, this implies via a neck stretching argument that there is a coherent orientation system $\mathfrak{o}_{\balpha_{0}, \bbeta_{0}}$ on $(\Sigma_{0}, \balpha_{0}, \bbeta_{0})$ for which
	$$\# \mathcal{M}_{(\ba, \ba)}(c) = \pm 1.$$
	By \cite[Lemma 8.7]{Disks1}, this coherent orientation system can be extended to a coherent orientation system $\mathfrak{o}_{\mathcal{T}_{0}}$ for which the same condition holds.
	This completes step 3, and the proof of the lemma.

	\begin{figure}[h!]
		\centering
		\begin{tikzpicture}
		\node[anchor=south west,inner sep=0] (image) at (0,0) {\includegraphics[width=0.5\textwidth]{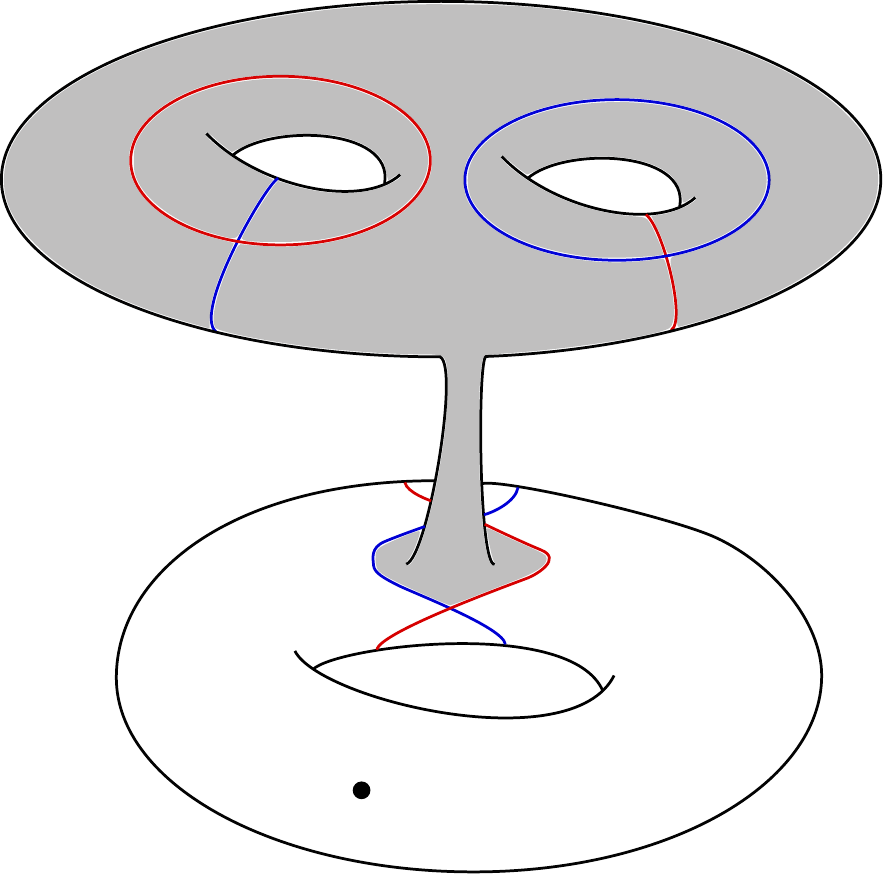}};
		\begin{scope}[x={(image.south east)},y={(image.north west)}]
		\end{scope}
		\end{tikzpicture}
		\caption{The diagram $\mathcal{H}_{S^{1} \times S^{2}}$ on the bottom of the figure is twice stabilized via a connect sum with $(\Sigma_{0}, \balpha_{0}, \bbeta_{0})$. Shaded in grey is a domain on the genus 3 diagram, the "twice stabilized bigon", which arises from one of the bigons in $\mathcal{H}_{S^{1} \times S^{2}}$.}
		\label{TwiceStabilizedDiagram}
	\end{figure}

\end{proof}

\section{The Surgery Exact Triangle}
\label{surgerytrianglesection}
In this section, we provide a brief explanation of how our results fit into the construction of the surgery exact sequence defined by Ozsv\'ath and Szab\'o in \cite[Section 9]{Disks2}. The fact that these constructions are compatible with ours turns out to be a matter of bookkeeping. We provide a sketch of the argument here with the hope that it will be useful in extending our naturality results to results about general cobordisms.

First, we recall one version of the construction of the surgery exact triangle and its relation to our naturality results. The relation between other versions of the statement of the exact triangle and our naturality results follows analogously. Let $Y$ be a closed oriented $3$-manifold, and $K \hookrightarrow Y$ be a knot with a longitude $\lambda$ and meridian $\mu$. We denote by $Y_{0}$ the $3$-manifold obtained by performing $\lambda$-surgery on $Y$, and by $Y_{1}$ the $3$-manifold obtained by performing $(\lambda + \mu)$-surgery on $Y$. Call any such triple $(Y, Y_{0}, Y_{1})$ of $3$-manifolds a \emph{triad}. Ozsv\`ath and Szab\`o showed:

\begin{thm}[Theorem 9.12 in \cite{Disks2}]
	\label{surgerytriangle}
	For any triad $(Y, Y_{0}, Y_{1})$ there are long exact sequences of $\mathbb{Z}[U]$-modules:
	
	\[
	\begin{tikzcd}
	HF^{+}(Y) \arrow[rr, "F"] & &   HF^{+}(Y_{0}) \arrow[ld, "F_{0}"]     \\
	& HF^{+}(Y_{1}) \arrow[lu, "F_{1}"]&  
	\end{tikzcd} 
	\] 
	
	and
	\[
	\begin{tikzcd}
	\widehat{HF}(Y) \arrow[rr, "\widehat{F}"] & &   \widehat{HF}(Y_{0}) \arrow[ld, "\widehat{F}_{0}"]     \\
	& \widehat{HF}(Y_{1}) \arrow[lu, "\widehat{F}_{1}"]&  
	\end{tikzcd} 
	\]

\end{thm}

The statement above is established via a corresponding statement made at the level of diagrams. To describe it, we recall a particular class of diagrams representing the manifolds in such a triad $(Y, Y_{0}, Y_{1})$. Let $(\mathcal{H}, \mathcal{H}_{0}, \mathcal{H}_{1})$ be a tuple of diagrams for the $3$-manifolds $(Y, Y_{0}, Y_{1})$ respectively. We will say the tuple of diagrams is \emph{subordinate to the surgery triad} if there is a pointed genus $g$ Heegaard quadruple diagram $(\Sigma, \balpha, \bbeta, \bgamma, \bdelta, z)$ satisfying the following properties:

\begin{itemize}
	\item The diagrams $(\Sigma, \balpha, \bbeta)$, $(\Sigma, \balpha, \bgamma)$, $(\Sigma, \balpha, \bdelta)$ represent $Y$, $Y_{0}$ and $Y_{1}$ respectively.
	\item For $i \neq g$, $\beta_{i}$, $\gamma_{i}$ and $\delta_{i}$ are isotopic translates of one another, each intersecting transversally in two points.
	\item $\gamma_{g}$ is isotopic to the juxtaposition of $\delta_{g}$ and  $\beta_{g}$ (See Figure 9 in \cite{Disks2} for a depiction of juxtaposition).
	\item Every multi-periodic domain on the quadruple diagram has positive and negative coefficients.
	
\end{itemize}
Existence of such subordinate diagrams was established by Ozsv\`ath and Szab\`o:

\begin{claim}[Lemma 9.2 in \cite{Disks2}]
	\label{subordinate exists}
	Given a triad $(Y, Y_{0}, Y_{1})$, there is a tuple of Heegaard diagrams $(\mathcal{H}, \mathcal{H}_{0}, \mathcal{H}_{1})$ subordinate to the triad.
\end{claim}

Theorem \ref{surgerytriangle} is then to be interpreted as a compact way of phrasing the following statement at the level of diagrams:

\begin{thm}
	\label{surgerytrianglediagramlevel}
	Let $(\mathcal{H}, \mathcal{H}_{0}, \mathcal{H}_{1})$ be a tuple subordinate to the triad $(Y, Y_{0}, Y_{1})$, and fix a coherent orientation system $\mathfrak{o}_{\mathcal{H}_{0}}$ on $\mathcal{H}_{0}$. Then there are coherent orientation systems on $\mathcal{H}$ and $\mathcal{H}_{1}$, and maps $F$, $F_{0}$ and $F_{1}$ induced by triangle counts as above, such that with respect to the chosen coherent orientation systems there are exact triangles:
	
	\[
	\begin{tikzcd}
	HF^{+}(\mathcal{H}) \arrow[rr, "F"] & &   HF^{+}(\mathcal{H}_{0}) \arrow[ld, "F_{0}"]     \\
	& HF^{+}(\mathcal{H}_{1}) \arrow[lu, "F_{1}"]&  
	\end{tikzcd} 
	\] 
	
	and
	\[
	\begin{tikzcd}
	\widehat{HF}(\mathcal{H}) \arrow[rr, "\widehat{F}"] & &   \widehat{HF}(\mathcal{H}_{0}) \arrow[ld, "\widehat{F}_{0}"]     \\
	& \widehat{HF}(\mathcal{H}_{1}) \arrow[lu, "\widehat{F}_{1}"]&  
	\end{tikzcd} 
	\] 
\end{thm}

We now use this restatement of the theorem at the level of diagrams to show that the surgery exact triangle is also well defined with respect to the transitive systems specified by Definition \ref{WeakHeegaardFloer1} and Theorem \ref{Functoriality}. 

Recall that Definition \ref{WeakHeegaardFloer1} and Theorem \ref{Functoriality} describe the four variants of Heegaard Floer homology as functors:

\[
HF^{\circ}: \text{Man}_{*} \rightarrow \text{Trans}(P(\mathbb{Z}[U]-\text{Mod}))
\]
The precise restatement of the surgery exact sequence that we wish to establish in this context is just that the exact sequence defined by Ozsv\`ath and Szab\`o extends to an exact sequence at the level of the transitive systems associated to a triad $(Y,Y_{0},Y_{1})$. Given transitive systems $T$, $T_{0}$ and $T_{1}$ and morphisms of transitive systems $F:T \rightarrow T_{0}$, $F_{0}: T_{0} \rightarrow T_{1}$ and $F_{1}: T_{1} \rightarrow T$, we will say the morphisms form an \emph{exact sequence of transitive systems} if the morphisms restricted to constituent objects form exact sequences. We then have:

\begin{cor}
	\label{surgerycorollary}
	For any triad $(Y, Y_{0},Y_{1})$, there are exact sequences of transitive systems \[
	\begin{tikzcd}
	HF^{+}(Y) \arrow[rr, "F"] & &   HF^{+}(Y_{0}) \arrow[ld, "F_{0}"]     \\
	& HF^{+}(Y_{1}) \arrow[lu, "F_{1}"]&  
	\end{tikzcd} 
	\] 
	
	and
	\[
	\begin{tikzcd}
	\widehat{HF}(Y) \arrow[rr, "\widehat{F}"] & &   \widehat{HF}(Y_{0}) \arrow[ld, "\widehat{F}_{0}"]     \\
	& \widehat{HF}(Y_{1}) \arrow[lu, "\widehat{F}_{1}"]&  
	\end{tikzcd} 
	\] 
\end{cor}
\begin{proof}[Proof of Corollary \ref{surgerycorollary}]
	Fix a triad $(Y,Y_{0},Y_{1})$, and a tuple $(\mathcal{H}, \mathcal{H}_{0}, \mathcal{H}_{1})$ of diagrams subordinate to the triad; such a subordinate tuple exists by Lemma \ref{subordinate exists}. By Theorem \ref{surgerytrianglediagramlevel}, applying $HF^{+}$ and $\widehat{HF}$ to the diagrams in this tuple yields long exact triangles relating the $\mathbb{Z}[U]$-modules associated to the diagrams. Note that Theorem \ref{surgerytrianglediagramlevel} ensures this statement is true with respect to any choice of coherent orientation system over $\mathcal{H}_{0}$, and the coherent orientations it induces on $\mathcal{H}$ and $\mathcal{H}_{1}$ via the triangle maps $F_{0}$ and $F_{1}$. For the remainder of the proof, we fix coherent orientations $(\mathfrak{o}_{\mathcal{H}},\mathfrak{o}_{\mathcal{H}_{0}}, \mathfrak{o}_{\mathcal{H}_{1}})$ on $(\mathcal{H}, \mathcal{H}_{0}, \mathcal{H}_{1})$ which are related in this way.
	
	Notice that by their definition, the transitive systems $HF^{\circ}(Y)$, $HF^{\circ}(Y_{0})$ and $HF^{\circ}(Y_{1})$ contain as constituent objects the modules $HF^{\circ}(\mathcal{H})$, $HF^{\circ}(\mathcal{H}_{0})$ and $HF^{\circ}(\mathcal{H}_{1})$. Thus the exact triangle associated to this tuple of diagrams in Theorem \ref{surgerytrianglediagramlevel} begins to partially define an exact sequence between the transitive systems. The situation is depicted in the leftmost column of Figure \ref{SurgeryCommutativity}. 
	
	It is easy to extend this partially defined triangle of maps between transitive systems to a (more) partially defined triangle, defined now on all objects which correspond to triples of diagrams subordinate to $(Y, Y_{0}, Y_{1})$. Given two tuples of diagrams $(\mathcal{H}, \mathcal{H}_{0}, \mathcal{H}_{1})$ and $(\mathcal{H}'', \mathcal{H}_{0}'', \mathcal{H}_{1}'')$ subordinate to $(Y, Y_{0},Y_{1})$, and equivalences $(\Psi'', \Psi_{0}'', \Psi_{1}'')$ induced by Heegaard moves relating the two tuples of diagrams, the maps $F_{i}$ and $F_{i}''$ appearing in the respective exact triangles commute (up to sign) with the equivalence maps by \cite[Theorem 4.4]{FourManifoldInvariants}. See Figure \ref{SurgeryCommutativity} for a depiction of the situation. In other words, the surgery triangle immediately extends by this result to a partially defined triangle of transitive systems, which is now defined on all diagrams occurring in a subordinate tuple. This can be described in the diagram of Figure \ref{SurgeryCommutativity} by saying that the front square faces in the diagram commute (up to sign). 
	
	The final thing that remains to be shown is that these partially defined morphisms of transitive systems can be extended to maps defined on the Heegaard Floer modules associated to any admisssible diagram, while preserving the consistency required of a morphism of transitive systems. With respect to the notation in Figure \ref{SurgeryCommutativity}, this can be phrased as asking for maps
	\[
	F': HF^{+}(\mathcal{H}') \rightarrow HF^{+}(\mathcal{H}_{0}')
	\]
	
	\[
	F_{0}': HF^{+}(\mathcal{H}_{0}') \rightarrow HF^{+}(\mathcal{H}_{1}')
	\]
	
	\[
	F_{1}': HF^{+}(\mathcal{H}_{1}') \rightarrow HF^{+}(\mathcal{H}')
	\]
	defined with respect to a tuple $(\mathcal{H}', \mathcal{H}_{0}', \mathcal{H}_{1}')$ which is not subordinate to $(Y,Y_{0},Y_{1})$, such that all of the faces in Figure \ref{SurgeryCommutativity} commute.
	
	This is again straightforward: let $(\Psi, \Psi_{0}, \Psi_{1})$ be a tuple of equivalences induced by Heegaard moves relating $(\mathcal{H}, \mathcal{H}_{0}, \mathcal{H}_{1})$ and $(\mathcal{H}', \mathcal{H}_{0}', \mathcal{H}_{1}')$, and $(\Psi', \Psi_{0}', \Psi_{1}')$ be a tuple of equivalences induced by Heegaard moves relating $(\mathcal{H}', \mathcal{H}_{0}', \mathcal{H}_{1}')$ and $(\mathcal{H}'', \mathcal{H}_{0}'', \mathcal{H}_{1}'')$. Again, we refer to Figure \ref{SurgeryCommutativity} to help recall the meaning of the notation. Define the map $F': HF^{+}(\mathcal{H}') \rightarrow HF^{+}(\mathcal{H}_{0}')$ by $F':= \Psi_{0} \circ F \circ \Psi^{-1}$ where $\Psi^{-1}$ is a homotopy inverse for the equivalence $\Psi$. Similarly, define $F_{0}':= \Psi_{1} \circ F_{0} \circ \Psi_{0}^{-1}$ and $F_{1}':= \Psi \circ F_{1} \circ \Psi_{1}^{-1}$. Note that 
	\begin{align*}
	F' &= \Psi_{0} \circ F \circ \Psi^{-1} \\
	&= \Psi_{0} \circ (\Psi_{0}'')^{-1} \circ F'' \circ \Psi'' \circ \circ \Psi^{-1} \text{ (by \cite[Theorem 4.4]{FourManifoldInvariants})} \\
	&= \pm (\Psi_{0}')^{-1} \circ F'' \circ \Psi' \text{ (by Theorem \ref{Functoriality} )}
	\end{align*}
	
	This shows that we can provide maps on all of the objects of our transitive systems, and furthermore by the computation above that this gives a well defined morphism of transitive systems. Exactness of all ``columns" follows by construction as well. This completes the proof. \qedhere
	
	\begin{figure}[!h]
		\[
		\begin{tikzcd}
		& HF^{+}(\mathcal{H'}) \arrow[dr,"\Psi'"] \arrow[dd, pos = .3, "F'", dashed] \arrow[from =dddd, dashed, swap, pos = .6, bend right = 40, "F_{1}'"]  &        \\
		HF^{+}(\mathcal{H}) \arrow[ur, "\Psi"] \arrow[rr, pos = .3, crossing over, "\Psi''"] \arrow[dd, "F"]& & HF^{+}(\mathcal{H}'') \arrow[dd, "F''"]  \\
		& HF^{+}(\mathcal{H}_{0}') \arrow[dr,crossing over, "\Psi_{0}'"] \arrow[dd, pos =.3, "F_{0}'", dashed] & \\
		HF^{+}(\mathcal{H}_{0}) \arrow[ur, "\Psi_{0}"] \arrow[rr, pos = .3, crossing over, "\Psi_{0}''"] \arrow[dd,"F_{0}"] & &HF^{+}(\mathcal{H}_{0}'') \arrow[dd, "F_{0}''"] \\
		& HF^{+}(\mathcal{H}_{1}') \arrow[dr,"\Psi_{1}'"]  & \\
		HF^{+}(\mathcal{H}_{1}) \arrow[ur, "\Psi_{1}"] \arrow[rr, pos=.3, "\Psi_{1}''"] \arrow[uuuu, bend left = 60, "F_{1}"] & &HF^{+}(\mathcal{H}_{1}'') \arrow[uuuu,swap,  bend right = 60, "F_{1}''"] 
		\end{tikzcd} 
		\] 
		\caption{A depiction of the diagrams involved in the proof of Corollary \ref{surgerycorollary}.}
		\label{SurgeryCommutativity}
	\end{figure}
\end{proof}

\bibliographystyle{alpha}
\bibliography{References}

\begin{thebibliography}{{Zem}15}

\bibitem[HLS16a]{simplicial}
Kristen {Hendricks}, Robert {Lipshitz}, and Sucharit {Sarkar}.
\newblock {A simplicial construction of G-equivariant Floer homology}.
\newblock {\em arXiv e-prints}, page arXiv:1609.09132, September 2016.

\bibitem[HLS16b]{FlexibleConstruction}
Kristen Hendricks, Robert Lipshitz, and Sucharit Sarkar.
\newblock A flexible construction of equivariant {F}loer homology and
  applications.
\newblock {\em J. Topol.}, 9(4):1153--1236, 2016.

\bibitem[HM17]{HM}
Kristen Hendricks and Ciprian Manolescu.
\newblock Involutive {H}eegaard {F}loer homology.
\newblock {\em Duke Math. J.}, 166(7):1211--1299, 2017.

\bibitem[JM08]{CobApps1}
Stanislav Jabuka and Thomas~E. Mark.
\newblock Product formulae for {O}zsv\'{a}th-{S}zab\'{o} 4-manifold invariants.
\newblock {\em Geom. Topol.}, 12(3):1557--1651, 2008.

\bibitem[JTZ12]{Naturality}
Andr{\'a}s {Juh{\'a}sz}, Dylan~P. {Thurston}, and Ian {Zemke}.
\newblock {Naturality and mapping class groups in Heegaard Floer homology}.
\newblock {\em arXiv e-prints}, page arXiv:1210.4996, October 2012.

\bibitem[Lip06]{Cylindrical}
Robert Lipshitz.
\newblock A cylindrical reformulation of {H}eegaard {F}loer homology.
\newblock {\em Geom. Topol.}, 10:955--1096, 2006.
\newblock Paging previously given as 955--1097.

\bibitem[LOT18]{Bordered}
Robert Lipshitz, Peter~S. Ozsvath, and Dylan~P. Thurston.
\newblock Bordered {H}eegaard {F}loer homology.
\newblock {\em Mem. Amer. Math. Soc.}, 254(1216):viii+279, 2018.

\bibitem[MS12]{MS}
Dusa McDuff and Dietmar Salamon.
\newblock {\em {$J$}-holomorphic curves and symplectic topology}, volume~52 of
  {\em American Mathematical Society Colloquium Publications}.
\newblock American Mathematical Society, Providence, RI, second edition, 2012.

\bibitem[MW95]{PositivityofIntersections}
Mario~J. Micallef and Brian White.
\newblock The structure of branch points in minimal surfaces and in
  pseudoholomorphic curves.
\newblock {\em Ann. of Math. (2)}, 141(1):35--85, 1995.

\bibitem[OS04a]{Disks2}
Peter Ozsv\'{a}th and Zolt\'{a}n Szab\'{o}.
\newblock Holomorphic disks and three-manifold invariants: properties and
  applications.
\newblock {\em Ann. of Math. (2)}, 159(3):1159--1245, 2004.

\bibitem[OS04b]{Disks1}
Peter Ozsv\'{a}th and Zolt\'{a}n Szab\'{o}.
\newblock Holomorphic disks and topological invariants for closed
  three-manifolds.
\newblock {\em Ann. of Math. (2)}, 159(3):1027--1158, 2004.

\bibitem[OS04c]{CobApps3}
Peter Ozsv\'{a}th and Zolt\'{a}n Szab\'{o}.
\newblock Holomorphic triangle invariants and the topology of symplectic
  four-manifolds.
\newblock {\em Duke Math. J.}, 121(1):1--34, 2004.

\bibitem[OS05]{Contact}
Peter Ozsv\'{a}th and Zolt\'{a}n Szab\'{o}.
\newblock Heegaard {F}loer homology and contact structures.
\newblock {\em Duke Math. J.}, 129(1):39--61, 2005.

\bibitem[OS06]{FourManifoldInvariants}
Peter Ozsv\'{a}th and Zolt\'{a}n Szab\'{o}.
\newblock Holomorphic triangles and invariants for smooth four-manifolds.
\newblock {\em Adv. Math.}, 202(2):326--400, 2006.

\bibitem[OS08]{DisksLinks}
Peter Ozsv\'{a}th and Zolt\'{a}n Szab\'{o}.
\newblock Holomorphic disks, link invariants and the multi-variable {A}lexander
  polynomial.
\newblock {\em Algebr. Geom. Topol.}, 8(2):615--692, 2008.

\bibitem[Rob08]{CobApps2}
Lawrence~P. Roberts.
\newblock Rational blow-downs in {H}eegaard-{F}loer homology.
\newblock {\em Commun. Contemp. Math.}, 10(4):491--522, 2008.

\bibitem[Sar11]{Sarkar}
Sucharit Sarkar.
\newblock Maslov index formulas for {W}hitney {$n$}-gons.
\newblock {\em J. Symplectic Geom.}, 9(2):251--270, 2011.

\bibitem[Sar15]{BasepointMoving}
Sucharit Sarkar.
\newblock Moving basepoints and the induced automorphisms of link {F}loer
  homology.
\newblock {\em Algebr. Geom. Topol.}, 15(5):2479--2515, 2015.

\bibitem[Vog73]{HomotopyLimits}
R.~M. Vogt.
\newblock Homotopy limits and colimits.
\newblock {\em Proceedings of the {I}nternational {S}ymposium on {T}opology and
  its {A}pplications ({B}udva, 1972)}, pages 235--241, 1973.

\bibitem[{Zem}15]{GraphTQFT}
Ian {Zemke}.
\newblock {Graph cobordisms and Heegaard Floer homology}.
\newblock {\em arXiv e-prints}, page arXiv:1512.01184, Dec 2015.

\end{thebibliography}

\end{document}